\documentclass[11pt,reqno]{amsart}

\usepackage{amsmath,amsthm,amssymb,amsfonts}
\usepackage{hyperref}
\usepackage{cleveref}
\usepackage{enumitem}

\newtheorem{theorem}{Theorem}[section]
\newtheorem{lemma}[theorem]{Lemma}
\newtheorem{definition}[theorem]{Definition} 
\newtheorem{remark}[theorem]{Remark}
\newtheorem{assumption}[theorem]{Assumption}
\newtheorem{proposition}[theorem]{Proposition}

\AddToHook{env/lemma/begin}{\crefalias{theorem}{lemma}}
\AddToHook{env/proposition/begin}{\crefalias{theorem}{proposition}}
\AddToHook{env/definition/begin}{\crefalias{theorem}{definition}}
\AddToHook{env/remark/begin}{\crefalias{theorem}{remark}}
\AddToHook{env/assumption/begin}{\crefalias{theorem}{assumption}}

\newcommand{\R}{\mathbb{R}}
\renewcommand{\S}{\mathbb{S}}
\newcommand{\dd}{{\, \mathrm d}}

\renewcommand{\d}{\mathrm{d}}
\newcommand{\dt}{\frac{{\d}}{{\d}t}}
\newcommand{\op}[1]{\mathsf{#1}}

\numberwithin{equation}{section}
\usepackage[left=2cm,right=2cm,top=2cm,bottom=2cm]{geometry}

\title[Quantitative stability of constant equilibria in a non-linear alignment model]{Quantitative stability of constant equilibria in a non-linear alignment model of self-propelled particles}

\author[E. Bouin]{Émeric Bouin} \address[E. Bouin]{CEREMADE - Université Paris-Dauphine, PSL Research University, UMR CNRS 7534, Place du Mar\'echal de Lattre de Tassigny, 75775 Paris Cedex 16, France.} \email{bouin@ceremade.dauphine.fr}

\author[A. Frouvelle]{Amic Frouvelle} \address[A. Frouvelle]{CEREMADE - Université Paris-Dauphine, PSL Research University, UMR CNRS 7534, Place du Mar\'echal de Lattre de Tassigny, 75775 Paris Cedex 16, France.} \email{frouvelle@ceremade.dauphine.fr}
\date\today

\begin{document}

\begin{abstract}
  We are interested in the long-time behaviour of the kinetic Vicsek equation, rigorously derived as the mean-field limit~\cite{bolley2012meanfield} of a coupled system of~$N$ stochastic differential equations describing particles moving at unit velocity and aligning with their neighbours. We focus on the local-in-space version (that may for instance appear as a moderate interaction limit instead of mean-field), which is not a priori globally well-posed and could explode in finite time. Despite its simple expression, little is rigorously established about the behaviour of its solutions. We use hypocoercivity methods to show that finite time explosion does not occur in the vicinity of uniform and homogeneous equilibria in space below the critical threshold. We recast the now-classic~\cite{villani2009hypocoercivity} approach of modifying Sobolev-type norms by adding cross-terms, linked to commutators between the different operators appearing in the kinetic equation. However, the fact that the velocity space is the sphere adds significant subtleties and requires to develop an adapted algebraic framework of operators. Taking advantage of this new framework, we manage to perform an approach \textit{à la} Hérau~\cite{herau2007short} to show the nonlinear stability. Our main results are a quantitative decay estimate in the case of the whole space, despite the absence of control of the $L^1$ norm of the perturbation, and a gain in regularity at the nonlinear level which allows to have well-posedness and stability in the space~$H^{s,0}(\mathbb{R}^d\times\S)$ for some~$s<\frac{d}2$ (that is to say without a priori uniform bound in space on the~$L^2$ norm in velocity, that would come with Sobolev injection in the case~$s>\frac{d}2$).

\end{abstract}

\maketitle

\section{Introduction}
In this paper, we are interested in qualitative properties of the following Vlasov--Fokker--Planck type model for alignment of self-propelled particles
\begin{equation} \label{eq-VFP}
  \partial_tf+v\cdot\nabla_xf+\op{\nabla}_v\cdot(\op{\nabla}_v(\mathbb{J}[f]\cdot v) \, f)=\op{\Delta}_vf,
\end{equation}
where~$f(t,x,v)$ represents the density at time~$t\geqslant0$ of particles at position~$x\in\mathbb{R}^d$ (or in a flat torus of dimension~$d$, with $d>1$) having velocity~$v\in\mathbb{S}$. There, $\S$ denotes the unit sphere of $\R^d$ and the operators~$\op{\nabla}_v$, $\op{\nabla}_v \cdot$ and $\op{\Delta}_v$ stand for the gradient, the divergence and the Laplacian of a function on the sphere, respectively. Finally~$\mathbb{J}[f]$ denotes the first moment in velocity of~$f$, given by
\begin{equation} 
\mathbb{J}[f]=\int_{\mathbb{S}}vf\,\d v,\label{def-Jf}
\end{equation}
where the measure on the sphere is the Lebesgue measure normalized to be of total mass one.

This evolution equation is the spatially localized version of the kinetic Vicsek equation, which models the alignment of self-propelled particles moving at constant speed, aligning their velocities with neighbors by dipolar potential and subject to angular noise.

More precisely, let us consider the following stochastic differential system of~$N$ interacting particles
\begin{equation}
  \begin{cases}
   \d X_k = c \,V_k\, \d t,\\
    \d V_k = \frac{\nu}{N}\sum\limits_{i=1}^N\frac1{R^d}K\big(\frac1{R}(X_i-X_j)\big)\nabla_{V_k}(V_i\cdot V_k)\,\d t + \sqrt{2\sigma} \,P_{V_k^\perp}\circ\d B_{t,k},
  \end{cases}
  \label{Vicsek-SDE}
\end{equation}
in which the $k^{\mathrm{th}}$ particle, at position~$X_k\in\mathbb{R}^d$ (or in a flat torus of dimension~$d$), moves at constant speed~$c$ in the direction of its orientation~$V_k\in\mathbb{S}$ ; this orientation~$V_k$ has a drift of intensity~$\nu$ towards the orientations of other neighbouring particles (the potential~$V_i\cdot V_k$ is maximized when $V_k=V_i$) weighted by the observation kernel $K$ (a probability measure) which emphasizes the contributions of particles located within distance of order~$R$ ; finally the orientation is subject to an angular noise of intensity~$\sqrt{2\sigma}$, under the form of the contribution $P_{V_k^{\perp}}\circ \d B_{t,k}$ of a Brownian motion on the sphere (here $P_{V_k^\perp}$ is the projection on the orthogonal of~$V_k$, $B_{t,k}$ are independent brownian motions on~$\mathbb{R}^d$, and~$\circ$ denotes the fact that the stochastic differential equation must be understood in Stratonovich formulation, see~\cite{hsu2002stochastic} for more on this topic). This is a time-continuous version (see~\cite{degond2008continuum} for instance) of the Vicsek model~\cite{vicsek1995novel}, which was originally designed as a discrete in time synchronous jump process for the velocities.

When the number~$N$ of particles is large, the empirical distribution of the system of stochastic differential equations~\eqref{Vicsek-SDE} converges to the solution of the following equation~\cite{bolley2012meanfield}:
\begin{equation} \label{eq-VFP-Kernel}
  \partial_tf+c\,v\cdot\nabla_xf+\nu\,\op{\nabla}_v\cdot(\op{\nabla}_v((K_R*_x\mathbb{J}[f])\cdot v) \, f)=\sigma\op{\Delta}_vf, 
\end{equation}
where~$K_R*_x\mathbb{J}[f]$ is the spatial convolution between the scaled kernel~$\frac1{R^d}K(\frac1R\cdot)$ and the moment~$\mathbb{J}[f]$ defined in~\eqref{def-Jf}.

With a careful change of scale in time and space, and by multiplying~$f$ by a positive scalar~$\rho$ (not considering it as a probability density anymore), we may take without loss of generality~$c=\nu=\sigma=1$, and therefore the evolution~\eqref{eq-VFP} is simply the formal limit of~\eqref{eq-VFP-Kernel} as~$R\to0$. 

Both limits~$N\to\infty$ and~$R\to0$ may also be taken simultaneously, by replacing~$R$ by~$R_N$ in \eqref{Vicsek-SDE}, and letting~$R_N\to0$ as~$N\to\infty$. In the “moderate” interaction regime, when~$R_N^d\cdot N \to \infty$ (which means that each particle interacts with a large number of particles), we expect that the empirical distribution still converges to a solution of~\eqref{eq-VFP}. The methods used to obtain this kind of limit to a local equation in space require a well-posed limit system (see for example~\cite{oelschlager1985law,jourdain1998propagation,diez2020propagation,chen2021rigorous,chaintron2022propagation2}).

Equation~\eqref{eq-VFP} also naturally appears in the formal derivation of macroscopic limits of the kinetic Vicsek equation~\cite{degond2013macroscopic}. Indeed, a hydrodynamic scaling by a parameter~$\varepsilon$ in space and time leads to a rescaled equation with~$R_\varepsilon\to0$ as~$\varepsilon\to0$, leading at the main order to the localized version of the operator.

The well-posedness and the long time behaviour of the spatial-homogeneous version of~\eqref{eq-VFP} is now well understood~\cite{frouvelle2012dynamics}. There is a phase transition depending on the density~$\rho=\int_{\S}f\d v$. When~$\rho\leqslant d$ (or if we keep the parameters as in~\eqref{eq-VFP-Kernel}, when~$\frac{\nu\rho}{\sigma}\leqslant d$, which means either a low alignment strength or a high level of noise), all solutions converge towards the isotropic distribution on the sphere, exponentially fast if~$\rho<d$. When~$\rho>d$, except for the initial conditions~$f^0$ such that~$\mathbb{J}[f^0]=0$, the solutions all converge towards a von Mises equilibrium on the sphere, concentrated around a given orientation, which depends on the initial condition, but in a non-explicit way. 

Regarding the space-inhomogeneous model~\eqref{eq-VFP}, to our knowledge, until recently, few results were available for this equation. The recent work~\cite{briant2023well} deals with a more generic class of kinetic equations for self-propelled particles. It can be straightly adapted to obtain the local in time well-posedness of the model in $\mathsf{L}^\infty(\mathbb{R}^d\times\mathbb{S})$, that is, existence and uniqueness of a bounded weak solution to \eqref{eq-VFP}. We also provide a quantitative pointwise (in space) estimate on the $L^\infty$ norm which allows to observe the finite speed of propagation.

\begin{theorem}\label{thm-existence-uniqueness-Linfini}
Given any nonnegative initial condition~$f^0\in L^\infty(\mathbb{R}^d\times\mathbb{S})$, there exists a unique weak solution to~\eqref{eq-VFP} in~$\mathcal{C}([0,T],L^\infty(\mathbb{R}^d\times\mathbb{S}))$ for all~$T<\frac{\sqrt{d}}{(d-1)\|f^0\|_\infty}$. It is nonnegative and satisfies for all~$t\in[0,T]$ and~$x\in\mathbb{R}^d$ the following estimate,
\begin{equation}\label{eq-Linfinity-estimate}
  \|f(t,x,\cdot)\|_{L^\infty(\S)}\leqslant\|f^0\|_{\infty,B(x,t)\times\S}+(d-1)\int_0^t\|\mathbb{J}[f](s)\|_{\infty,B(x,s-t)}\|f(s)\|_{\infty,B(x,s-t)\times\S}\, \d s,
\end{equation}
where the left-hand side denotes the limit of~$\|f(t)\|_{\infty,B(x,r)\times\S}$ as~$r\to0$.
\end{theorem}
    
In our discussion, the density~$f$ is not assumed to have finite total mass (if~$x\in\mathbb{R}^d$). Observe that~$f \equiv \rho$ is a stationary solution to~\eqref{eq-VFP}. We are interested in the behaviour around this constant stationary solution. As said before, when~$\rho<d$, this is the only homogeneous stationary solution in space, since it is then a stationary solution to the spatial-homogeneous version of~\eqref{eq-VFP}. The main goal of the present paper is to prove that the stability result of this isotropic distribution is also true in the space-inhomogeneous framework. Our aim is to show an energy~--~energy-dissipation estimate on $g$, implying that the solution is global for small perturbations. Our strategy is to follow an hypocoercive-hypoelliptic approach based on high order Sobolev norms.

As we can see in the estimate~\eqref{eq-Linfinity-estimate}, we need a $L^\infty$ bound on~$\mathbb{J}[f]$, which is itself controlled by the~$L^2$ norm in velocity, so we will not actually need to control derivatives in~$v$. And regarding the~$L^\infty$ norm in $x$, a natural Sobolev space would therefore be~$H^{s,0}$ (defined through~$\|f\|_{H^{s,0}}^2=\|f\|_{2}^2+\|(-\op{\Delta}_x)^{\frac{s}2}f\|_2^2)$, with~$s>\frac{d}2$, thanks to Sobolev injections. However, the regularizing properties of the equation (of hypoelliptic type), and the fact that the nonlinearity is governed by moments in velocity allows us to gain some derivatives in space and prove the local well-posedness in a larger space, which is one of the main results of the present paper.

\begin{theorem}\label{thm-well-posedness-Hs0}
  Let $s>\frac{d}2-\frac14$ for $d\geqslant3$, or~$s>\frac78$ in the case~$d=2$. For any~$M>0$, there exists a time~$T_{M}>0$ such that for any~$f^0\in H^{s,0}(\mathbb{R}^d\times\mathbb{S})$ satisfying~$\|f^0\|_{H^{s,0}}\leqslant M$, there exists a unique weak solution~$f$ to~\eqref{eq-VFP} in~$C([0,T_M),H^{s,0}(\mathbb{R}^d\times\mathbb{S}))$ with initial condition~$f^0$. Furthermore the solution depends continuously on the initial condition~$f^0$, and for any~$T<T_M$, there exists a constant~$C>0$ such that for all~$t\in(0,T]$, we have,
  \[\|\nabla_vf\|_{H^{s,0}}\leqslant\frac{C}{\sqrt{t}}\|f^0\|_{H^{s,0}}, \quad \|P_{v^\perp}\nabla_xf\|_{H^{s,0}}\leqslant\frac{C}{t^{\frac32}}\|f^0\|_{H^{s,0}}, \quad \|v\cdot\nabla_xf\|_{H^{s,0}}\leqslant\frac{C}{t^2}\|f^0\|_{H^{s,0}}.\]
  The time~$T_M$ can be explicitly computed (depending on~$d$,~$s$ and~$M$), as well as the constant~$C$ (which depends only on~$d$,~$s$ and~$\frac{T}{T_M}$).
\end{theorem}

The solutions are then actually strong solutions, which are smooth for arbitrary small times, and the finite speed of propagation allows then to define solutions for~$f^0\in H^{s,0}_{\text{loc}}(\mathbb{R}^d\times\S)$, as soon as the local time of existence is bounded below, which is for instance the case if $\|f^0\|_{H^{s,0}(B(x,1)\times\S)}$ is uniformly bounded in~$x$. Therefore we look at a small perturbation of the constant solution~$\rho$ by looking at an initial condition~$f^0=\rho+g^0$ where~$g^0\in H^{s,0}(\mathbb{R}^d\times\mathbb{S})$, which ensures local in time existence and uniqueness of a solution~$f=\rho+g$. The second main result of our work is a quantitative stability result up to the threshold of phase transition, with explicit decay rates.

\begin{theorem}\label{thm-stabilite-nonlineaire-Hs1} Let $s>\frac{d}2-\frac14$ for $d\geqslant3$, or~$s>\frac78$ in the case~$d=2$, and let~$\rho<d$. There exist positive constants~$\eta$,~$C$ and~$\epsilon$ (and~$\lambda$ in the cas of a flat torus), only depending on~$d$,~$\rho$ and~$s$, such that for any initial condition~$g^0$ such that~$\|g^0\|_{H^{s,0}}\leqslant\eta$, the solution~$f$ of equation~\eqref{eq-VFP} with initial condition~$f^0=\rho+g^0$ is global, and satisfies the following decay estimates:
  \begin{enumerate}[label=(\roman*)]
    \item When the domain is a flat torus (and~$g$ has zero average), exponential decay in large time: for all~$t\geqslant0$,
      \[\|g(t,\cdot)\|_{H^{s,0}}\leqslant C\|g^0\|_{H^{s,0}}e^{-\lambda t}.\]
     \item Algebraic decay controlled by spatial $L^2$ moments in~$\mathbb{R}^d$: for all~$t\geqslant0$
      \[\|g(t,\cdot)\|_{H^{s,0}}\leqslant C\frac{\|g^0\|_{H^{s,0}}}{\big(1+\frac{\|g^0\|^2_{H^{s,0}}}{\|g^0\|^2_{H^{s,0}}+\int_{\mathbb{R}^d\times\S}|x|^2(g^0(x,v))^2\dd x \dd v}\,t\big)^{\epsilon}}.\]
    \end{enumerate}
\end{theorem}

It is worth noticing that one important structural difficulty that we have manage to overcome is the fact that the nonlinear operator (and even the linearized counterpart) driving the solution $g$ (see \eqref{eq-VFPg-condensed}) does not conserve positivity nor the $\Vert \cdot \Vert_{L^1}$ norm. This makes a significant difference with, \textit{e.g.}, \cite{bouin2020hypocoercivity} where Nash inequalities and mode-by-mode estimates could be used in the whole space $\R^d$. One may also observe this difficulty in a even more recent paper by Carrapatoso and Gervais \cite{carrapatoso2024boltzmann} on the Boltzmann equation in the full space, where rates of decay in $L^2$ are not obtained without using higher $L^p$ spaces, see \cite[Theorem 1.1]{carrapatoso2024boltzmann}. We shall now comment on the way we prove these results. The main difficulty to construct an hypocoercivity and hypoellipticity functional \textit{à la} Villani \cite{villani2009hypocoercivity} is due to the fact that the velocities lie on a strict submanifold of $\R^d$. One can be convinced, taking a simple example in dimension~$2$, that the standard chain of commutators of~\cite{villani2009hypocoercivity} is not enough. Let us consider the law of the kinetic Brownian motion (also called velocity spherical Brownian motion), corresponding to~\eqref{Vicsek-SDE} with~$\nu=0$ and~$\sigma=1$ : a single particle moving at constant unit speed, with velocity given by a Brownian motion on the sphere. Writing in a comfortable way for the sequel~$v=\binom{\cos\theta}{\sin\theta}$, with~$\theta\in\mathbb{R}/2\pi\mathbb{Z}$,~$\op{T}=\binom{\cos{\theta}}{\sin \theta}\cdot\nabla_x$ and~$\op{A}=\partial_\theta$, we obtain
\[\partial_tf+\op{T}f=\op{A}^2f.\]
Villani's chain of commutators would be to write~$\op{C}_0=\op{A}$, $\op{C}_{i+1}=[\op{C}_{i},\op{T}]$ and hope to obtain the missing directions. But by noting~$\op{S}=[\op{A},\op{T}]=\binom{-\sin{\theta}}{\cos\theta} \cdot \nabla_x$, we would then obtain~$\op{C}_1=\op{S}$ and $\op{C}_2=0$. This chain of commutators does not cover all the directions required to satisfy the parabolic hypoellipticity condition of Hörmander; to satisfy it we must not consider~$[\op{S},\op{T}]$ (which is zero) but use again a commutation with~$\op{A}$, getting that~$[\op{A},\op{S}]=-\op{T}$ covers the missing direction.

In dimensions greater than~$2$, the spherical gradient, divergence and Laplacian give rise to an additional difficulty. Thus, we first design an algebraic framework of operators in \Cref{section-algebraic-framework} that allows to operate nicely in any dimension $d$ and that has relevant commutation properties for our purposes. Then, we derive energy~--~energy-dissipation inequalities. First of all, we present the methodology for the kinetic Brownian motion to show important calculations, for which we can use $H^1$ estimations.
The results we obtain for the kinetic Brownian motion are interesting by themselves, and as far as we know they are new in the case of the whole space~$\mathbb{R}^d$ (or for a generic dimension~$d\geqslant4$). They can be summarized under the following theorem.
\begin{theorem}\label{thm-kinetic-Brownian-motion}
  There exist positive constants $C,\lambda,\epsilon$ such that for any smooth and compactly supported solution~$f$ to
  \begin{align*}
  &\partial_tf+v\cdot\nabla_xf=\op{\Delta}_vf\\[5pt]  &f(t=0,\cdot,\cdot) = f^0 ,
  \end{align*} 
  we have the following explicit estimates (the constants~$C$ and~$\epsilon$ are explicitly computable, and only depend on the dimension~$d$ and the shape of the spatial domain).
  \begin{enumerate}[label=(\roman*)]
    \item \label{short-time-KBM} Short-time regularizing estimates: for all~$t\in(0,1]$,
      \[\|\nabla_vf\|_{L^2}\leqslant\frac{C}{\sqrt{t}}\|f^0\|_{L^2}, \quad \|P_{v^\perp}\nabla_xf\|_{L^2}\leqslant\frac{C}{t^{\frac32}}\|f^0\|_{L^2}, \quad \|v\cdot\nabla_xf\|_{L^2}\leqslant\frac{C}{t^2}\|f^0\|_{L^2}.\]
    \item \label{long-time-KBM-torus}When the domain is a flat torus (and~$f$ has zero average), exponential decay in large time: for all~$t\geqslant0$,
      \[\|f(t,\cdot)\|_{H^1}\leqslant C\|f^0\|_{H^1}e^{-\lambda t}.\]
    \item \label{long-time-KBM-Nash}Algebraic decay controlled by~$L^1$ norm in~$\mathbb{R}^d$: for all~$t\geqslant0$,
      \[\|f(t,\cdot)\|_{H^1}\leqslant C\frac{\|f^0\|_{H^1}}{\big(1+\min(1,\frac{\|f^0\|_{H^1}}{\|f^0\|_{L^1}})^{\frac4d}\,t\big)^{\frac{d}4}}.\]
    \item \label{long-time-KBM-Heisenberg}Algebraic decay controlled by spatial $L^2$ moments in~$\mathbb{R}^d$: for all~$t\geqslant0$
      \[\|f(t,\cdot)\|_{H^1}\leqslant C\frac{\|f^0\|_{H^1}}{\big(1+\frac{\|f^0\|^2_{H^1}}{\|f^0\|^2_{H^1}+\int_{\mathbb{R}^d\times\S}|x|^2(f^0(x,v))^2\dd x \dd v}\,t\big)^{\epsilon}}.\]
    \end{enumerate}
  \end{theorem}
  Some of the results of this theorem were already present in some way in other works. In~\cite{baudoin2018hypocoercive}, estimates similar to~\ref{short-time-KBM} were given with less optimal weights in time, in a more general case where the space is also a Riemannian manifold. In~\cite{cotizelati2023orientation}, the first two estimates of~\ref{short-time-KBM} were given in the case of the flat torus (it actually may be shown thanks to commutators that it implies the third one in that case) of dimension~$d=3$. They also present an estimate corresponding to~\ref{long-time-KBM-torus} (in~$L^2$ norm instead), where an additional enhanced dissipation effect is described: when the intensity of the Brownian motion in velocity is small (or equivalently when the size of the torus is small), effects due to mixing allow to improve the constant~$\lambda$. Let us insist that in our work, we are not interested in this kind of effect: we are in the regime corresponding to a noise above the threshold of phase transition, and we really want to understand the behaviour independently of the shape of the domain. This is why we focus on result on the whole~$\mathbb{R}^d$, such as~\ref{long-time-KBM-Nash} and~\ref{long-time-KBM-Heisenberg}. More comments regarding these two works are given in~\Cref{remark-baudoin-tardif,remark-CZDGV}. To our knowledge, the results~\ref{long-time-KBM-Nash} and~\ref{long-time-KBM-Heisenberg} in the whole space are new, as long as~\ref{short-time-KBM} and~\ref{long-time-KBM-torus} in the case~$d\geqslant4$. 

  Then we use higher order functionals, in the spirit of the seminal paper of Mouhot and Neumann~\cite{mouhot2006quantitative} to deal with the nonlinear alignment term in~\eqref{eq-VFP}. The fact that we do not need to have~$s>\frac{d}2$ to use~$L^\infty$ estimates through the standard Sobolev injections is due to the fact that the derivatives in space of the moments in velocity (the projection on spherical harmonics) are actually controled by a weaker norm than the~$H^{1,0}$ one, thanks to~\Cref{lemma-Sg2} which given in appendix. For instance, from~\Cref{thm-kinetic-Brownian-motion}--\ref{short-time-KBM}, we would obtain that~$\|\mathbb{J}[f](t)\|_{\dot{H}^{1}(\mathbb{R}^d)}\leqslant\frac{C}{t^{\frac32+\frac{\mu}4}}\|f^0\|_{L^2}$, where~$\mu=1$ when~$d=2$, $\mu$ can be arbitrarily chosen in~$(0,1]$ for~$d=3$, and~$\mu=0$ for~$d\geqslant4$. At the end, we may consider that we have obtained a gain of one fourth of a derivative (when~$d\geqslant3$), which is reminiscent to the results of kinetic averaging lemmas. We were not able though to make an explicit link with this theory of averaging lemmas.

      When considering the large time behaviour in the whole space, an additional difficulty came from the fact that we could not control the~$L^1$ norm of the perturbation, even at the linearized level, to be able to use an inequality of Nash type to obtain a result similar to~\Cref{thm-kinetic-Brownian-motion}--\ref{long-time-KBM-Nash}. This is why we needed an alternative result such as~\Cref{thm-kinetic-Brownian-motion}--\ref{long-time-KBM-Heisenberg}, that we proved thanks to Heisenberg’s uncertainty principle instead of Nash inequality, at the price of having a more complex (and second order) differential inequality to solve, and at having an exponent~$\epsilon$ in time which actually depends on the constant arising in the energy~--~energy-dissipation inequality, which is not the case when we use an a priori bound on the~$L^1$ norm.

      Let us now comment on two very recent works that are related to the present paper.

      In~\cite{merino2025stability}, the authors consider the localized Vicsek-BGK equation on the torus, for which the phase transition behaviour at the homogeneous level is similar to the case of Fokker--Planck (though the orientation of~$\mathbb{J}$ is preserved in time, which helps to describe the long-time behaviour). They prove that it is well-posed in~$H^{s,0}$ with~$s>\frac{d}2$, which is natural since we do not expect any regularizing effect from the BGK relaxation operator. Regarding the long-time behaviour, they prove nonlinear stability of the homegeneous states up to the threshold and even a little bit after the threshold (by a perturbation argument). However a crucial assumption is that they need a sufficiently large speed of the particle (or equivalently, a sufficiently small torus). We believe that our method can be used to prove similar results on the whole space~$\mathbb{R}^d$, when~$\rho<d$. Some preliminary computations indicate that, at least at the linearized level, we have a nonincreasing energy~$\mathcal{F}$ equivalent to the~$H^{1}$ norm, and that a quantitative decay in time such as in~\Cref{thm-kinetic-Brownian-motion}-\ref{long-time-KBM-Heisenberg} is available. A more precise study in this direction is left to a future work.

      In the very recent paper~\cite{gu2025mixing}, the authors consider the kinetic Vicsek equation on the torus with a slightly more general (but smooth) kernel of observation. They obtain similar results as~\cite{merino2025stability}, that may persist a little bit after the threshold of phase transition. They also study effects of enhanced diffusion, which corresponds to the regime where the interaction operator is very weak compared to the free transport (or equivalently, with a small torus or a high speed). The kernel of observation may be compactly supported in space, but the constants depend on this kernel, so the intensity of the collision operator have to be very small as the kernel is close to a Dirac mass, and we do not believe their approach may be easily extended to our localized equation.

      The rest of this paper is organised as follows. After a quick reminder of useful formulas on the sphere at the end of this introduction section, \Cref{sec-cauchy-l-infini} deals with the Cauchy problem, naturally set in the space~$L^\infty(\mathbb{R}^d\times\S)$, and describes the proof of~\Cref{thm-existence-uniqueness-Linfini} along the strategy of~\cite{briant2023well}. Then, in~\Cref{section-algebraic-framework}, we describe a convenient algebraic framework of operators which allows to rewrite the evolution equation~\eqref{eq-VFP} in a condensed form, and we give the main properties of commutations of these operators. \Cref{section-kinetic-BM-Hypo-Hypo}, we present the results of hypocoercivity and hypoellipticity by constructing the energy and energy dissipation functionals in the simpler model of the kinetic Brownian motion. \Cref{thm-kinetic-Brownian-motion} is proved through~\Cref{prop-decay-linear-H1,prop-short-time-kineticBM,prop-decay-linear-H1-M}. Next, in \Cref{section-well-posedness-Hs0}, we describe the higher order functionals and we obtain estimates coming from the nonlinear part of~\eqref{eq-VFP}, in order to prove \Cref{thm-well-posedness-Hs0}. Finally, we also use these estimates in \Cref{hypocoercivity-nonlinear} and some refined estimates coming from the perturbative setting, in order to obtain stability results for the energy functional, for short and long time in order to get \Cref{thm-stabilite-nonlineaire-Hs1}. The lemma used to control the spatial derivative of~$\mathbb{J}[f]$ is presented in~Appendix~\ref{section-appendix}.

\subsection{Useful formulas on the sphere}
In order to help a reader that would not be familiar with differential calculus on the sphere, we found relevant to recall some useful formulas (see \cite{frouvelle2012dynamics} for proofs) that we may use a lot without further notice. For smooth $f,g:\mathbb{S}\to\mathbb{R}$, $F:\mathbb{S}\to\mathbb{R}^d$, and any $\mathcal{J}\in\mathbb{R}^d$, one has:
\begin{equation}
  \nabla_v(v\cdot\mathcal{J})=P_{v^\perp}\mathcal J, \qquad \nabla_v\cdot(P_{v^\perp}\mathcal{J})=-(d-1)v\cdot\mathcal{J},\label{gradvdotJ}
\end{equation}
and
\begin{equation*}
  \int_\mathbb{S}F(v)\cdot\nabla_vf  \d v =-\int_\mathbb{S}\nabla_v\cdot(P_{v^\perp}F(v))f(v)\, \d v,%\label{IPPscalarSphere}
\end{equation*}
From there, one gets very important identities for what follows: 
\begin{gather}
  \int_\mathbb{S}\nabla_vg\,\d v =(d-1)\int_\mathbb{S} v \, g \,\d v,\label{IntGradv}\\
  \int_\mathbb{S}f \, \nabla_v g \,\d v= \int_\mathbb{S}g \, [(d-1)vf-\nabla_vf] \,\d v.\label{IPPSphere}
\end{gather}
We also have (for instance from~\eqref{IntGradv} with~$g=v\cdot\mathcal{J}$):
\begin{equation}
  |\mathcal{J}|^2=d\int_\S(v\cdot\mathcal{J})^2\d v=\frac{d}{d-1}\int_\S|P_{v^\perp}\mathcal{J}|^2\d v,\label{J2}
\end{equation}
and therefore we obtain that~$f-d\,v\cdot\mathbb{J}[f]$ and~$d\, v\cdot\mathbb{J}[f]$ are orthogonal in~$L^2(\S)$, which means that the orthogonal projection~$\op{\Pi}_1f$ on spherical harmonics of degree~$1$ (that is to say functions of the form~$v\mapsto\mathcal{J}\cdot v$) is given by
\begin{equation}\label{eq-def-Pi1}
  \op{\Pi}_1f(v)=d\, v\cdot\mathbb{J}[f],
\end{equation}
and  providing
\begin{equation}\label{J2fL2}
  |\mathbb{J}[f]|^2=\frac1d\int(d\,v\cdot\mathbb{J}[f])^2\,\d v\leqslant\frac1d\|f\|^2_{L^2(\S)}.
\end{equation}
  \section{Cauchy problem in~\texorpdfstring{$L^\infty$}{L∞}: existence, uniqueness, and finite speed propagation}
\label{sec-cauchy-l-infini}

This section is devoted to show how the strategy of the recent work~\cite{briant2023well} can be directly adapted to obtain~\Cref{thm-existence-uniqueness-Linfini}. The main addition (which were not explicitly written in~\cite{briant2023well} but quite straightforward from their proofs) are the explicit estimation~\eqref{eq-Linfinity-estimate} and the fact that the initial condition does not need to be in~$L^2$.

We recall their strategy here. First of all, using the same method as in~\cite[Appendix A]{degond1986global}, one may solve, for a prescribed $\mathcal{J}\in L^\infty([0,T]\times\mathbb{R}^d)$ and initial condition~$f^0\in L^2\cap L^\infty(\mathbb{R}^d\times\mathbb{S})$, the following linear equation:
\begin{equation}\partial_tf+v\cdot\nabla_xf+\nabla_v\cdot(P_{v^\perp}\mathcal{J} f)=\op{\Delta}_vf \label{eq-FP-Lin}.\end{equation}

Writing this equation as
\[\partial_tf+v\cdot\nabla_xf+\mathcal{J}\cdot\nabla_vf-\op{\Delta}_vf=(d-1)\left(v\cdot\mathcal{J}\right) \, f,\]
and considering the right-hand side as a source term, we obtain the following estimate that comes from the maximum principle:
\begin{equation}
  \label{maximum-principle}\|f(t)\|_\infty\leqslant\|f^0\|_\infty+(d-1)\int_0^t\||\mathcal{J}(s)|f(s)\|_\infty\,\d s,
\end{equation}
which gives, after using the Grönwall lemma (integral form), 
\begin{equation*}
    \|f(t)\|_\infty \leqslant \|f^0\|_\infty \exp \Big(  (d-1)\int_0^t\|\mathcal{J}(s)\|_\infty\,\d s\Big). 
  \end{equation*}
Furthermore if~$f^0$ is nonnegative, we obtain that~$f$ is also nonnegative.

From there, we shall use a fixed-point iteration. Indeed, since we have existence and uniqueness of a solution to~\eqref{eq-FP-Lin} in~$C([0,T],L^2\cap L^\infty(\mathbb{R}^d\times\mathbb{S}))$, a weak solution to our nonlinear model~\eqref{eq-VFP} corresponds to a weak solution~$f$ of~\eqref{eq-FP-Lin} satisfying furthermore~$\mathcal{J}=\mathbb{J}[f]$.

We therefore denote
\[\Omega=\Big\lbrace f\in L^\infty([0,T]\times\mathbb{R}^d\times\mathbb{S}), \quad \forall t\in[0,T], \qquad \|f(t,\cdot)\|_{\infty}\leqslant\frac1{\|f^0\|_\infty^{-1}-\frac{d-1}{\sqrt d}t}\Big\rbrace.\]
From there, and using~\eqref{J2fL2}, we get that if~$f_{n}\in \Omega$, then~$f_{n+1}$ solving~\eqref{eq-FP-Lin} with~$\mathcal{J}=\mathbb{J}[f_n]$ satisfies
\begin{align*}
  \|f_{n+1}(t)\|_\infty &\leqslant \|f^0\|_\infty \exp \left(  (d-1)\int_0^t\|\mathbb{J}[f_n](s)\|_\infty\,\d s\right) \\
                 &\leqslant \|f^0\|_\infty \exp \left( - \frac{d-1}{\sqrt d}\, \Big[\frac{\sqrt{d}}{d-1} \ln \big(\|f^0\|_\infty^{-1} - \frac{d-1}{\sqrt{d}} s\big)\Big]_0^t\right) = \frac1{\|f^0\|_\infty^{-1}-\frac{d-1}{\sqrt d}t},
\end{align*}
and thus we get that~$f_{n+1}\in\Omega$. By careful weak compactness, one may then extract a subsequence that converges to a weak solution~$f$, thus proving the existence.

For uniqueness, let us improve a little bit the result of~\cite{briant2023well}, using the fact that information propagates at speed less than $1$, in order to prove furthermore the existence when we do not suppose~$f^0\in L^2(\mathbb{R}^d\times\mathbb{S})$. We suppose $f$ and~$\widetilde{f}$ are two solutions in $C([0,T],L^\infty(\mathbb{R}^d\times\mathbb{S}))$, with initial conditions~$f^0$ and~$\widetilde{f}^0$ in~$L^\infty(\mathbb{R}^d\times\mathbb{S})$. Denoting~$u=f-\widetilde{f}$ their difference, we obtain
\[\partial_t u + v\cdot\nabla_xu=\op{\Delta}_vu-\nabla_v\cdot(P_{v^\perp}\mathbb{J}[f]u)-\nabla_v\cdot(P_{v^\perp}\mathbb{J}[u]\widetilde{f}).\]
Let~$\varphi$ be a smooth nonnegative compactly supported function of~$x$ and~$t$, we compute
\begin{equation}
  \label{ddtphiu2}
  \frac12 \frac{\d}{\d t}\int_{\mathbb{R}^d\times\mathbb{S}}\varphi u^2 = \int_{\mathbb{R}^d\times\mathbb{S}}(\partial_t\varphi+v\cdot\nabla_x\varphi)u^2-\varphi\Big(|\nabla_vu|^2+\tfrac{d-1}2u^2\mathbb{J}[f]\cdot v-\widetilde{f}\mathbb{J}[u]\cdot\nabla_vu\Big).
\end{equation}
We now take $\varphi(x,t)=\chi(|x-x_0|+t)$ where~$\chi$ is a smooth non-increasing and nonnegative function, in order to obtain~$\partial_t\varphi+v\cdot\nabla_x\varphi=(1+v\cdot\frac{x-x_0}{|x-x_0|})\chi'(|x-x_0|+t)\leqslant0$, since~$|v|\leqslant1$. 

By denoting $C(t)=\max\big(\|f(t,\cdot)\|_\infty,\|\widetilde{f}(t,\cdot)\|_\infty\big)$, we have that~$\widetilde{f}\mathbb{J}[u]\cdot\nabla_vu\leqslant\frac14C(t)^2|\mathbb{J}[u]|^2+|\nabla_vu|^2$, and therefore using~\eqref{J2fL2}, we obtain
\[
  \frac12 \frac{\d}{\d t}\int_{\mathbb{R}^d\times\mathbb{S}}\varphi u^2\leqslant(\tfrac{d-1}2C(t)+ \tfrac1{4d}C(t)^2)\int_{\mathbb{R}^d\times\mathbb{S}}\varphi u^2.
\]
Solving this Grönwall estimate, we get $\int_{\mathbb{R}^d\times\mathbb{S}}\varphi u^2\leqslant\widetilde{C}(t)\int_{\mathbb{R}^d\times\mathbb{S}}\varphi^0(u^0)^2$. Taking now~$\chi(r)=0$ for $r\geqslant R$ and~$\chi(r)=1$ for~$r\in[0,R-\varepsilon]$, we have in the limit
\begin{equation}
  \label{uniqueness-estimate-L2}
  \int_{B(x_0,R-t)\times\mathbb{S}}|f-\widetilde{f}|^2\leqslant\widetilde{C}(t)\int_{B(x_0,R)\times\mathbb{S}}|f^0-\widetilde{f}^0|^2.
\end{equation}
This gives the uniqueness of the solution, and furthermore provides a way to define a solution for any initial condition in $L^\infty([0,T]\times\mathbb{R}^d\times\mathbb{S})$. We proceed by truncating: if $\widetilde{f}^0=\mathbf{1}_{|x-x_0|\leqslant R+T}f^0$, we set~$\widetilde{f}$ the solution with initial condition~$\widetilde{f}^0$ and we get that if~$f$ is a solution with initial condition~$f^0$, it should satisfy $f(t,v,x)=\widetilde{f}(t,v,x)$ for $|x-x_0|\leqslant R$. This gives the definition of~$f$ for any~$x$ on~$[0,T]$ by taking~$R$ sufficiently large (we recall that~$T$ only depends on the initial $L^\infty$ norm of~$f^0$ and therefore~$\widetilde{f}$ is well-defined on~$[0,T]$ for any~$R$). We directly check that it gives a weak solution, since the test functions have compact support, using the fact that~$\widetilde{f}$ is itself a weak solution.

Finally, to prove the estimate~\eqref{eq-Linfinity-estimate}, we start by proving that if~$\widetilde{f}$ is the solution to the linear equation~\eqref{eq-FP-Lin} with initial condition~$\widetilde{f}^0=\mathbf{1}_{|x-x_0|\leqslant t_0}f^0$ and with~$\mathcal{J}(t,x)=\mathbf{1}_{|x-x_0|+t\leqslant t_0}\mathbb{J}[f](t,x)$, then $\widetilde{f}$ and the solution~$f$ to the nonlinear problem~\eqref{eq-VFP} agree on the set~$\{|x-x_0|+t\leqslant t_0\}=B_{x_0,t_0-t}\times\S$. To this aim, we proceed very similarly as for the uniqueness. The computations are similar, still denoting~$u=f-\widetilde{f}$, and~$\varphi=\chi(|x-x_0|+t)$, with $\chi$ smooth, nonincreasing, such that~$\chi(r)=0$ for~$r\geqslant t_0$ and~$\chi(r)=1$ for~$r\in[0,t_0-\varepsilon]$. We end up with the same equation as~\eqref{ddtphiu2} where~$\mathbb{J}[u]$ is replaced by~$(\mathbb{J}[f]-\mathcal{J})$. But since this last quantity is zero on the support of~$\varphi$, we do not even have to worry about this term and the same strategy leads to the same estimate as~\eqref{uniqueness-estimate-L2} with~$t_0$ instead of~$R$.

We can then directly use the estimate~\eqref{maximum-principle} for~$\widetilde{f}$, and we get
\begin{align*}
  \|f(t)\|_{\infty,B(x_0,t_0-t)\times\S}&=\|\widetilde{f}(t)\|_{\infty,B(x_0,t_0-t)\times\S}  \leqslant\|\widetilde{f}(t)\|_{\infty}\\
  &\leqslant\|\widetilde{f}^0\|_\infty+(d-1)\int_0^t\||\mathcal{J}(s)|\widetilde{f}(s)\|_\infty\,\d s,\\
  &\leqslant\|f^0\|_{\infty,B(x_0,t_0)\times\S}+(d-1)\int_0^t\|\mathbb{J}[f](s)\|_{\infty,B(x_0,t_0-s)\times\S}\|f(s)\|_{\infty,B(x_0,t_0-s)\times\S}\,\d s,
\end{align*}
therefore letting~$t\to t_0$ and using the continuity of~$t\mapsto\|f(t)\|_\infty$, we obtain the desired estimate~\eqref{eq-Linfinity-estimate}, which ends the proof of~\Cref{thm-existence-uniqueness-Linfini}.

\section{Algebraic framework: angular momentum operator and \texorpdfstring{$2$}{2}-vectors}
\label{section-algebraic-framework}

The object of this section is to rewrite our kinetic equation~\eqref{eq-VFP} under the following condensed form:
\begin{equation}
  \label{eq-VFP-condensed}
  \partial_tf+\op{T}f+\frac1d\,\op{A}(f\op{U}f)=\op{A}^2f,
\end{equation}
where~$\op{T}=v\cdot\nabla_x$ is the transport operator,~$\op{A}$ is an angular momentum operator and~$\op{U}$ is a linear operator related to the alignment term. The reason we introduce these operators~$\op{A}$ and~$\op{U}$ rely on their useful properties in terms of antisymmetry and commutation behaviour, summarized in the forthcoming~\Cref{lemmeA2,lemmeAU}.

These operators act on functions from~$\S$ to $\mathbb{R}$ and return functions from~$\S$ to~$\bigwedge^2(\mathbb{R}^d)$, the space of so-called $2$-vectors, which is simply a euclidean space with orthonormal basis denoted~$(e_i\wedge e_j)_{1\leqslant i<j\leqslant d}$, where~$(e_i)_{1\leqslant i\leqslant d}$ is the canonical basis of $\R^d$. There is a natural bilinear and antisymmetric map called exterior product, from~$\mathbb{R}^d\times\mathbb{R}^d$ to~$\bigwedge^2(\mathbb{R}^d)$, given, for two vectors~$u=(u_i)$ and~$w=(w_i)$ in~$\mathbb{R}^d$, by
\begin{equation*}
  u\wedge w=\sum_{i<j}(u_iw_j-u_jw_i)\,e_i\wedge e_j,%\label{def-wedge}
\end{equation*}
which is compatible with the notation~$e_i\wedge e_j$ for the elements of the basis.
\begin{remark}\label{remark-dot-exterior} An useful observation with respect to the (bilinear, antisymmetric) exterior product is the following: if~$v,w,\tilde{v},\tilde{w}$ are four vectors of~$\mathbb{R}^d$, a direct computation shows that
\[(v\wedge w)\cdot(\tilde{v}\wedge\tilde{w})=\sum_{i<j}(v_iw_j-v_jw_i)(\tilde{v}_i\tilde{w}_j-\tilde{v}_j\tilde{w}_i)=\sum_{i\neq j}v_i\tilde{v}_iw_j\tilde{w}_j-v_i\tilde{w}_i\tilde{v}_jw_j=(v\cdot\tilde{v})(w\cdot\tilde{w})-(v\cdot\tilde{w})(\tilde{v}\cdot w).\]
Consequently, if~$v\in\S$ and~$v\cdot w=0$ or~$v\cdot\tilde{w}=0$, then~$(v\wedge w)\cdot(v\wedge\tilde{w})=w\cdot\tilde{w}$.
\end{remark}

We are now ready to define our linear operators~$\op{A}$ and~$\op{U}$, which have very concise forms once we have introduced this convenient framework.
\begin{definition}
  \label{def-A-U}
  For a smooth function~$f$ on~$\S$, we define~$\op{A}f$ and~$\op{U}f$ by
  \begin{equation*}
    \forall v\in\S,\qquad (\op{A}f)(v)=v\wedge\nabla_vf(v),\qquad (\op{U}f)(v)=d\,v\wedge\mathbb{J}[f].
  \end{equation*}
\end{definition}
We may describe~$\op{A}$ component-wise as follows. We denote~$v_i$ the function~$v\in\S\mapsto v \cdot e_i$. Then, for a smooth function~$f$ on~$\S$, we have~$\op{A}f=\sum_{i<j}\op{A}_{i,j}f \,(e_i\wedge e_j),$ where for $i\neq j$, $\op{A}_{i,j}$ is the following linear operator of degree~$1$ on the sphere:
\begin{equation}
    \op{A}_{i,j}=v_i \, e_j \cdot \nabla_v- v_j\,e_i \cdot \nabla_v,\label{def-Aij}
\end{equation}
which corresponds to the infinitesimal generator of the rotation in the plane generated by~$e_i$ and~$e_j$, sending~$e_i$ onto~$e_j$. Said differently, if we write $f(v)=\widetilde{f}(\varphi,w)$ when~$v=\sin\theta(\cos\varphi e_i+\sin\varphi e_j)+\cos\theta\,w$ where~$w\in\S\cap\{e_i,e_j\}^\perp$, then $\op{A}_{i,j}f(v)=\partial_\varphi\widetilde{f}(\varphi,w)$ (in particular in dimension~$d=2$, we recover the setting given in the introduction, with~$\op{A}=\partial_\theta$ for~$v=\binom{\cos\theta}{\sin\theta}$).

In all what follows, we will only consider scalar operators or operators with values in~$\bigwedge^2(\mathbb{R}^d)$ (such as~$\op A$ and~$\op{U}$), and any composition~$\op B\widetilde{\op B}$ between two such objects has to be understood as their (scalar) contraction\footnote{When more than one composition is involved, to avoid ambiguity, we use parentheses to depict the order of contractions; for instance~$\op B \widetilde{\op B}^2\op B=\sum_{i<j}\sum_{k<\ell}\op B_{i,j}  \widetilde{\op B}_{k,\ell}^2\op B_{i,j}$ while $(\op B \widetilde{\op B})(\widetilde{\op B}\op B)=\sum_{i<j}\sum_{k<\ell}\op B_{i,j}  \widetilde{\op B}_{i,j}\widetilde{\op B}_{k,\ell}\op B_{k,\ell}$.}~$\sum_{i<j}\op B_{i,j}\widetilde{\op B}_{i,j}$. In particular, the terms appearing in~\eqref{eq-VFP-condensed} have to be understood as
\[\op{A}^2f=\sum_{i<j}\op{A}_{i,j}^2f,\qquad\op{A}(f\op{U}f)=\sum_{i<j}\op{A}_{i,j}(f\op{U}_{i,j}f).\]

The useful properties regarding~$\op{A}$, the gradient and the Laplace--Beltrami operators on the sphere are given in the following lemma. 

\begin{lemma}\label{lemmeA2}
Fixing~$1\leqslant i,j,k,\ell\leqslant d$ with $i\neq j$ and~$k\neq\ell$, we have the following results.
  \begin{enumerate}[label=(\roman*)]
    \item The operator~$\op{A}_{i,j}$ is antiselfadjoint on~$\mathbb{S}$.\label{Antiself}
    \item If~$f,g\in C^1(\mathbb{S})$, then~$\nabla_vf \cdot \nabla_vg=\op{A}f\cdot\op{A}g$. Consequently if~$f\in C^2(\mathbb{S})$, then~$\op{\Delta}_vf=\op{A}^2f$.\label{A2Delta}
    \item We have~$\op{A}_{i,j}v_k=\delta_{jk}v_i-\delta_{ik}v_j$.\label{comAijvk}
    \item If $\{i,j\}\cap\{k,\ell\}=\emptyset$, then $[\op{A}_{i,j},\op{A}_{k,\ell}]=0$. If~$k\notin\{i,j\}$, then~$[\op{A}_{i,j},\op{A}_{k,j}]=-\op{A}_{i,k}$. As a consequence, we have that~$[\op{A}_{i,j},\op{A}^2]=0$, which means that~$\op{A}$ and~$\op{\Delta}_v$ commute.\label{comAijAkl}
    \end{enumerate}
  \end{lemma}
  Before proving this lemma, let us remark that the first two properties~\ref{Antiself}--\ref{A2Delta} (together with~\Cref{remark-dot-exterior}) directly justify the condensed formula~\eqref{eq-VFP-condensed} of our kinetic equation~\eqref{eq-VFP}, as for any smooth function $g$ and~$f$, we have 
  \begin{align*}
    \int_\S g\sum_{i<j}\frac{1}d\,&\op{A}_{i,j}(f\op{U_{i,j}}f)\,\d v=-\int_\S\frac{1}d\,\sum_{i<j}\op{A}_{i,j}gf\op{U}_{i,j}f\,\d v=-\int_\S\frac{1}d\,\op{A}g\cdot f\op{U}f\,\d v\\
    &=-\int_\S(v\wedge\nabla_vg)\cdot(v\wedge\mathbb{J}[f]f)\,\d v=-\int_\S\nabla_vg\cdot\mathbb{J}[f]f\,\d v=\int_\S g\nabla_v\cdot(P_{v^\perp}\mathbb{J}[f]f)\,\d v.
  \end{align*}
  \begin{proof}[\textbf{Proof of~\Cref{lemmeA2}}]
    We start with the proof of~\ref{Antiself}. For smooth~$f$ and~$g$, using the integration by parts~\eqref{IPPSphere}, we have
    \begin{align*}
    \int_{\S}fv_i e_j\cdot\nabla_vg \, \d v& = e_j \cdot \int_{\S}g\, [(d-1)vfv_i-\nabla_v(f v_i)] \,\d v \\
    &= \int_\S(d-1)(v_iv_j-e_j\cdot P_{v^\perp}e_i)fg \, \d v - \int_{\S}gv_i e_j\cdot\nabla_vf \, \d v.
    \end{align*}
    And therefore we get, using the same computation and exchanging~$i$ and~$j$,
    \[\int_{\S}f \,\op{A}_{i,j}g \, \d v=- \int_{\S}g \,v_i e_j\cdot\nabla_vf \, \d v+ \int_{\S}g\,v_j e_i\cdot\nabla_vf \, \d v=-\int_{\S}g \,\op{A}_{i,j}f \, \d v.\]
    The first part of~\ref{A2Delta} comes from~\Cref{def-A-U} and~\Cref{remark-dot-exterior}. The next part comes from~\ref{Antiself}, as we then have, for all smooth~$g$,
    \[\int_{\S}g\op{\Delta}_v f \, \d v = -\int_{\S}\nabla_vg\cdot\nabla_v f \, \d v =   -\int_{\S}\op{A}g\cdot\op{A} f \, \d v = \int_{\S}g\cdot\op{A}^2 f \, \d v.\]
    The next point~\ref{comAijvk}, using the fact thanks to~\eqref{gradvdotJ} that $e_i\cdot\nabla_v(v_j)=e_i\cdot P_{v^\perp}e_i=\delta_{ij}-v_iv_j$, is a direct computation from~\eqref{def-Aij}.
   To prove the first part of~\ref{comAijAkl}, which are equalities on first order operators, we only need to prove their validity when applied to the functions~$(v_m)_{1\leqslant m\leqslant d}$. Indeed if two first order linear operators~$\op{B}$ and~$\widetilde{\op{B}}$ agree on such functions, then by induction they agree on any polynomial $P(v)$ of these functions thanks to the fact that \[\op{B}(v_mP(v))=(\op{B}v_m)P(v)+v_m\op{B}(P(v))=(\widetilde{\op{B}}v_m)P(v)+v_m\widetilde{\op{B}}(P(v))=\widetilde{\op{B}}(v_mP(v)).\]
   Therefore by density they agree on any smooth function.

   Finally, we have, thanks to~\ref{comAijvk},
   \[ [\op{A}_{i,j},\op{A}_{k,\ell}]v_m=\op{A}_{i,j}(\delta_{\ell m}v_k-\delta_{km}v_\ell)-\op{A}_{k,\ell}(\delta_{jm}v_i-\delta_{im}v_j),\]
   which is zero thanks to~\ref{comAijvk} again whenever~$\{i,j\}\cap\{k,\ell\}=\emptyset$. Taking~$\ell=j$ and~$k\notin\{i,j\}$, we get
   \[ [\op{A}_{i,j},\op{A}_{k,j}]v_m=-\delta_{km}v_i+\delta_{im}v_k=-\op{A}_{i,k}v_m,\]
   which proves the first part of~\ref{comAijAkl}. We then have, for any $1 \leqslant  i < j \leqslant d$ and $1 \leqslant k < \ell \leqslant d$,
\[ [\op{A}_{i,j},\op{A}_{k,\ell}^2]=\op{A}_{k,\ell}[\op{A}_{i,j},\op{A}_{k,\ell}] + [\op{A}_{i,j},\op{A}_{k,\ell}]\op{A}_{k,\ell},\]
which gives immediately that $[\op{A}_{i,j},\op{A}_{k,\ell}^2] = 0$ when~$\{k,\ell\}\cap\{i,j\}=\emptyset$, and we therefore get, since~$\op{A}_{k,\ell}^2=\op{A}_{\ell,k}^2$,
% n(n-1)/2 (k<\ell)- (n-2)(n-1)/2 (k<\ell, both \notin{i,j}) - 1 (k=i,\ell=j) terms = 2n-4 terms 
\begin{align*}
   [\op{A}_{i,j},\op{A}^2]&=\sum_{k<\ell}\,[\op{A}_{i,j},\op{A}_{k,\ell}^2]=\sum_{m\notin\{i,j\}}[\op{A}_{i,j},\op{A}_{i,m}^2]+[\op{A}_{i,j},\op{A}_{m,j}^2]\\
&=\sum_{m\notin\{i,j\}} \op{A}_{i,m}[\op{A}_{i,j},\op{A}_{i,m}] + [\op{A}_{i,j},\op{A}_{i,m}]\op{A}_{i,m} + \op{A}_{m,j}[\op{A}_{i,j},\op{A}_{m,j}] + [\op{A}_{i,j},\op{A}_{m,j}]\op{A}_{m,j}, \\
&=\sum_{m\notin\{i,j\}} -\op{A}_{i,m}\op{A}_{j,m} - \op{A}_{j,m}\op{A}_{i,m} - \op{A}_{m,j}\op{A}_{i,m} - \op{A}_{i,m}\op{A}_{m,j}=0. \qedhere
 \end{align*}
 \end{proof} 

We now turn to properties related to the linear operator~$\op{U}$, which may also be described component-wise, by 
\[\op{U}_{i,j}f= d \left(v_i \, e_j \cdot \mathbb{J}[f] - v_j  \, e_i\cdot \mathbb{J}[f] \right). \]

\begin{lemma}\label{lemmeAU} Fixing~$1\leqslant i,j,k\leqslant d$ with~$i\neq j$, we have the following properties.
  \begin{enumerate}[label=(\roman*)]
  \item The operator~$\op{U}_{i,j}$ is antiselfadjoint on~$\S$. \label{Ustar}
    \item We have~$\op{U}_{i,j}v_k=\delta_{jk}v_i-\delta_{ik}v_j$. \label{Uijvk}
    \item If~$f\in C^1(\S)$, then $\op{U}^2f=\op{AU}f=\op{UA}f=-(d-1)d\,v \cdot \mathbb{J}[f]$, and~$d(d-1)|\mathbb{J}[f]|^2=\|\op{U}f\|_2^2\leqslant\|\op{A}f\|_2^2$.\label{U2}
      \item The operator~$\op{U}_{i,j}$ commutes with~$\op{U}^2$ and we have \label{AU2}
\[
\op{AU}^2=\op{U}^2\op{A}=\op{A}^2\op{U}=\op{UA}^2=\op{U}^3=-(d-1)\op{U}.
\]
  \end{enumerate}
\end{lemma}

Before the proof, let us remark that the property~\ref{Uijvk} corresponds exactly to~\Cref{lemmeA2}--\ref{comAijvk}. This motivates the introduction of the factor~$d$ in~\Cref{def-A-U} of~$\op{U}$. Since~$\mathbb{J}[f]=0$ for any spherical harmonic of degree different from~$1$, this operator~$\op{U}$ is therefore the composition of~$\op{A}$ and the projection~$\op{\Pi}_1$ on spherical harmonics of degree~$1$, which also enlightens the last two properties of this lemma.

\begin{proof}[\textbf{Proof of~\Cref{lemmeAU}}] For~$f,g\in C^1(\S)$, we have~$\int f\op{U}g=d \, \mathbb{J}[f] \wedge \mathbb{J}[g]$, which gives~\ref{Ustar}.
  From~\eqref{J2}, since~$\int_\mathbb{S}v_k^2\d v=\frac1d$, one obtains~$\mathbb{J}[v_k]=\frac1de_k$, and we get~\ref{Uijvk}. Therefore $\mathbb{J}[\op{U}_{i,j}f]=(e_j \cdot \mathbb{J}[f]) e_i - (e_i \cdot \mathbb{J}[f]) e_j$, so we get
\[\op{U}_{i, j}^2f=-d\,(v_i(e_i \cdot \mathbb{J}[f])+v_j(e_j \cdot \mathbb{J}[f])),\]
and summing over~$i<j$ we obtain
\[\op{U}^2f=-(d-1)d\,v \cdot \mathbb{J}[f].\]

In fact, by \Cref{lemmeA2}--\ref{comAijvk}, we also have $\op{A}_{i,j}\op{U}_{i,j}f=-d (v_ie_i \cdot \mathbb{J}[f]+v_je_j\cdot \mathbb{J}[f])$, so by summing over~$i<j$ we also get~$\op{A}\op{U}=\op{U}^2$, and by antiselfadjointness this proves the first part of~\ref{U2}. From there we obtain~$\|\op{U}f\|_2^2=\langle\op{A}f,\op{U}f\rangle\leqslant\|\op{U}f\|_2\|\op{A}f\|_2$, and, thanks to~\eqref{J2},~$d(d-1)|\mathbb{J}[f]|^2=\langle-f,\op{U}^2f\rangle=\|\op{U}f\|_2^2$, which proves the last part of~\ref{U2}.

Similarly, still using \Cref{lemmeA2}--\ref{comAijvk}, we obtain that~$\op{A}_{i,j}\op{U}^2=- (d-1) \op{U}_{i,j}$. Then, since $\op{A}^2$ is the Laplace--Beltrami operator, from~\eqref{gradvdotJ} we get~$\op{A}^2v_{k}=-(d-1)v_k$, and therefore~$\op{A}^2\op U_{i,j}=-(d-1)\op{U}_{i,j}$. Finally, a direct computation of~$\op{U}_{i,j}\op{U}^2$ and~$\op{U}^2\op{U}_{i,j}$ using~\ref{Uijvk} and~\ref{U2}, together with the antiselfadjoint character of~$\op{U}$ and~$\op{A}$, gives~\ref{AU2}.
\end{proof}

\section{The kinetic Brownian motion: hypoellipticity and hypocoercivity.}\label{section-kinetic-BM-Hypo-Hypo}

As a starter to study \eqref{eq-VFP}, we leave aside the nonlinear alignment term for the moment. We are therefore interested in the Kolmogorov equation on the sphere, given by
\begin{equation*}
  \partial_tf+v\cdot\op{\nabla}_xf=\op{\Delta}_vf,%\label{eq-toy-model}
\end{equation*}
resulting of being the law of kinetic Brownian motion in velocity, which corresponds to the following stochastic process:
\begin{equation*}
  \begin{cases} 
    \d X= V \d t, \\
    \d V = P_{V^\perp}\circ\d B_{t},
  \end{cases}
\end{equation*}
studied by Baudoin and Tardif \cite{baudoin2018hypocoercive} (in a more general framework of Riemannian manifold in space as well). Using estimates based on~$\Gamma$-generalized calculus, they obtain a quantitative estimate of hypoelliptic (non-time-optimal) regularity. The goal of this section is to show that we can in fact remain in the classical case of Sobolev-type quadratic functionals, without deploying the machinery of the~$\Gamma$-calculus, and that it may provides precise and constructive decay estimates, both in short and long time.

As before, we note~$\op T=v \cdot \nabla_x$ and the equation is thus written
\begin{equation}
  \partial_tf+\op T f=\op A^2f.\label{eq-FPtoy}
\end{equation}

\subsection{The (quadratic) dissipation associated to the kinetic Brownian motion.}

For a linear (scalar) operator~$\mathsf X$ and a real ~$\mathcal{C}^\infty$ function $f$ with compact support~in~$\mathbb{R}^d\times\S$, we define the quadratic quantity
\begin{equation*}
  Q_{\op X}(f)=\int_{\mathbb{R}^d\times\mathbb{S}}f \, \op Xf \, \dd x \dd v,
\end{equation*}
which we will simply denote~$Q_{\op X}$ if the context is clear (and which therefore depends only on the self-adjoint part of~$\op X$). If~$f$ is a solution~to~\eqref{eq-FPtoy}, we have, after a direct computation, 
\begin{equation}\label{ddtQX}
  \frac{\d}{\d t}Q_{\op X}(f)=Q_{\Phi(\op{X})}(f),
\end{equation}
where the operator $\Phi(\op X)$ (which is self-adjoint when~$X$ is self-adjoint) goes as follows:
\begin{equation}
  \Phi(\op X)=\op A^2\op X+\op X\op A^2+[\op T,\op X]=2\op A\op X\op A+[\op A,[\op A,\op X]]+[\op T,\op X],\label{def-Phi}
\end{equation}
where in the last expression we recall that~$\op A\op X\op A$ stands for~$\sum_{i<j}\op A_{i,j}\op X\op A_{i,j}$. 

We now introduce~$\op{S}=[\op A,\op T]$. Thanks to \Cref{lemmeA2}--\ref{comAijvk}, we obtain $\op{S}_{i,j}=v_i\partial_{x_j}-v_j\partial_{x_i}$, which we can write~$\op S=v\wedge\nabla_x=v\wedge P_{v^\perp}\nabla_x$.
By direct calculation, we obtain:
\begin{equation}\label{eq-decomp-Deltax}
  \op{S}^2+\op T^2=\op{\Delta}_x,
\end{equation}
which is the main interest of our framework. 
The operators~$\op S$ and~$\op T$ commute, but thanks again to \Cref{lemmeA2}--\ref{comAijvk}, we have~$\op {[\op{A}_{i,j},S_{i,j}]}=-v_i \partial_{x_i}-v_j\partial_{x_j}$ and summing over~$i<j$ we find that~$[\op A,\op S]=-(d-1)\op T$.

To improve readability, we first calculate $\Phi(\op X)$ for all the operators that will be of interest.
\begin{proposition}\label{prop-computations-Phi} We have
  \begin{align}
    &\Phi(\op{Id})=2 \op A^2,\label{PhiId}\\
   &\Phi(- \op A^2)=-2 \op A^4 + \op{(AS+SA)},\label{PhiA2}\\
    &\Phi(-(\op{AS+SA}))=2\op S^2-2(\op{A}^2\op{SA}+\op{AS}\op{A}^2)+(d-1)(\op{AS+SA}),\label{PhiASpSA}\\
   & \Phi(-\op S^2)=-2\op{AS}^2\op{A}-2(d-1)\op{T}^2+2\,\op{S}^2.\label{PhiS2}\\
   & \Phi(-\op T^2)=-2\op{AT}^2\op{A}+2(d-1)\op{T}^2-2\,\op{S}^2.\label{PhiT2}
  \end{align}
\end{proposition}

\begin{proof}[\textbf{Proof of \Cref{prop-computations-Phi}.}] 
Since $[\op T,\op{Id}]=0$, we immediately get~\eqref{PhiId}. We combine $\op{[A,T]=S}$ (by definition) and $[\op T,\op A^2] = \op{A [T,A] + [T,A] A}$  to obtain~\eqref{PhiA2}. With the first expression of~$\Phi$ in~\eqref{def-Phi}, we use~$\op{[A,S]}=-(d-1)\op{T}$ and the fact that~$\op{S}$ and~$\op{T}$ commute to write
\begin{align*}
  \Phi(\op{AS+SA})&=\op A^2(2\op{SA}-(d-1)\op{T}) + (2\op{AS}+(d-1)\op{T})\op{A}^2+[\op T,\op{AS}]+[\op T, \op{SA}]\\
               & =2 (\op{A}^2\op{SA}+\op{ASA}^2)+(d-1)[\op{T},\op{A}^2]+[\op T,\op{A}]\op S+\op S[\op T, \op{A}]\\
  &=2 (\op{A}^2\op{SA}+\op{ASA}^2)-(d-1)(\op{AS+SA})-2\op{S}^2,
\end{align*}
which gives~\eqref{PhiASpSA}. Finally, to get \eqref{PhiS2}, we first proceed to prove~\eqref{PhiT2}. Indeed, since $\op{A}$ and $\op{\Delta}_x$ commute, we have $\Phi_{-\op{\Delta}_x}=-2\op{A\Delta}_x\op{A}$, which proves \eqref{PhiS2} from~\eqref{PhiT2} thanks to~\eqref{eq-decomp-Deltax}.

 We have, thanks to the fact that~$\op{S}$ and~$\op{T}$ commute:
\begin{equation}
    [\op A,\op T^2] = \op{T[A,T] + [A,T]T}=2\op{TS}.\label{comAT2}
\end{equation}
We then have
  \begin{equation}
    \op{[A,TS]=T[A,S]+[A,T]S}=-(d-1)\op{T}^2+\op S^2,\label{comATS}
  \end{equation}
and with the second expression of~$\Phi$ in~\eqref{def-Phi}, this gives~\eqref{PhiT2}.
\end{proof}

\subsection{Constructing suitable hypocoercivity functionals on the sphere.}

\begin{definition}\label{def-F-D}
  For~$\alpha,\beta,\gamma,\delta$ some positive parameters that we will be adapted later on, we define two energy functionals~$\mathcal{F}_0$,~$\mathcal{F}_1$ and two dissipation functionals~$\mathcal{D}_0$,~$\mathcal{D}_1$ as follows: 
 \begin{align*}
  \mathcal{F}_0& =Q_{\op{Id}},&\mathcal{F}_1& =\alpha Q_{-\op{A}^2}+\beta Q_{-\op{SA-AS}}+\gamma Q_{-\op{S}^2}+\delta Q_{-\op{T}^2},\\%\label{def-F0-F1}\\
  \mathcal{D}_0&=Q_{-\op{A}^2},&\mathcal{D}_1& =\alpha Q_{\op{A}^4}+\beta Q_{\op{A}^2\op{SA+ASA}^2}+\beta Q_{-\op{S}^2}+\gamma Q_{\op{AS}^2\op{A}}+\delta Q_{\op{AT}^2\op{A}}.%\label{def-D0-D1}
 \end{align*}

 Furthermore, we define the total energy and dissipation functionals as
 \begin{equation*}
   \mathcal{F}=\mathcal{F}_0+\mathcal{F}_1,\quad \mathcal{D}=\mathcal{D}_0+\mathcal{D}_1.
 \end{equation*}
\end{definition}

The main interest of the definition is the following energy~--~energy-dissipation identity, which proof is immediate from~\eqref{ddtQX}--\eqref{def-Phi} and \Cref{prop-computations-Phi}.

\begin{proposition}\label{prop-def-B}
 For any smooth and compactly supported solution~$f$ of equation~\eqref{eq-FPtoy} (which is the case if~$f^0$ is compactly supported), we have~$\forall t >0$,
  \begin{equation*}
    \frac{\d}{\d t}\mathcal{F}_0(f(t,\cdot))=-2\mathcal{D}_0(f(t,\cdot)), \qquad \frac{\d}{\d t}\mathcal{F}_1(f(t,\cdot))=-2\mathcal{D}_1(f(t,\cdot))+\mathcal{B}(f(t,\cdot)),    
  \end{equation*}
  where the functional~$\mathcal{B}$ (corresponding to “bad” terms without sign which will be controlled by the dissipative ones) is defined as follows:
  \begin{equation}
    \label{def-B}
   \mathcal{B}=-(\alpha+\beta(d-1))Q_{-\op{AS-SA}} + 2(\gamma-\delta)[(d-1)Q_{-\op{T}^2}-\,Q_{-\op{S}^2}].
 \end{equation}
\end{proposition}

\noindent An interested reader would point out that~$\mathcal{F}_0(f)=\|f\|_2^2$ and, 
\begin{equation*}
\mathcal{F}_1(f)=\alpha\|\nabla_vf\|_2^2+2\beta\langle\nabla_vf, \nabla_xf\rangle+\gamma\|P_{v^\perp}\nabla_xf\|^2_2+\delta\|v\cdot\nabla_xf\|_2^2,
\end{equation*}
and from there it seems clear that~$\mathcal{F}_0+\mathcal{F}_1$ is equivalent to the square of the~$H^1$ norm of~$f$, as long as~$\beta<\alpha\gamma$.
Similar terms appearing in the dissipation functionals can also be put under the form of standard Sobolev semi-norms, such as~$\mathcal{D}_0(f)=Q_{-\op{A}^2}=\|\nabla_vf\|^2_2=\|f\|_{\dot{H}^{0,1}}^2$, or~$Q_{\op{A}^4}=\|\op{\Delta}_vf\|^2_2=\|f\|_{\dot{H}^{0,2}}^2$ appearing in~$\mathcal{D}_1$, but the meaning of some other terms is somewhat more intricate. The next proposition quantifies more precisely these equivalences with Sobolev semi-norms, and provides the main estimates of this paragraph.

\begin{proposition}\label{prop-estimates-FDB}

  Assume $\beta^2<\alpha\gamma$. Then the functional~$\mathcal{F}_1$ is equivalent to the square of the~$\dot{H}^1$ semi-norm, and more precisely, for any function~$f\in C_c^\infty(\mathbb{R}^d\times\mathbb{S})$, we have
  \begin{align}\label{eq-F-Hs1-upper}
    \mathcal{F}_1(f)&\leqslant(1+\tfrac{\beta}{\sqrt{\alpha\gamma}})\big(\alpha\|\op{A}f\|_2^2+\gamma\|\op{S}f\|_2^2\big)+\delta\|\op{T}f\|_2^2,\\
    \label{eq-F-Hs1-lower}\mathcal{F}_1(f)&\geqslant(1-\tfrac{\beta}{\sqrt{\alpha\gamma}})\big(\alpha\|\op{A}f\|_2^2+\gamma\|\op{S}f\|_2^2\big)+\delta\|\op{T}f\|_2^2,
  \end{align}
  and we recall that $\|f\|^2_{\dot{H}^1}=\|f\|^2_{\dot{H}^{0,1}}+\|f\|^2_{\dot{H}^{1,0}}=\|\op{A}f\|_2^2+\|\op{S}f\|_2^2+\|\op{T}f\|_2^2$.
  
  Similarly~$\mathcal{D}_1$ is equivalent to the square of a Sobolev-type semi-norm given by~$\|\cdot\|^2_{\dot{H}^{0,2}}+\|\cdot\|^2_{\dot{H}^{1,0}}+ \|\cdot\|^2_{\dot{H}^{1,1}}$ through the following estimates:
  \begin{align}
    \mathcal{D}_1(f)&\leqslant(1+\tfrac{\beta}{\sqrt{\alpha\gamma}})(\alpha\|\op{A}^2f\|^2_2+\gamma\|\op{S}\otimes\op{A}f\|_2^2)+\beta\|\op{S}f\|^2_2+\delta \|\op{TA}f\|^2_2,\label{eq-D-Hs2-upper}\\
    \mathcal{D}_1(f)&\geqslant(1-\tfrac{\beta}{\sqrt{\alpha\gamma}})\big(\alpha\|\op{A}^2f\|_2^2+\gamma\|\op{S}\otimes\op{A}f\|_2^2\big)+\beta\|\op{S}f\|_2^2+\delta\|\op{TA}f\|^2_2\label{eq-D-Hs2-lower1}\\
                      &\geqslant(1-\tfrac{\beta}{\sqrt{\alpha\gamma}})\big(\alpha\|\op{A}^2f\|_2^2+\gamma\|\op{S}\otimes\op{A}f\|_2^2\big)+\tfrac{d-1}2\min(\beta,\sqrt{\beta\delta})\|\op{T}f\|^2_2\label{eq-D-Hs2-lower2}
  \end{align}
  where~$\|\op{S}\otimes\op{A}f\|_2^2=Q_{\op{AS}^2\op{A}}=\sum_{i<j}\sum_{k<\ell}\|\op{S}_{i,j}\op{A}_{k,\ell}f\|_2^2=\sum_{k<\ell}\|\op{S}\op{A}_{k,\ell}f\|_2^2=\sum_{i<j}\|\op{S}_{i,j}\op{A}f\|_2^2$, therefore giving
  \[\|f\|^2_{\dot{H}^{0,2}}+\|f\|^2_{\dot{H}^{1,0}}+ \|f\|^2_{\dot{H}^{1,1}}=\|\op{A}^2f\|_2^2+\|\op{S}f\|_2^2+\|\op{T}f\|_2^2+\|\op{S}\otimes\op{A}f\|_2^2+\|\op{TA}f\|^2_2.\]

Finally, the “bad” term~$\mathcal{B}(f)$ given by~\Cref{prop-def-B} is controlled by the dissipation: for any function~$f\in C_c^\infty(\mathbb{R}^d\times\mathbb{S})$,
\begin{equation}
\label{estimate-B-D0-D1} \mathcal{B}(f)\leqslant\frac{\alpha+(d-1)\beta}{\sqrt{\beta}}\big(\mathcal{D}_0(f)+\mathcal{D}_1(f)\big)+\frac{2|\gamma-\delta|}{\sqrt{\beta\delta}}\mathcal{D}_1(f).
\end{equation}

In particular, we obtain~$\frac{\d}{\d t}\mathcal{F}(t)\leqslant-(2-\varepsilon)\mathcal{D}(t)$ for $\varepsilon$ as small as we want, as soon as we have
    \begin{equation}\label{conditions-coeffs-notime}
      \alpha\ll\sqrt{\beta},\quad \beta\ll1,\quad |\gamma-\delta|\ll\sqrt{\beta\delta},
    \end{equation}
    where we use the symbol~$\ll$ to denote inequality up to a small positive constant~$C_\varepsilon$ that may depend only on~$\varepsilon$ and~$d$.
\end{proposition}

\begin{proof}[\textbf{Proof of \Cref{prop-estimates-FDB}.}]
  We start with the estimation of~$\mathcal{F}_1$. Let us recall that we have~$\op A=v\wedge \nabla_v$ and~$\op S=v\wedge P_{v^\perp}\nabla_x$, hence we get:
  \[Q_{-\op{A}^2}=\|\op{A}f\|_2^2=\|\nabla_vf\|_2^2=\|f\|_{\dot{H}^{0,1}}^2,\quad Q_{-\op{S}^2}+Q_{-\op{T}^2}=\|\op{S}f\|_2^2+\|\op{T}f\|_2^2=Q_{-\op{\Delta}_x}=\|f\|_{\dot{H}^{1,0}}^2.\]

  We therefore only need to estimate $Q_{\op{SA+AS}}$. By Cauchy--Schwarz and (anti)selfadjointness of $\op{A}$ and~$\op{S}$, we have $|Q_{\op{SA+AS}}|\leqslant2\sqrt{Q_{-\op{A}^2}Q_{-\op{S}^2}}\leqslant\sqrt\frac{\alpha}{\gamma}Q_{-\op{A}^2}+\sqrt\frac{\gamma}{\alpha}Q_{-\op{S}^2}$ by Young’s inequality, and this provides the upper and lower bounds~\eqref{eq-F-Hs1-upper}--\eqref{eq-F-Hs1-lower}.

We now estimate $\mathcal{D}_1$. We only have to estimate $Q_{-\op{A}^2\op{SA-ASA}^2}$, since $Q_{\op{A}^4}$, $Q_{-\op{S}^2}$, $Q_{\op{AS}^2\op{A}}$ and $Q_{\op{A\Delta}_x\op{A}}$ are good terms. Using same type of arguments as above, we write, by Cauchy--Schwarz,
\[|Q_{\op{A}^2\op{SA}}|=|Q_{\op{ASA}^2}|\leqslant\sqrt{Q_{\op{A}^4}Q_{\op{(AS)(SA)}}}.\] 
To estimate $Q_{\op{(AS)(SA)}}$ one needs to pay attention to the subtleties of parentheses (and use the fact that~$\op{S}_{i,j}$ and~$\op{S}_{k,\ell}$ commute, and again Cauchy--Schwarz and Young). We write
\begin{align}
  Q_{\op{(AS)(SA)}}&=\sum_{i<j}\sum_{k<\ell}Q_{\op{A}_{i,j}\op{S}_{i,j}\op{S}_{k,\ell}\op{A}_{k,\ell}}=\sum_{i<j}\sum_{k<\ell}Q_{\op{A}_{i,j}\op{S}_{k,\ell}\op{S}_{i,j}\op{A}_{k,\ell}}\nonumber\\
&\leqslant\sum_{i<j}\sum_{k<\ell}\sqrt{Q_{\op{A}_{i,j}\op{S}_{k,\ell}^2\op{A}_{i,j}}}\sqrt{Q_{\op{A}_{k,\ell}\op{S}_{i,j}^2\op{A}_{k,\ell}}}\leqslant\sum_{i<j}\sum_{k<\ell}\frac12(Q_{\op{A}_{i,j}\op{S}_{k,\ell}^2\op{A}_{i,j}}+Q_{\op{A}_{k,\ell}\op{S}_{i,j}^2\op{A}_{k,\ell}})=Q_{\op{AS}^2\op{A}}. \label{ASSAvsAS2A} 
\end{align}
As a consequence, $|Q_{\op{A}^2\op{SA+ASA}^2}| \leqslant \sqrt\frac{\alpha}{\gamma} Q_{\op{A}^4} + \sqrt\frac{\gamma}{\alpha} Q_{\op{AS}^2\op{A}}$, and thus we obtain the upper bound and the first lower bound for~$\mathcal{D}_1$ given in~\eqref{eq-D-Hs2-upper}--\eqref{eq-D-Hs2-lower1}.

To prove the last lower bound for~$\mathcal{D}_1$ given in~\eqref{eq-D-Hs2-lower1}, let us show how this combination of~$\|\op{S}f\|_2^2=Q_{-\op{S}^2}$ and~$\|\op{T}f\|_{\dot{H}^{1,0}}^2=Q_{\op{A\Delta}_x\op{A}}$ allow to control~$Q_{-\op{T}^2}=\|\op{T}f\|_2^2$.
Recall that $(\op{STA})^\star = \op{A}^\star \op{T}^\star \op{S}^\star = - \op{ATS}$, and thus, by Cauchy--Schwarz~$|Q_{\op{ATS}}|=|Q_{\op{STA}}|\leqslant\sqrt{Q_{\op{AT}^2\op{A}}Q_{-\op{S}^2}}$. Using the commutator calculated previously in~\eqref{comATS}, whe have~$(d-1)\op{T}^2 - \op{S}^2=\op{STA-ATS}$, and using Young’s inequality we obtain
\begin{equation}
  |(d-1)Q_{-\op{T}^2}-Q_{-\op{S}^2}|\leqslant2\sqrt{Q_{\op{AT}^2\op{A}}Q_{-\op{S}^2}}\leqslant\sqrt{\frac{\beta}{\delta}} Q_{-\op{S}^2} + \sqrt{\frac{\delta}{\beta}} Q_{\op{AT}^2\op{A}}=\frac{1}{\sqrt{\beta\delta}}(\beta\|\op{S}f\|_2^2+\delta\|\op{TA}f\|_2^2),\label{Q-T2-S2-boundedbyD1}
\end{equation}
and therefore 
\begin{align}
  (d-1)Q_{-\op{T}^2}&\leqslant Q_{-\op{S}^2}+ \frac{1}{\sqrt{\beta\delta}}(\beta\|\op{S}f\|_2^2+\delta\|\op{TA}f\|^2_2) \nonumber\\
&\leqslant\left( \frac{1}{\beta}+ \frac{1}{\sqrt{\beta\delta}} \right)(\beta\|\op{S}f\|_2^2+\delta\|\op{TA}f\|_2^2)\leqslant\frac{2}{\min(\beta,\sqrt{\delta\beta})}(\beta\|\op{S}f\|_2^2+\delta\|\op{TA}f\|_2^2),\label{Q-T2-boundedbyD1} 
\end{align}
which completes the proof of the estimate~\eqref{eq-D-Hs2-lower2}.

To estimate~$\mathcal{B}$, we have
\begin{equation}
  \label{estimate-ASplusSA}
  |Q_{-\op{AS-SA}}|\leqslant2\sqrt{Q_{-\op{A}^2}Q_{-\op{S}^2}}\leqslant\frac{2\sqrt{ \mathcal{D}_0\mathcal{D}_1}}{\sqrt{\beta}}\leqslant\frac{1}{\sqrt{\beta}}(\mathcal{D}_0+\mathcal{D}_1),\qquad |(d-1)Q_{-\op{T}^2}-Q_{-\op{S}^2}|\leqslant\frac{1}{\sqrt{\beta\delta}}\mathcal{D}_1,
\end{equation}
thanks to~\eqref{Q-T2-S2-boundedbyD1} and~\eqref{eq-D-Hs2-lower1}. Therefore, with the definition~\eqref{def-B} of~$\mathcal{B}$, it gives~\eqref{estimate-B-D0-D1} and this ends the proof.
\end{proof}

These estimates ensure the decay of the solution to a constant isotropic state, when the initial condition belongs to~$H^1(\mathbb{R}^d\times\S)$, as stated in the next proposition.

\begin{proposition}\label{prop-decay-linear-H1} Given a $d$-dimensional flat torus, there exist positive constants $C$ and~$\lambda$ such that for any smooth solution~$f$ of~\eqref{eq-FPtoy} with zero average, we have for all positive time~$t$:
  \[\|f(t,\cdot)\|_{H^1}\leqslant C\|f(0,\cdot)\|_{H^1}e^{-\lambda t}.\]
  On~$\mathbb{R}^d$, there exists a positive constant $C$ such that for any smooth and compactly supported solution~$f$ of~\eqref{eq-FPtoy}, we have for all positive time~$t$:
  \[\|f(t,\cdot)\|_{H^1}\leqslant C\frac{\|f(0,\cdot)\|_{H^1}}{\big(1+\min(1,\frac{\|f(0,\cdot)\|_{H^1}}{\|f(0,\cdot)\|_{L^1}})^{\frac4d}\,t\big)^{\frac{d}4}}.\]

  The constants in this statement can be constructed explicitly and only depend on the dimension~$d$ and the shape of the domain.
\end{proposition}

\begin{proof}[\textbf{Proof of \Cref{prop-decay-linear-H1}.}] We fix coefficients~$\alpha$,~$\beta$,~$\gamma$ and~$\delta$ satisfying~\eqref{conditions-coeffs-notime}, ensuring for instance that~$\frac{\d}{\d t}\mathcal{F}\leqslant-\mathcal{D}$. This is doable by taking for instance $(\beta,\gamma,\delta)=(\alpha^{\frac32},4\alpha^2,4\alpha^2)$ with $\alpha$ sufficiently small (remember we must ensure $\beta^2<\alpha\gamma$).

    Using the Poincaré inequalities on the torus and on the sphere we obtain, for a smooth function~$f$ with zero average (denoting~$\op{\Pi}_0f$ its local average~$\int_{\S}f(\cdot,v)\d v$):
    \[\mathcal{F}_0=\|f\|_2^2=\|f-\op{\Pi}_0f\|_2^2+\|\op{\Pi}_0f\|_2^2\leqslant\frac1{d-1}\|\op{A}f\|_2^2+\frac1{\Lambda}\|\op{\Pi}_0f\|_{\dot{H}^1(\mathbb{R}^d)}^2\leqslant\frac1{d-1}\|\op{A}f\|_2^2+\frac1{\Lambda}\|f\|_{\dot{H}^{1,0}}^2.\]
    From there, we use~\Cref{prop-estimates-FDB} to get that~$\mathcal{D}\geqslant\widetilde{\Lambda}\mathcal{F}$ for some positive constant~$\widetilde{\Lambda}$ (only depending on the shape of the torus and the chosen coefficients~$\alpha$,~$\beta$,~$\gamma$,~$\delta$). By Gronwal’s inequality and equivalence of~$\mathcal{F}$ with the~$H^1$ squared norm, we obtain the first claim of the proposition.

    Let us now deal with the case of~$\mathbb{R}^d$. We use the fact that the $L^1$ norm of~$f$ is decreasing, as can be checked by multiplying the equation by the sign of~$f$ and integrating by parts. By Nash inequality on $\mathbb{R}^d$ we obtain
    \[\|\op{\Pi}_0f\|_2^{2+\frac4d}\leqslant C_{d}\|\op{\Pi}_0f\|_{\dot{H}^1(\mathbb{R}^d)}^2\|\op{\Pi}_0f\|_{L^1}^{\frac4d}\leqslant C_{d}\|f\|_{\dot{H}^1(\mathbb{R}^d)}^2\|f(0,\cdot)\|_{L^1}^{\frac4d}.\]
    Therefore we get, as before (using~$\lesssim$ to denote inequality up to an explicit constant only depending on~$d$), 
    \[\mathcal{F}_0\lesssim\|\op{A}f\|_2^2+\|f\|_{\dot{H}^{1,0}}^{\frac{2d}{d+2}}\|f(0,\cdot)\|_{L^1}^{\frac4{d+2}}.\]
    Thanks to~\Cref{prop-estimates-FDB}, we therefore get $\mathcal{F}\lesssim\mathcal{D}+\mathcal{D}^{\frac{d}{d+2}}\|f(0,\cdot)\|_{L^1}^{\frac4{d+2}}$. Using the fact that~$\mathcal{F}$ is nondecreasing, this leads (by separating the cases~$\mathcal{D}\leqslant\|f(0,\cdot)\|_{L^1}^2$ and~$\mathcal{D}\geqslant\|f(0,\cdot)\|_{L^1}^2$) to \[\mathcal{F}^{1+\frac2d}\lesssim\max(\|f(0,\cdot)\|_{L^1}^{\frac4d},\mathcal{F}(0)^{\frac2d})\mathcal{D}.\]
    Solving the corresponding differential inequality and using the equivalence of~$\mathcal{F}$ with the~$H^1$ squared norm ends the proof.
\end{proof}
  
From there, we now present some improvements regarding the short and long time behaviour, in view of the study of the nonlinear equation~\eqref{eq-VFP-condensed}. In short time, we want to obtain regularizing estimates (of hypoellipticity type à la Hérau~\cite{herau2007short}) allowing to have local well-posedness in a larger space. And in long time, we would like to obtain a constructive decay estimate in the whole space~$\mathbb{R}^d$ without relying on the~$L^1$ norm, which is unfortunately not controlled in the presence of the alignment term. Those improvements are easier to present in the simple case of the kinetic Brownian motion, which is done in the next two subsections.

\subsection{Short-time hypoelliptic estimates.}
\label{subsec-short-time}
The regularizing estimation à la Hérau~\cite{herau2007short} consists in letting the coefficients of the energy functional depend on time. By chosing appropriate powers of time, we obtain a decreasing quantity which, at time~$0$, is the $L^2$ norm. In order to keep light notations, we now suppose that the coefficients $\alpha$, $\beta$, $\gamma$, $\delta$ are functions of a parameter $\tau$, so we may write~$\mathcal{F}(\tau,f)=\mathcal{F}_0(f)+\mathcal{F}_1(\tau,f)$ to emphasize this dependence. We therefore have, after~\Cref{def-F-D},
\begin{equation*} \partial_\tau\mathcal{F}=\partial_\tau\mathcal{F}_1=\alpha'Q_{-\op{A}^2}+\beta'Q_{-\op{SA-AS}}+\gamma'Q_{-\op{S}^2}+\delta'Q_{-\op{T}^2}.%\label{def-dtauF}
\end{equation*}
 However, the functionals will be always evaluated for~$\tau=t$, so we will still denote~$\mathcal{F}(t)$ as a shortcut for $\mathcal{F}(t,f(t,\cdot))$. We proceed similarly for the other functionals that involve those coefficients, such as~$\mathcal{D}$ and~$\mathcal{B}$. Therefore we get from~\Cref{prop-def-B}, still denoting~$\mathcal{D}=\mathcal{D}_0+\mathcal{D}_1$
\begin{equation*}%\label{ddtFvariable}
  \frac{\d}{\d t}\mathcal{F}(t)=-2\mathcal{D}(t)+\mathcal{B}(t)+\partial_\tau\mathcal{F}_1(t).
\end{equation*}
We can then estimate~$\partial_\tau\mathcal{F}_1$ as was done for~$\mathcal{B}$ in~\eqref{estimate-B-D0-D1}, in order to obtain conditions ensuring decay of the total energy.
  \begin{proposition}\label{prop-conditions-decreaseF-time}
    We have
    \[\partial_\tau\mathcal{F}_1\leqslant|\alpha'|\mathcal{D}_0+\frac{|\beta'|}{\sqrt{\beta}}(\mathcal{D}_0+\mathcal{D}_1)+\Big(\frac{|\gamma'|}{\beta}+\frac{2|\delta'|}{(d-1)\min(\beta,\sqrt{\beta\delta})}\Big)\mathcal{D}_1.\]

    Therefore we obtain~$\frac{\d}{\d t}\mathcal{F}(t)\leqslant-(2-\varepsilon)\mathcal{D}(t)$ for $\varepsilon$ as small as we want, as soon as $\beta<\sqrt{\alpha\gamma}$ and
    \begin{equation}\label{conditions-coeffs-time}
      \alpha\ll\sqrt{\beta},\quad \alpha'\ll1,\quad \beta\ll1,\quad \beta'\ll\sqrt{\beta},\quad |\gamma-\delta|\ll\sqrt{\beta\delta},\quad |\gamma'|\ll\beta,\quad |\delta'|\ll\min(\beta,\sqrt{\beta\delta}).
    \end{equation}
    
    \end{proposition}
\begin{proof}[\textbf{Proof of \Cref{prop-conditions-decreaseF-time}.}]
The estimate is directly obtained from the proof of~\Cref{prop-estimates-FDB}, since we have~$Q_{-\op{A}^2}\leqslant\mathcal{D}_0$, $\beta Q_{-\op{S}^2}\leqslant\mathcal{D}_1$, and the estimates for~$Q_{-\op{AS-SA}}$ and~$Q_{-\op{T}^2}$ are given in~\eqref{estimate-ASplusSA}--\eqref{Q-T2-boundedbyD1}. The new conditions on the coefficients, compared to~\eqref{conditions-coeffs-notime}, therefore follow.
\end{proof}    

Let us now focus on the case of powers of time, therefore we write $\alpha=\alpha_0\tau^{p_\alpha}$, $\beta=\beta_0\tau^{p_\beta}$, $\gamma=\gamma_0\tau^{p_\gamma}$, and~$\delta=\delta_0\tau^{p_\delta}$. We want to satisfy the conditions~\eqref{conditions-coeffs-time} on~$\tau\in[0,1]$. First of all, for~$\tau=1$, we only need the following for the coefficients~$\alpha_0$,~$\beta_0$,~$\gamma_0$,~and~$\delta_0$ (remember that we also need~$\beta_0<\sqrt{\alpha_0\gamma_0}$):
    \begin{equation}\label{conditions-coeffs0}
      \alpha_0\ll\sqrt{\beta_0},\quad \beta_0\ll1,\quad \gamma_0\ll\sqrt{\beta_0\delta_0},\quad \delta_0\ll\beta_0.
    \end{equation}
    Indeed the conditions $\alpha_0\ll1$, $\beta_0\ll\sqrt{\beta_0}$, $\gamma_0\ll\beta_0$ are then implied. As before in the beginning of the proof of~\Cref{prop-decay-linear-H1}, taking $(\beta_0,\gamma_0,\delta_0)=(\alpha_0^{\frac32},4\alpha_0^2,4\alpha_0^2)$ with $\alpha_0\ll1$ allows to satisfy these conditions.

    And then the following conditions on the exponents~$p_\alpha$,~$p_\beta$,~$p_\gamma$ and~$p_\delta$ ensure that the conditions~\eqref{conditions-coeffs-time} (as long as~$\beta<\sqrt{\alpha\gamma}$) are satisfied for all~$\tau\in(0,1]$:
    \begin{equation*}%\label{conditions-exponents-short-time}
      2p_\beta\geqslant p_\alpha+p_\gamma,\quad 2p_\alpha\geqslant p_\beta,\quad p_\alpha\geqslant1, \quad p_\beta\geqslant2, \quad 2p_\gamma\geqslant p_\beta+p_\delta, \quad p_\gamma\geqslant p_\beta+1,\quad p_\delta\geqslant p_\beta+2. 
    \end{equation*}
    Indeed the remaining conditions $p_\beta\geqslant0$, $p_\delta\geqslant p_\beta$, $p_\delta\geqslant p_\beta+1$ are then implied. And from here, it is clear that the minimal exponents that we can take are~$(p_\alpha,p_\beta,p_\gamma,p_\delta)=(1,2,3,4)$. We therefore have~$\mathcal{F}$ nonincreasing on~$[0,1]$. In the following proposition, we summarize the consequences in terms of short time estimates for the components of the $H^1$ norm of~$f$, and we also present a small improvement in weight regarding for averaged quantities such as~$\mathbb{J}[f]$.
    \begin{proposition}\label{prop-short-time-kineticBM}
      There exists a constant~$C$ (only depending on the dimension~$d$) such that for any smooth and compactly supported solution~$f$ of~\eqref{eq-FPtoy}, we have for all~$t\in(0,1]$ :
      \[\|\nabla_vf\|_{L^2}\leqslant\frac{C}{\sqrt{t}}\|f^0\|_{L^2}, \quad \|P_{v^\perp}\nabla_xf\|_{L^2}\leqslant\frac{C}{t^{\frac32}}\|f^0\|_{L^2}, \quad \|v\cdot\nabla_xf\|_{L^2}\leqslant\frac{C}{t^{2}}\|f^0\|_{L^2}.\]
      Furthermore, if we fix~$\mu=1$ when~$d=2$,~$\mu\in(0,1]$ when~$d=3$ or~$\mu=0$ for~$d\geqslant4$, there exists a constant~$C$ (depending on~$\mu$ for~$d=3$) such that for all~$t\in(0,1]$ :
      \[\|\nabla_x\mathbb{J}[f]\|_{L^2(\mathbb{R}^d)}\leqslant\frac{C}{t^{\frac32+\frac{\mu}4}}\|f^0\|_{L^2}.\]
    \end{proposition}
    \begin{proof}[\textbf{Proof of~\Cref{prop-short-time-kineticBM}.}]
      First of all, once we have chosen~$(p_\alpha,p_\beta,p_\gamma,p_\delta)=(1,2,3,4)$, and the coefficients~$\alpha_0$,~$\beta_0$,~$\gamma_0$ and~$\delta_0$ as in~\eqref{conditions-coeffs-time} (with~$\beta_0<\sqrt{\alpha_0\gamma_0}$), the nonincreasing quantity~$\mathcal{F}$ is equivalent to~$\|f\|_{L^2}^2+t\|\nabla_vf\|_{L^2}^2+t^3\|\op{S}f\|_{L^2}^2+t^4\|\op{T}f\|_{L^2}^2$.
      Since~$\|\op{S}f\|_{L^2}^2=\|P_{v^\perp}\nabla_xf\|_{L^2}^2$ and~$\|\op{T}f\|_{L^2}^2=\|v\cdot\nabla_xf\|_{L^2}^2$, this gives the first part of the proposition.
      The second part comes from~\Cref{lemma-Sg2} given in appendix, which implies that~$\|\nabla_x\mathbb{J}[f]\|_{L^2(\mathbb{R}^d)}^2\leqslant C(\|\op{S}f\|_{L^2}^2+\|\op{S}f\|_{L^2}^{2-\mu}\|\op{T}f\|_{L^2}^\mu$). Notice that, still using~\Cref{lemma-Sg2}, we would obtain the same estimates for the local density, or any local quantity only involving a finite number of modes in~$v$.
    \end{proof}

    \begin{remark} \label{remark-baudoin-tardif} To obtain the regularity estimates in~\cite{baudoin2018hypocoercive}, if we were to re-translate their~$\Gamma$-generalized-calculus approach into our framework, it would amount to using the following functional:
  \begin{equation*}
    Q_{\op{Id}}+at^2Q_{-\op{A}^2}+bt^4Q_{\op{-SA-AS}}+ct^6Q_{-\op{S}^2}+\hat{a}t^4Q_{\op{A}^4}+\hat{b}t^6Q_{[\op{T},\op{A}^2]}+\hat{c}t^8Q_{-\op{T}^2}.
  \end{equation*}
  In our case, we have used exponents two times smaller, and our functional is written with fewer terms:
  \[\mathcal{F}=Q_{\op{Id}}+\alpha_0tQ_{-\op{A}^2}+\beta_0t^2Q_{\op{-SA-AS}}+\gamma_0t^3Q_{-\op{S}^2}+\delta_0t^4Q_{-\op{T}^2}.\]
As we understand it, the presence of the term~$Q_{\op{A}^4}$ is there in~\cite{baudoin2018hypocoercive} simply to ensure that the functional is positive by taking as criterion~$\hat{b}^2<\hat{a}\hat{c}$. But we've actually seen in the beginning of the proof of~\Cref{prop-computations-Phi} that~$Q_{[\op{T},\op{A}^2]}=Q_{\op{-SA-AS}}$, so we don't need this term (which is the same as an already existing term, with a lower weight). We don't know if in the context of a Riemannian manifold also in space, which is the case in~\cite{baudoin2018hypocoercive}, our simple approach would work (there may be more terms with the curvature of the variety, in our case we used for example that~$\op{\Delta}_v$ and~$\op{\Delta}_x$ commute).
\end{remark}

\begin{remark} \label{remark-CZDGV} The enhanced dissipation estimates from~\cite{cotizelati2023orientation} come with hypocoercive estimates on the torus, derived through an energy, mode by mode in space. By reconstructing the total energy, we would recover an energy of the form to~$\|f\|_{L^2}^2+at\|\nabla_vf\|_{L^2}^2+bt^2\langle\nabla_xf,\nabla_vf\rangle+ct^3\|P_{v^\perp}\nabla_xf\|_{L^2}^2$, which corresponds to the first terms of our functional. Even if in~\cite{cotizelati2023orientation} this energy is not proven to be decreasing, its derivative is controlled mode by mode, and it allows the authors to obtain hypocoercive decay in $L^2(\S)$ for any mode. If this energy is bounded, then it would be actually possible (after some work using the fact that~$v\cdot\nabla_xf$ can be seen as a commutator) to recover the control of the form~$\|v\cdot\nabla_xf\|_{L^2}^2\leqslant\frac{C}{t^2}\|f^0\|_{L^2}$. However on the whole space, since the amplitude of the modes is not bounded below, we did not manage to control such an energy without adding the term~$\delta t^4\|v\cdot\nabla_xf\|_{L^2}^2$ directly into the functional, except in the case~$d\geqslant6$. In that case, the estimate~\eqref{Q-T2-S2-boundedbyD1} can be written differently, as
  \[(d-1)Q_{-\op{T}^2}-Q_{-\op{S}^2}\leqslant2\sqrt{Q_{-\op{T}^2}Q_{\op{AS}^2\op{A}}}\leqslant\sqrt{d-1}Q_{-\op{T}^2}+\frac{1}{\sqrt{d-1}}Q_{\op{AS}^2\op{A}},\]
  which gives, since~$Q_{-\op{S}^2}\geqslant0$ that~$(d-1)Q_{-\op{T}^2}-Q_{-\op{S}^2}\leqslant\frac{1}{\sqrt{d-1}-1}Q_{\op{AS}^2\op{A}}$. When~$d>5$ we have~$\sqrt{d-1}-1>1$ and therefore the last term in the expression~\eqref{def-B} of~$\mathcal{B}$ can now be controlled by~$c\mathcal{D}$ with~$c<2$ (compared to $\varepsilon\mathcal{D}$ for $\varepsilon$ as small as we want when we include the term~$\delta_0t^4Q_{-\op{T}^2}$), which is sufficient to get the decay.

  Let us also remark that, even if in our work we are not interested in the case of small diffusion~$\sigma$, our functional gives another related approach to obtain their enhanced dissipation result (which therefore is valid in any dimension~$d\geqslant2$). Indeed, by keeping~$\sigma$ in front of~$\op{A}^2$ in~\eqref{eq-FPtoy}, and rewriting our functional as \[Q_{\op{Id}}+\sigma\alpha_0tQ_{-\op{A}^2}+\sigma\beta_0t^2Q_{\op{-SA-AS}}+\sigma\gamma_0t^3Q_{-\op{S}^2}+\sigma^2\delta_0t^4Q_{-\op{T}^2},\] we still obtain that it is nonincreasing under the same conditions~\eqref{conditions-coeffs-time} on the coefficients~$\alpha_0$,~$\beta_0$,~$\gamma_0$,~$\delta_0$ (for~$\sigma t$ bounded). Together with the Poincaré inequality, this directly provides~$(1+C\sigma^2t^4)\|f\|_{L_2}^2\leqslant\|f^0\|_{L_2}^2$, thus giving exponential decay with rate controlled by~$\sqrt{\sigma}$ when~$\sigma$ is small. In our approach, what plays the role of their interpolation between~$\|P_{v^\perp}e f\|^2_{L^2_v}$ (for a unit vector~$e$) and~$\|\nabla_vf\|^2_{L^2_v}$ to control~$\|f\|_{L^2_v}$, is our inequality~\eqref{Q-T2-boundedbyD1} which allows to control~$\|v\cdot\nabla_xf\|^2$ by terms in the dissipation functional~$\mathcal{D}$. Let us mention that the interplay between index of hypoellipticity and time scale of enhanced dissipation has already been noticed in the literature, see for instance~\cite{albritton2022enhanced}.
\end{remark}

\subsection{Long-time estimates.}\label{subsec-long-time-KBM}

In this section, we are interested in deriving a long time estimate without using any information on the mass, that is, forgetting about the conservation of $\Vert f \Vert_{L^1(\R^d \times \S)}$. Indeed, this latter conservation will no longer be available for the nonlinear Vicsek equation (and neither for its linearised version) on the whole space. The next proposition is therefore quite similar to the second part of~\Cref{prop-decay-linear-H1}, but we control the $H^1$ norm thanks to $L^2$ moments in~$x$ instead of the~$L^1$ norm (using Heisenberg uncertainty principle instead of the Nash inequality), at the price of a worse exponent in time. 

\begin{proposition}\label{prop-decay-linear-H1-M} 
  There exist positive constants $C,\epsilon$ (which can be computed explicitly and only depend on the dimension~$d$) such that for any smooth and compactly supported solution~$f$ of~\eqref{eq-FPtoy} on~$\mathbb{R}^d$, denoting~$\mathcal{M}(t)=\int_{\R^d \times \S} \vert x \vert^2 \, f^2 \, \dd x\dd v$, we have for all positive time~$t$:
  \[\|f(t,\cdot)\|_{H^1}\leqslant C\frac{\|f(0,\cdot)\|_{H^1}}{\big(1+\frac{\|f(0,\cdot)\|^2_{H^1}}{\mathcal{M}(0)+\|f(0,\cdot)\|^2_{H^1}}\,t\big)^{\epsilon}}.\]
\end{proposition}
\begin{proof}[\textbf{Proof of~\Cref{prop-decay-linear-H1-M}}]

  We proceed as in the proof of~\Cref{prop-decay-linear-H1}, by first fixing the coefficients in order to have~$\frac{\d}{\d t}\mathcal{F}\leqslant-\mathcal{D}$. Now we have
\begin{align}
  \frac12  \dt \mathcal{M}(t) & = \int_{\R^d \times \S} \vert x \vert^2 \, f \left( - v \cdot \nabla_x f + \Delta_v f\right)  \, \dd x\dd v \nonumber \\
& = \int_{\R^d \times \S}  \left(v \cdot x \right) f^2\, \dd x\dd v - \int_{\R^d \times \S} \vert x \vert^2 \,\left\vert \nabla_v f\right\vert^2   \, \dd x\dd v \nonumber \\ 
& = \frac{2}{d-1}\int_{\R^d \times \S}  (x \cdot \nabla_v f)\, f \,\dd x\dd v - \int_{\R^d \times \S} \vert x \vert^2 \,\left\vert \nabla_v f\right\vert^2   \, \dd x\dd v \nonumber \\ 
&\leqslant  \int_{\R^d \times \S} -\big(x\cdot\nabla_vf-\frac{f}{d-1}\big)^2+\frac{f^2}{(d-1)^2} \, \dd x\dd v \leq \frac{1}{(d-1)^2} \mathcal{F}_0(t).\label{estimate-ddtM-linear}
\end{align}
where we have used that $  \int_\mathbb{S} x \cdot \nabla_v (f^2)\,\d v =(d-1)\int_\mathbb{S} (v \cdot x) \, f^2 \,\d v$ after what has been stated in \eqref{IntGradv}.

By Heisenberg's uncertainty principle,  
\begin{equation*}
\mathcal{F}_0(t)=\int_{\R^d \times \S} f^2 \, \dd x\dd v \leq \frac{2}{d} \left( \int_{\R^d \times \S} \vert x \vert^2 f^2 \, \dd x\dd v \right)^{\frac{1}{2}}\left( \int_{\R^d \times \S} \vert \nabla_x f  \vert^2 \, \dd x\dd v \right)^{\frac{1}{2}}
\end{equation*}
Combining with the hypocoercive estimate of \Cref{prop-estimates-FDB}, we obtain
\[\mathcal{F}(t)=\mathcal{F}_0(t)+\mathcal{F}_1(t)\leqslant C(\sqrt{\mathcal{M}(t)}\sqrt{\mathcal{D}(t)}+\mathcal{D}(t)),\]
from which we get (by separating the cases~$\mathcal{D}\leqslant\mathcal{F}$ and~$\mathcal{D}\geqslant\mathcal{F}$):
\[\mathcal{F}(t)^2\leqslant\widetilde{C}(\mathcal{M}(t)+\mathcal{F}(t))\mathcal{D}(t).\]

We therefore have
\begin{align}
  \dt \mathcal{F}(t) &\leq - c \, \frac{\mathcal{F}(t)^2}{\mathcal{F}(t)+\mathcal{M}(t)}, \label{dtF-F2}\\
  \dt \mathcal{M}(t) &\leq \frac{2}{(d-1)^2} \mathcal{F}(t). \label{dtM-F}
\end{align}
The functional $\mathcal{U}(t) = \mathcal{F}(0) + \mathcal{M}(0) + \frac{2}{(d-1)^2}\int_0^t \mathcal{F}(s) \, \dd s$ satisfies, 
\begin{equation*}
    \mathcal{U}''(t) = \frac{2\mathcal{F}'(t) }{(d-1)^2}\leq - \frac{(d-1)^2c}{2} \, \frac{\mathcal{U}'(t)^2}{\mathcal{F}(t)+\mathcal{M}(t)},
  \end{equation*}
and by using~\eqref{dtM-F} and the fact that~$\mathcal{F}(t)$ is nondecreasing, we get
\begin{equation}
  \mathcal{F}(t)+\mathcal{M}(t) \leq \mathcal{F}(0)+\mathcal{M}(0) + \frac{2}{(d-1)^2}\int_0^t \mathcal{F}(s) \, \dd s = \mathcal{U}(t).\label{eq-FM-U}
\end{equation}

Thus we finally obtain~$\mathcal{U}''(t)\mathcal{U}(t) \leq - p\,\mathcal{U}'(t)^2$ for~$p=\frac{(d-1)^2c}{2}$, from which one deduces after integration that
\begin{equation*}
  \mathcal{U}'(t)\mathcal{U}(t)^{p}\leqslant\mathcal{U}'(0)\mathcal{U}(0)^p=\tfrac{2}{(d-1)^2}\mathcal{F}(0)\mathcal{U}(0)^p,
\end{equation*}
and therefore by a second integration
\[\mathcal{U}(t)\leqslant\Big(\mathcal{U}(0)^{p+1}+\tfrac{2(p+1)}{(d-1)^2}\mathcal{F}(0)\mathcal{U}(0)^pt\Big)^\frac1{p+1}=\mathcal{U}(0)\Big(1+\tfrac{2(p+1)\mathcal{F}(0)}{(d-1)^2\mathcal{U}(0)}t\Big)^{\frac1{p+1}}.\]
Using~\eqref{dtF-F2} and~\eqref{eq-FM-U}, we therefore get, after another integration
\[  \frac{1}{\mathcal{F}(t)}\geqslant\frac1{\mathcal{F}(0)}+\frac{c}{\mathcal{U}(0)}\int_0^t\Big(1+\tfrac{2(p+1)\mathcal{F}(0)}{(d-1)^2\mathcal{U}(0)}s\Big)^{\frac{-1}{p+1}}\d s\geqslant\frac{\widetilde{c}}{\mathcal{F}(0)}\Big(1+ \tfrac{\mathcal{F}(0)}{\mathcal{U}(0)}t\Big)^{\frac{p}{p+1}}.\]
Using~$\mathcal{U}(0)=\mathcal{F}(0)+\mathcal{M}(0)$ and the equivalence of~$\mathcal{F}$ with the~$H^1$ squared norm ends the proof.
\end{proof}

\section{Well-posedness and continuity with respect to initial condition in \texorpdfstring{$H^{s,0}$}{Hˢ⁰}.}
\label{section-well-posedness-Hs0}
\subsection{Higher order functionals and their derivatives in time}
We now study the nonlinear kinetic equation under its condensed form~\eqref{eq-VFP-condensed}. As in the previous section, when~$\op X$ is a selfadjoint linear operator, we study the same quadratic quantity~$Q_{\op X}(f)=\int_{\mathbb{R}^d\times\S}f\op Xf\dd x\dd v$. New terms appear in its derivative in time, corresponding to the nonlinear alignment term:
\begin{equation}\label{ddtQXnonlin}
  \frac{\d}{\d t}Q_{\op{X}}(f)=Q_{\Phi(\op{X})}(f)+\frac2d\int_{\mathbb{R}^d\times\mathbb{S}}(\op{A}\op{X} f)\,(\op{U} f)\,f\,\dd x\dd v.
\end{equation}
where~$\Phi(\op X)$ is defined in~\eqref{def-Phi}. Regarding the existence and uniqueness of solutions to~\eqref{eq-VFP-condensed}, we will also be interested in the difference between solutions to the equation, and in solutions to a similar equation when some of the terms have been frozen, as in~\eqref{eq-FP-Lin}, in order to interpret the solution as a fixed point. Therefore, the cubic term in the right-hand side may be taken with different functions instead of three times the function~$f$, which motivates the following definition.

\begin{definition}\label{def-RX}
   For~$f$,~$g$,~$h$ in~$C_c^\infty(\mathbb{R}^d\times\S)$, we define
\[  R_{\op{X}}(f,g,h)=\int_{\mathbb{R}^d\times\mathbb{S}}(\op{A}\op{X} f) \,g\,(\op{U} h)\,\dd x\dd v.
\]
\end{definition}

We start by describing the main components of our higher-order energy functionals. 
\begin{definition}\label{def-Fs-Ds-Bs-Rs}
  For $s\geqslant0$, and~$\alpha,\beta,\gamma,\delta$ some positive parameters (as in~\Cref{subsec-short-time}, they may later depend on a parameter~$\tau$) we define,~
  \begin{align*}
    \mathcal{F}_{0,s}&=\mathcal{F}_{0}((-\op{\Delta}_x)^{\frac{s}2}\cdot)=Q_{(-\op{\Delta}_x)^s}=\|\cdot\|^2_{\dot{H}^{s,0}},& \mathcal{F}_{1,s}&=\mathcal{F}_{1}((-\op{\Delta}_x)^{\frac{s}2}\cdot)=Q_{(-\op{\Delta}_x)^s\op{X}_1},\\%\label{def-F0-F1-s}\\
    \mathcal{D}_{0,s}&=\mathcal{D}_{0}((-\op{\Delta}_x)^{\frac{s}2}\cdot)=Q_{-(-\op{\Delta}_x)^s\op{A}^2},& \mathcal{D}_{1,s}&=\mathcal{D}_{1}((-\op{\Delta}_x)^{\frac{s}2}\cdot)=Q_{(-\op{\Delta}_x)^s\op{Y}_1},%\label{def-D0-D1-s}
  \end{align*}
  where the functionals $\mathcal{F}_0$,~$\mathcal{F}_1$,~$\mathcal{D}_0$,~$\mathcal{D}_1$ are defined in~\Cref{def-F-D}, giving
  \begin{align*}
    \op{X}_1&=-\alpha\op{A}^2-\beta(\op{SA+AS})-\gamma\op{S}^2-\delta\op{T}^2,\\%\label{def-X1}\\
    \op{Y}_1&=\alpha\op{A}^4+\beta(\op{A}^2\op{SA+ASA}^2)-\beta\op{S}^2+\gamma\op{AS}^2\op{A}+\delta\op{AT}^2\op{A}.%\label{def-Y1}
  \end{align*}

  We then define similarly the equivalent “bad” terms, as in~\eqref{def-B}, by
  \begin{equation*}
    \mathcal{B}_{s}=\mathcal{B}((-\op{\Delta}_x)^{\frac{s}2}\cdot)=Q_{(-\op{\Delta}_x)^s\op{Z}},%\label{def-B-s}
  \end{equation*}
  where
  \begin{equation*}
    \op{Z}=[(\alpha+\beta(d-1))(\op{AS+SA})+2(\gamma-\delta)(\op{S}^2-(d-1)\op{T}^2)].%\label{def-Z}.
  \end{equation*}
  Finally, for any smooth functions $f$,~$g$ and~$h$, we define the following remainder functionals:
  \begin{equation*}
    \mathcal{R}_{0,s}(f,g,h)=R_{(-\op{\Delta}_x)^s}(f,g,h),\qquad \mathcal{R}_{1,s}(f,g,h)=R_{(-\op{\Delta}_x)^s\op{X}_1}(f,g,h).%\label{def-R0-R1-s}
  \end{equation*}
\end{definition}

From these definitions, when~$f$ is a smooth solution of~\eqref{eq-VFP-condensed}, and when the parameter~$\tau$ corresponds to time~$t$, we immediately obtain the derivatives of the functionals thanks to~\eqref{ddtQXnonlin}. For all $s\geqslant0$, we have the following identities:
  \begin{align}
    \label{ddtFs0f}
    \frac{\d}{\d t}\mathcal{F}_{0,s}(f)&=-2\mathcal{D}_{0,s}(f)+\frac2d\mathcal{R}_{0,s}(f,f,f)\\
    \label{ddtFs1f}
    \frac{\d}{\d t}\mathcal{F}_{1,s}(f)&=-2\mathcal{D}_{1,s}(f)+\mathcal{B}_s(f)+\partial_\tau\mathcal{F}_{1,s}(f)+\frac2d\mathcal{R}_{1,s}(f,f,f).   
  \end{align}

  In order to simplify the expressions coming from the estimations of~\Cref{prop-estimates-FDB}, we make the following assumptions, which will be compatible with the final conditions on the coefficients.
  \begin{assumption}\label{assumption-coeffs}
    In all this section, whenever we pick some parameters~$\alpha,\beta,\gamma,\delta$ and~$\nu$, we suppose
    \[\alpha\leqslant1, \quad \beta\leqslant\frac12\sqrt{\alpha\gamma}, \quad \delta\leqslant\gamma\leqslant\sqrt{\beta\delta}, \quad \text{and}\quad \nu\leqslant\min(1,\frac{\gamma}2,\delta).\]
  \end{assumption}
  Let us then summarize the inequalities we can obtain from~\Cref{prop-estimates-FDB} and~\Cref{lemma-Sg2}.
  \begin{lemma}\label{lemma-inequalities-s} For any smooth function~$f$ and any~$s\geqslant0$, we have
    \begin{equation}
      \nu\,\mathcal{F}_{0,s+1}\leqslant \mathcal{F}_{1,s},\quad \nu\,\mathcal{D}_{0,s+1}\leqslant\mathcal{D}_{1,s},\label{eq-nu-sp1} 
    \end{equation}
    and
    \begin{align}
      \|\op{A}f\|_{\dot{H}^{s,0}}&\leqslant\sqrt{\frac{2}{\alpha}\,\mathcal{F}_{1,s}}, & \|\op{A}f\|_{\dot{H}^{s,0}}&=\sqrt{\,\mathcal{D}_{0,s}}, & \|\op{A}^2f\|_{\dot{H}^{s,0}}&\leqslant\sqrt{\frac{2}{\alpha}\,\mathcal{D}_{1,s}},\label{eq-A-A2-s}\\
      \|\op{S}f\|_{\dot{H}^{s,0}}&\leqslant\sqrt{\frac{2}{\gamma}\,\mathcal{F}_{1,s}}, & \|\op{S}f\|_{\dot{H}^{s,0}}&\leqslant\sqrt{\frac1{\beta}\,\mathcal{D}_{1,s}},& \|\op{S}\otimes\op{A}f\|_{\dot{H}^{s,0}}&\leqslant\sqrt{\frac{2}{\gamma}\,\mathcal{F}_{1,s}}\label{eq-S-SA-s},\\
      \|\op{T}f\|_{\dot{H}^{s,0}}&\leqslant\sqrt{\frac{1}{\delta}\,\mathcal{F}_{1,s}}, & \|\op{T}f\|_{\dot{H}^{s,0}}&\leqslant\sqrt{\frac{2}{(d-1)\sqrt{\beta\delta}}\,\mathcal{D}_{1,s}},& \|\op{TA}f\|_{\dot{H}^{s,0}}&\leqslant\sqrt{\frac{1}{\delta}\,\mathcal{F}_{1,s}}.\label{eq-T-TA-s}
    \end{align}
   Finally, letting~$\mu=1$ for~$d=2$,~$\mu\in(0,1]$ for~$d=3$, or~$\mu=0$ for~$d=4$, we have
    \begin{equation}
      \|(-\op{\Delta}_x)^{\frac12}\mathbb{J}[f]\|_{\dot{H}^{s}(\mathbb{R}^d)}\lesssim\sqrt{\frac{1}{\gamma^{1-\frac{\mu}2}\delta^{\frac{\mu}2}}\,\mathcal{F}_{1,s}},\quad  \|(-\op{\Delta}_x)^{\frac12}\mathbb{J}[f]\|_{\dot{H}^{s}(\mathbb{R}^d)}\lesssim\sqrt{\frac{1}{\beta^{1-\frac{\mu}4}\delta^{\frac{\mu}4}}\,\mathcal{D}_{1,s}},\label{eq-gradJf-s}
    \end{equation}
    where the symbol~$\lesssim$ depicts inequality up to a multiplicative constant which only depends on~$\mu$ and~$d$.
  \end{lemma}
  \begin{proof}[\textbf{Proof of~\Cref{lemma-inequalities-s}}]
    We have thanks to~\Cref{assumption-coeffs}:\[\nu Q_{-\op{\Delta}_x}=\nu Q_{-\op{S}^2-\op{T}^2}\leqslant\frac{\gamma}2Q_{-\op{S}^2}+\delta Q_{-\op{T}^2},\quad \nu Q_{\op{A}^2\op{\Delta}_x}=\nu Q_{\op{A}\op{S}^2\op{A}+\op{A}\op{T}^2\op{A}}\leqslant\frac{\gamma}2Q_{\op{A}\op{S}^2\op{A}}+\delta Q_{\op{A}\op{T}^2\op{A}},\] and therefore we get~$\nu\mathcal{F}_{0,1}\leqslant\mathcal{F}_{1,0}$ and~$\nu\mathcal{D}_{0,1}\leqslant\mathcal{D}_{1,0}$ thanks to~\eqref{eq-F-Hs1-lower} and~\eqref{eq-D-Hs2-lower1}, which applied to~$(-\op{\Delta}_x)^{\frac{s}2}f$ give~\eqref{eq-nu-sp1}. The next inequalities~\eqref{eq-A-A2-s},~\eqref{eq-S-SA-s} and~\eqref{eq-T-TA-s} directly come from~\Cref{prop-estimates-FDB} applied to~$(-\op{\Delta}_x)^{\frac{s}2}f$, since~\Cref{assumption-coeffs} implies that~$\delta\leqslant\beta$ and therefore~$\sqrt{\beta\delta}\leqslant\beta$. Finally, the last inequalities come from~\Cref{lemma-Sg2} applied to~$(-\op{\Delta}_x)^{\frac{s}2}f$, giving
    \[\|(-\op{\Delta}_x)^{\frac12}\mathbb{J}[f]\|_{\dot{H}^{s}(\mathbb{R}^d)}\lesssim\|\op{S}f\|_{\dot{H}^{s,0}}+\|\op{S}f\|_{\dot{H}^{s,0}}^{1-\frac{\mu}2}\|\op{T}f\|_{\dot{H}^{s,0}}^{\frac{\mu}2}.\]
    Using~\eqref{eq-S-SA-s},~\eqref{eq-T-TA-s} and~\Cref{assumption-coeffs}, we obtain~\eqref{eq-gradJf-s}.
  \end{proof}

We finally define the main functionals, combining the different orders.
\begin{definition}\label{def-F-D-B-R-combined}
Let~$m>\frac{d}2$,~$\theta\in[0,1]$ and~$\alpha,\beta,\gamma,\delta,\nu$ be given as positive parameters (once more, they may depend on a parameter~$\tau$). For a smooth function~$f$, we define its total energy functional~$\overline{\mathcal{F}}(f)$ as
  \[
    \overline{\mathcal{F}}(f)=\mathcal{F}_{0,0}(f)+\mathcal{F}_{0,m-\theta}(f)+\mathcal{F}_{1,0}(f)+\mathcal{F}_{1,m-\theta}(f)+\nu\,\mathcal{F}_{1,m-\theta+1}(f).%\label{def-F-combined}
  \]
  Similarly, we define the total energy dissipation, total “bad” term, and total remainder (for smooth functions~$g$ and~$h$) as
  \begin{align*}
    \overline{\mathcal{D}}(f)&=\mathcal{D}_{0,0}(f)+\mathcal{D}_{0,m-\theta}(f)+\mathcal{D}_{1,0}(f)+\mathcal{D}_{1,m-\theta}(f)+\nu\,\mathcal{D}_{1,m-\theta+1}(f),\\%\label{def-D-combined}\\
    \overline{\mathcal{B}}(f)&=\mathcal{B}_{0}(f)+\mathcal{B}_{m-\theta}(f)+\nu\,\mathcal{B}_{m-\theta+1}(f),\\%\label{def-B-combined}\\
    \overline{\mathcal{R}}(f,g,h)&=\mathcal{R}_{0,0}(f,g,h)+\mathcal{R}_{0,m-\theta}(f,g,h)+\mathcal{R}_{1,0}(f,g,h)+\mathcal{R}_{1,m-\theta}(f,g,h)+\nu\,\mathcal{R}_{1,m-\theta+1}(f,g,h).%\label{def-R-combined}
  \end{align*}

\end{definition}

In virtue of~\eqref{ddtFs0f}--\eqref{ddtFs1f}, we therefore have, when~$f$ is a smooth solution of~\eqref{eq-VFP-condensed}:
  \begin{equation*}
    %\label{ddtFf-combined}
    \frac{\d}{\d t}\overline{\mathcal{F}}(f)=-2\overline{\mathcal{D}}(f)+\overline{\mathcal{B}}(f)+\partial_\tau\overline{\mathcal{F}}(f)+\frac2d\overline{\mathcal{R}}(f,f,f).
  \end{equation*}

  We thus want to control the terms~$\overline{\mathcal{B}}$,~$\partial_\tau\overline{\mathcal{F}}(f)$ and~$\overline{\mathcal{R}}(f,f,f)$ in terms of the dissipation~$\overline{\mathcal{D}}$.  
  The first two (quadratic terms) are directly estimated as in~\Cref{section-kinetic-BM-Hypo-Hypo} for the kinetic Brownian motion, as stated in the following proposition. The study of the cubic quantities, to control~$\overline{\mathcal{R}}(f,g,h)$, is the object of the next subsection.
  \begin{proposition}
    \label{prop-bar-B-dtauF} For any smooth function~$f$, we have
\begin{equation}
\label{estimate-bar-B-D} \overline{\mathcal{B}}(f)\leqslant\Big(\frac{2(\alpha+(d-1)\beta)}{\sqrt{\beta}}+\frac{2|\gamma-\delta|}{\sqrt{\beta\delta}}\Big)\overline{\mathcal{D}}(f),
\end{equation}
\begin{equation}
\label{estimate-bar-dtauF-D} \partial_\tau\overline{\mathcal{F}}(f)\leqslant\Big(|\alpha'|+\frac{|(\nu\alpha)'|}{\nu}+\frac{|\beta'|+\frac{|(\nu\beta)'|}{\nu}}{\sqrt{\beta}}+\frac{\max(|\gamma'|,\frac{|(\nu\gamma)'|}{\nu})}{\beta}+\frac{2\max(|\delta'|,\frac{|(\nu\delta)'|}{\nu})}{(d-1)\sqrt{\beta\delta}}\Big)\overline{\mathcal{D}}(f).
\end{equation}
  \end{proposition}
  \begin{proof}[\textbf{Proof of~\Cref{prop-bar-B-dtauF}}]
    From~\Cref{def-Fs-Ds-Bs-Rs}, thanks to~\Cref{prop-estimates-FDB} we directly obtain
    \begin{equation}\label{estimate-Bs}
      \mathcal{B}_s(f)\leqslant\frac{\alpha+(d-1)\beta}{\sqrt{\beta}}\big(\mathcal{D}_{0,s}(f)+\mathcal{D}_{1,s}(f)\big)+\frac{2|\gamma-\delta|}{\sqrt{\beta\delta}}\mathcal{D}_{1,s}(f).
    \end{equation}
    For~$s=m-\theta+1$, we use~\eqref{eq-nu-sp1} to get 
\[  \nu\mathcal{B}_{m-\theta+1}(f)\leqslant\frac{\alpha+(d-1)\beta}{\sqrt{\beta}}\big(\mathcal{D}_{1,m-\theta}(f)+\nu\mathcal{D}_{1,m-\theta+1}(f)\big)+\frac{2|\gamma-\delta|}{\sqrt{\beta\delta}}\nu\mathcal{D}_{1,m-\theta+1}(f),
\]
to which we add~\eqref{estimate-Bs} (for~$s=0$ and~$s=m-\theta$), thus obtaining~\eqref{estimate-bar-B-D} thanks to~\Cref{def-F-D-B-R-combined}.

Similarly, thanks to~\Cref{prop-conditions-decreaseF-time}, we have
\begin{equation}\label{estimate-dtauFs}
  \partial_\tau\mathcal{F}_{1,s}\leqslant|\alpha'|\mathcal{D}_{0,s}+\frac{|\beta'|}{\sqrt{\beta}}(\mathcal{D}_{0,s}+\mathcal{D}_{1,s})+\Big(\frac{|\gamma'|}{\beta}+\frac{2|\delta'|}{(d-1)\sqrt{\beta\delta}}\Big)\mathcal{D}_{1,s},
\end{equation}
and with the same estimates we obtain   \[\partial_\tau(\nu\mathcal{F}_{1,m-\theta+1})\leqslant|(\nu\alpha)'|\mathcal{D}_{0,m-\theta+1}+\frac{|(\nu\beta)'|}{\sqrt{\beta}}(\mathcal{D}_{0,m-\theta+1}+\mathcal{D}_{1,m-\theta+1})+\Big(\frac{|(\nu\gamma)'|}{\beta}+\frac{2|(\nu\delta)'|}{(d-1)\sqrt{\beta\delta}}\Big)\mathcal{D}_{1,m-\theta+1}.\]
    Since~$\partial_\tau\mathcal{F}_{0,0}=\partial_\tau\mathcal{F}_{0,m-\theta}=0$, we use~\eqref{eq-nu-sp1} and add~\eqref{estimate-dtauFs} (for~$s=0$ and~$s=m-\theta$), to get~\eqref{estimate-bar-dtauF-D} thanks to~\Cref{def-F-D-B-R-combined}.
  \end{proof}

\subsection{Estimates for the cubic terms.}

We want to provide estimates for the term~$\overline{\mathcal{R}}(f,g,h)$, and a first preliminary step consists in deriving some estimates for the cubic quantities~$R_{\op{X}}$ given by~\Cref{def-RX}, for the different operators~$\op{X}$ that will be needed for our energy functionals.

 \begin{lemma}\label{prop-estimates-RX}
Let $m>\frac{d}2$, and $s\geqslant0$ (not necessarily integers). We use the symbol~$\lesssim$ to denote inequalities up to a multiplicative constant, which may only depend on~$m$, $s$ and the dimension~$d$.
To have lighter notation, we denote
\[N_{m,s}(g,h)=\|g\|_{H^{m,0}} \,\|\mathbb{J}[h]\|_{\dot{H}^{s}(\mathbb{R}^d)} + \|g\|_{\dot{H}^{s,0}}\,\|\mathbb{J}[h]\|_{H^{m}(\mathbb{R}^d)} ,\]
and we then have the following estimates:
  \begin{align}
    R_{(-\op{\Delta}_x)^s}(f,g,h) &\lesssim \|\op Af\|_{\dot{H}^{s,0}}\,N_{m,s}(g,h), \label{eq-RDeltas}\\
    R_{-\op A^2(-\op{\Delta}_x)^s}(f,g,h) &\lesssim \|\op A^2f\|_{\dot{H}^{s,0}}\big(N_{m,s}(g,h)+N_{m,s}(\op{A}g,h)\big), \label{eq-RA2Deltas}\\
    R_{-(\op{AS+SA})(-\op{\Delta}_x)^s}(f,g,h) &\lesssim \Big(\|\op S\otimes\op Af\|_{\dot{H}^{s,0}}+\|\op{T}f\|_{\dot{H}^{s,0}}\Big)\big(N_{m,s}(g,h)+N_{m,s}(\op{A}g,h)\big),\label{eq-RASpSADeltas}\\
    R_{-\op{S}^2(-\op{\Delta}_x)^s}(f,g,h) &\lesssim \big(\|\op S\otimes\op Af\|_{\dot{H}^{s,0}}+\|\op{T}f\|_{\dot{H}^{s,0}}\big) \Big(N_{m,s}(\op{S}g,h)+N_{m,s}(g,(-\op{\Delta}_x)^{\frac12}h)\big),\label{eq-RS2Deltas} \\
    R_{-\op{T}^2(-\op{\Delta}_x)^s}(f,g,h) &\lesssim \big(\|\op T\op Af\|_{\dot{H}^{s,0}}+\|\op{S}f\|_{\dot{H}^{s,0}}\big) \Big(N_{m,s}(\op{T}g,h)+N_{m,s}(g,(-\op{\Delta}_x)^{\frac12}h)\big).\label{eq-RT2Deltas}
  \end{align}
\end{lemma}

\begin{proof}[\textbf{Proof of \Cref{prop-estimates-RX}}]
  We use the following useful estimate, when~$u_1$ and~$u_2$ are functions of the space variable~$x$ only, which can be obtained by Fourier decomposition thanks to Young’s convolution inequality, and the fact that $(\widehat{u_1}* \widehat{u_2})(\cdot)|\cdot|^s\leqslant2^{s-1}(\widehat{u_1(\cdot)|\cdot|^s}*\widehat{u_2}+\widehat{u_1}*\widehat{u_2(\cdot)|\cdot|^s})$:
  \begin{equation*}
    %\label{estimateh1h2Hs}
    \|u_1u_2\|_{\dot{H}^{s}} \lesssim \|\widehat{u_1}\|_{L^1}\|u_2\|_{\dot{H}^s}+\|\widehat{u_2}\|_{L^1}\|u_1\|_{\dot{H}^s}\lesssim\|u_1\|_{H^{m}}\|u_2\|_{\dot{H}^s}+\|u_2\|_{H^m}\|u_1\|_{\dot{H}^s} .
  \end{equation*}

  From there, we can then estimate products in~$\dot{H}^{s,0}$ as long as we have a uniform bound in~$v$ for one of the two terms, which, as we will see, is the case for any term involving~$\op{U}$. For instance we have
  \[ \|g\op{U}h\|^2_{\dot{H}^{s,0}}=\int_{\S}\|g(v,\cdot)\op Uh(v,\cdot)\|_{\dot{H}^s}^2\d v \lesssim \int_{\S}\|g(v,\cdot)\|_{\dot{H}^s}^2\|\op Uh(v,\cdot)\|_{H^m}^2+\|g(v,\cdot)\|_{H^m}^2\|\op Uh(v,\cdot)\|_{\dot{H}^s}^2\,\d v. \]
Now we fix~$v\in\S$. Recalling that $|v\wedge\mathcal{J}|^2=|P_{v^\perp}\mathcal{J}|^2\leqslant|\mathcal{J}|^2$ for all~$\mathcal{J}\in\mathbb{R}^d$, we have
\begin{equation*}%\label{UhL2Rd}
  \|\op Uh(v,\cdot)\|_{L^2(\mathbb{R}^d)}^2=d^2\|P_{v^\perp}\mathbb{J}[h]\|_{L^ 2(\mathbb{R}^d)}^2\leqslant d^2\|\mathbb{J}[h]\|_{L^ 2(\mathbb{R}^d)}^2.
\end{equation*}
Therefore, using the fact that~$\op{U}$ and~$\mathbb{J}$ commute with any (fractional) derivative in~$x$, we may replace~$h$ by~$(-\op{\Delta}_x)^{\frac{s}2}h$ in the previous inequality to get
\[\|\op Uh(v,\cdot)\|_{\dot{H}^s(\mathbb{R}^d)}^2\lesssim \|\mathbb{J}[h]\|^2_{\dot{H}^{s}(\mathbb{R}^d)},\]
and combining both similarly to get~$H^m$ instead of~$\dot{H}^s$. We obtain
\begin{equation}
  \|g\op{U}h\|_{\dot{H}^{s,0}}\lesssim \big(\|g\|^2_{\dot{H}^{s,0}}\|\mathbb{J}[h]\|^2_{H^{m}(\mathbb{R}^d)} + \|g\|^2_{H^{m,0}}\|\mathbb{J}[h]\|^2_{\dot{H}^{s}(\mathbb{R}^d)}\big)^{\frac12}\lesssim N_{m,s}(g,h),\label{gUh-Hs0}
\end{equation}
and this provides the first estimate~\eqref{eq-RDeltas}, as we have
    \[R_{(-\op{\Delta}_x)^s}(f,g,h)  = \int_{\mathbb{R}^d\times\S}(-\op{\Delta}_x)^{\frac{s}2}\op Af(-\op{\Delta}_x)^{\frac{s}2}(g\op Uh) \leqslant \|\op Af\|_{\dot{H}^{s,0}}\|g\op{U}h\|_{\dot{H}^{s,0}}.\]

For the next estimates, we begin similarly, estimating~$\|\op{A}(g\op{U}h)\|_{\dot{H}^{s,0}}\leqslant\|\op{A}g\op{U}h\|_{\dot{H}^{s,0}}+\|g\op{AU}h\|_{\dot{H}^{s,0}}$ first. We have
    \[ \|\op{A}g\op{U}h\|^2_{\dot{H}^{s,0}}\lesssim \int_{\S}\|\op{A}g(v,\cdot)\|_{\dot{H}^s}^2\|\op Uh(v,\cdot)\|_{H^m}^2+\|\op{A}g(v,\cdot)\|_{H^m}^2\|\op Uh(v,\cdot)\|_{\dot{H}^s}^2\,\d v \leqslant N_{m,s}(\op{A}g,h)^2.\]
Using~\Cref{lemmeAU}--\ref{U2}, we also get as before that for all~$v\in\S$,
    \[\|\op{AU}h(v,\cdot)\|_{L^2(\mathbb{R}^d)}^2=d^2(d-1)^2\|v\cdot\mathbb{J}[h]\|_{L^ 2(\mathbb{R}^d)}^2\leqslant d^2(d-1)^2\|\mathbb{J}[h]\|_{L^ 2(\mathbb{R}^d)}^2,\]
    and we proceed the same (replacing~$h$ with~$(-\Delta_x)^{\frac s2}h$) to get 
   \[\|g\op{AU}h\|^2_{\dot{H}^{s,0}}\lesssim \int_{\S}\|g(v,\cdot)\|_{\dot{H}^s}^2\|\op{AU}h(v,\cdot)\|_{H^m}^2+\|g(v,\cdot)\|_{H^m}^2\|\op{AU}h(v,\cdot)\|_{\dot{H}^s}^2\,\d v \leqslant N_{m,s}(g,h)^2.
   \]

This provides the second estimate~\eqref{eq-RA2Deltas}, as we have
    \[R_{-\op{A}^2(-\op{\Delta}_x)^s}(f,g,h) \leqslant \|\op A^2f\|_{\dot{H}^{s,0}}\|\op{A}(g\op{U}h)\|_{\dot{H}^{s,0}}\lesssim \|\op A^2f\|_{\dot{H}^{s,0}} \big(N_{m,s}(\op{A}g,h)+N_{m,s}(g,\op{A}h)\big).\]
To get the third estimate~\eqref{eq-RASpSADeltas}, we use the fact that that~$\op{A}\op{S}+\op{S}\op{A}=2\op{S}\op{A}-(d-1)\op{T}$. We obtain
  \[ R_{-(\op{A}\op{S}+\op{S}\op{A})(-\op{\Delta}^{s})}(f,g,h)\leqslant\|(2\op{SA}-(d-1)\op{T})f\|_{\dot{H}^{s,0}}\|\op{A}(g\op{U}h)\|_{\dot{H}^{s,0}}.
  \]
  and we use~\eqref{ASSAvsAS2A} to get $\|\op{SA}f\|_{\dot{H}^{s,0}}=Q_{(\op{AS})(\op{SA})}((-\op{\Delta}_x)^{\frac s2}f)\leqslant Q_{\op{AS}^2\op{A}}((-\op{\Delta}_x)^{\frac s2}f)=\|\op{S}\otimes\op{A}f\|_{\dot{H}^{s,0}}$, which gives~\eqref{eq-RASpSADeltas}.
  We treat the fourth estimate~\eqref{eq-RA2Deltas} similarly, using the commutator~\eqref{comAT2} and the fact that~$\op{A}$ commutes with~$\op{\Delta}_x=\op{S}^2+\op{T}^2$, which gives~$[\op{A},\op{S^2}]=-[\op{A},\op{T}^2]=-2\op{TS}$:
  \begin{align}\nonumber
    R_{-\op S^2(-\op{\Delta}_x)^s}&(f,g,h) = \int_{\mathbb{R}^d\times\S}(-\op{\Delta}_x)^{\frac{s}2}(-2\op{ST}+\op S^2\op{A})f(-\op{\Delta}_x)^{\frac{s}2}(g\op Uh)\\             \nonumber &\leqslant2\|\op Tf\|_{\dot{H}^{s,0}}\,\|\op S(g\op Uh)\|_{\dot{H}^{s,0}}+\sum_{i<j}\|\op S\op A_{i,j}f\|_{\dot{H}^{s,0}}\,\|\op  S(g\op U_{i,j}h)\|_{\dot{H}^{s,0}}\\
    &\leqslant(2\|\op Tf\|_{\dot{H}^{s,0}}+\|\op{S}\otimes\op{A}f\|_{\dot{H}^{s,0}})\Bigg(\Big(\sum_{i<j}\|\op  Sg\,\op U_{i,j}h\|^2_{\dot{H}^{s,0}}\Big)^{\frac12}+\Big(\sum_{i<j}\|g\,\op  S\op U_{i,j}h)\|^2_{\dot{H}^{s,0}}\Big)^{\frac12}\Bigg), \label{RS2Deltas-firstestimate}
  \end{align}
  where we have proceeded as in~\eqref{ASSAvsAS2A} (using the fact that~$S_{i,j}$ and~$S_{k,\ell}$ commute) to get
  \begin{align*}
    \|\op S(g\op Uh)\|_{\dot{H}^{s,0}}^2&=\sum_{i<j}\sum_{k<\ell}\langle\op S_{i,j}(g\op U_{i,j} h),\op S_{k,\ell}(g\op U_{k,\ell} h)\rangle_{\dot{H}^{s,0}}=\sum_{i<j}\sum_{k<l}\langle\op S_{k,\ell}(g\op U_{i,j} h),\op S_{i,j}(g\op U_{k,\ell} h)\rangle_{\dot{H}^{s,0}}\\
    &\leqslant\sum_{i<j}\sum_{k<\ell}\frac12\big(\|\op S_{k,\ell}(g\op U_{i,j} h)\|_{\dot{H}^{s,0}}+\|\op S_{i,j}(g\op U_{k,\ell} h)\|_{\dot{H}^{s,0}}\big)=\sum_{i<j}\|\op S(g\op U_{i,j} h)\|_{\dot{H}^{s,0}}^2,
  \end{align*}
  and then used Cauchy--Schwarz. We then have, proceeding as usual:
  \begin{align*}
    \sum_{i<j}\|\op  Sg\,\op U_{i,j}h\|^2_{\dot{H}^{s,0}}
    &\lesssim \sum_{i<j}\int_{\S}\big(\|\op  Sg(v,\cdot)\|^2_{\dot{H}^s}\|\op U_{i,j}h(v,\cdot)\|^2_{H^{m}}+ \|\op  Sg(v,\cdot)\|^2_{H^m}\|\op U_{i,j}h(v,\cdot)\|^2_{\dot{H}^{s}}\big)\d v\\
    &\lesssim \int_{\S}\big(\|\op  Sg(v,\cdot)\|^2_{\dot{H}^s}\|\op Uh(v,\cdot)\|^2_{H^{m}}+ \|\op  Sg(v,\cdot)\|^2_{H^m}\|\op Uh(v,\cdot)\|^2_{\dot{H}^{s}}\big)\d v\lesssim N_{m,s}(\op Sg,h)^2.
  \end{align*}
  Similarly, we have
  \[\sum_{i<j}\|g\,\op  S\op U_{i,j}h\|^2_{\dot{H}^{s,0}}\lesssim \sum_{i<j}\int_{\S}\big(\|g(v,\cdot)\|^2_{\dot{H}^s}\|\op  S\op U_{i,j}h(v,\cdot)\|^2_{H^{m}}+ \|g(v,\cdot)\|^2_{H^m}\|\op  S\op U_{i,j}h(v,\cdot)\|^2_{\dot{H}^{s}}\big)\d v.\]
  Fixing~$v\in\S$ and recalling that~$\op{S}=v\wedge\nabla_x$, we have
  \[\|\op  S\op U_{i,j}h(v,\cdot)\|^2_{L^2(\mathbb{R}^d)}=\|v\wedge\nabla_x\op U_{i,j}h(v,\cdot)\|_{L^2(\mathbb{R}^d)}^2\leqslant\|\nabla_x\op U_{i,j}h(v,\cdot)\|_{L^2(\mathbb{R}^d)}^2=\|\op U_{i,j}(-\op{\Delta}_x)^{\frac12}h(v,\cdot)\|_{L^2(\mathbb{R}^d)}^2,\]
  and therefore using~\eqref{J2}, we obtain
  \[\sum_{i<j}\|\op  S\op U_{i,j}h(v,\cdot)\|^2_{L^2(\mathbb{R}^d)}\leqslant\|\op U(-\op{\Delta}_x)^{\frac12}h(v,\cdot)\|_{L^2(\mathbb{R}^d)}^2\leqslant d^2\|\mathbb{J}[(-\op{\Delta}_x)^{\frac12}h]\|^2_{L^2(\mathbb{R}^d)}.\]
  As usual, we replace~$h$ by $(-\op{\Delta}_xh)^{\frac s2}$, and we obtain
 \[\sum_{i<j}\|g\,\op  S\op U_{i,j}h\|^2_{\dot{H}^{s,0}}\lesssim \|g\|^2_{\dot{H}^{s,0}}\|\mathbb{J}[(-\op{\Delta}_x)^{\frac12}h]\|_{H^{m}}^2+\|g\|^2_{H^{m}}\|\mathbb{J}[(-\op{\Delta}_x)^{\frac12}h]\|_{\dot{H}^{s}}^2\lesssim N_{m,s}(g,(-\op{\Delta}_x)^{\frac12}h)^2.\]
 Thanks to~\eqref{RS2Deltas-firstestimate}, that gives~\eqref{eq-RS2Deltas}.
 Finally, we also use the commutator~\eqref{comAT2} to get
  \[
    R_{-\op T^2(-\op{\Delta}_x)^s}(f,g,h) = \int_{\mathbb{R}^d\times\S}(-\op{\Delta}_x)^{\frac{s}2}(2\op{TS}+\op T^2\op{A})f(-\op{\Delta}_x)^{\frac{s}2}(g\op Uh)\leqslant\big(\|\op Sf\|_{\dot{H}^{s,0}}+\|\op{TA}f\|_{\dot{H}^{s,0}}\big)\,\|\op T(g\op Uh)\|_{\dot{H}^{s,0}}.
  \]
  As before, we have~$\|\op Tg\op Uh\|_{\dot{H}^{s,0}}\leqslant N_{m,s}(\op{T}g,h)$, and
  \[\|\op{TU}h(v,\cdot)\|_{L^2(\mathbb{R}^d)}^2=\|v\cdot\nabla_x\op Uh(v,\cdot)\|_{L^2(\mathbb{R}^d)}^2\leqslant\|\nabla_x\op Uh(v,\cdot)\|_{L^2(\mathbb{R}^d)}^2\leqslant d^2\|\mathbb{J}[(-\op{\Delta}_x)^{\frac12}h]\|^2_{L^2(\mathbb{R}^d)},\]
  which gives that~$\|g\op{TU}h\|_{\dot{H}^{s,0}}\leqslant N_{m,s}(g,(-\op{\Delta}_x)^{\frac12}h)$. We apply therefore the same method to get~\eqref{eq-RT2Deltas}, and this ends the proof. 
\end{proof}

We are now ready to provide the estimates for~$\overline{\mathcal{R}(f,g,h)}$. We provide two types of estimates, one for which we control thanks to~$\overline{\mathcal{F}}$ only for~$g$ and~$h$, which will be used to have short time bounds, while the second one uses a mix of~$\overline{\mathcal{F}}$ and~$\overline{\mathcal{D}}$ for~$g$ and~$h$, and will be useful for the decay in large time.

\begin{proposition}\label{prop-estimates-barR}
 We let~$\mu=1$ when~$d=2$,~$\mu\in(0,1]$ for~$d=3$, or~$\mu=0$ for~$d\geqslant4$. Here we use the symbol~$\lesssim$ to denote inequalities up to a multiplicative constant that only depends on~$m$,~$\theta$,~$\mu$ and the dimension~$d$ (and in particular this constant does not depend on the choice of the coefficients, as soon as they satisfy~\Cref{assumption-coeffs}). Then for any smooth functions~$f$,~$g$ and~$h$
  \begin{align}
    \overline{\mathcal{R}}(f,g,h) &\lesssim\sqrt{\nu^{-\theta}}\,\Big(\frac{\gamma}{\delta}\Big)^{\frac{\mu}4}\,\sqrt{\,\overline{\mathcal{D}}(f)\overline{\mathcal{F}}(g)\overline{\mathcal{F}}(h)},\label{estimate-bar-R-D-F2}\\
    \overline{\mathcal{R}}(f,g,h)&\lesssim \sqrt{\frac{\nu^{-\theta}}{\alpha}} \,\sqrt{\,\overline{\mathcal{D}}(f)}\Big(\sqrt{\,\overline{\mathcal{D}}(g)\overline{\mathcal{F}}(h)}+\sqrt{\,\overline{\mathcal{F}}(g)\overline{\mathcal{D}}(h)}\Big).\label{estimate-bar-R-D2-F}          
  \end{align}
\end{proposition}
\begin{proof}[\textbf{Proof of~\Cref{prop-estimates-barR}}]
  We start with interpolation estimates, in order to control the~$H^{m,0}$ norms appearing in~\Cref{prop-estimates-RX}.

  Since~$m=(1-\theta)(m-\theta)+\theta(m-\theta+1)$, we get~$\|f\|_{\dot{H}^{m,0}}\leqslant\|f\|^{1-\theta}_{\dot{H}^{m-\theta,0}}\|f\|^{\theta}_{\dot{H}^{m-\theta+1,0}}=\mathcal{F}_{0,m-\theta}^{\frac{1-\theta}2}\mathcal{F}_{0,m-\theta+1}^{\frac{\theta}2}$. Using~\eqref{eq-nu-sp1} and~\Cref{assumption-coeffs}, we have
  \[\|f\|_{H^{m,0}}^2=\|f\|_{L_2}^2+\|f\|_{\dot{H}^{m,0}}^2\leqslant\mathcal{F}_{0,0}+\nu^{-\theta}\,\mathcal{F}_{0,m-\theta}^{1-\theta}\mathcal{F}_{1,m-\theta}^\theta\leqslant\nu^{-\theta}\,\big(\mathcal{F}_{0,0}+\max(\mathcal{F}_{0,m-\theta},\mathcal{F}_{1,m-\theta})\big),\]
  and we use~\Cref{def-F-D-B-R-combined} to obtain:
  \begin{equation}\label{fHm-barF}
    \|f\|_{H^{m,0}}\leqslant\sqrt{\nu^{-\theta}\,\overline{\mathcal{F}}(f)}.
  \end{equation}
Similarly, using~\eqref{eq-A-A2-s},~\eqref{eq-S-SA-s},~\eqref{eq-T-TA-s} and~\eqref{eq-gradJf-s} we obtain 
  \begin{align}
    \|\op{A}f\|_{H^{m,0}}&\leqslant\sqrt{\frac{2\nu^{-\theta}}{\alpha}\,\overline{\mathcal{F}}(f)},& \|\op{A}f\|_{H^{m,0}}&\leqslant\sqrt{\nu^{-\theta}\,\overline{\mathcal{D}}(f)}\label{AfHm-barFD},\\
   \|\op{S}f\|_{H^{m,0}}&\leqslant\sqrt{\frac{2\nu^{-\theta}}{\gamma}\,\overline{\mathcal{F}}(f)},&\quad \|\op{S}f\|_{H^{m,0}}&\leqslant\sqrt{\frac{\nu^{-\theta}}{\beta}\,\overline{\mathcal{D}}(f)},\label{SfHm-barFD} \\
    \|\op{T}f\|_{H^{m,0}}&\leqslant\sqrt{\frac{\nu^{-\theta}}{\delta}\,\overline{\mathcal{F}}(f)},&\quad \|\op{T}f\|_{H^{m,0}}&\leqslant\sqrt{\frac{2\nu^{-\theta}}{(d-1)\sqrt{\beta\delta}}\,\overline{\mathcal{D}}(f)},\label{TfHm-barFD}\\
    \|(-\op{\Delta}_x)^{\frac12}\mathbb{J}[f]\|_{H^{m}(\mathbb{R}^d)}&\lesssim\sqrt{\frac{\nu^{-\theta}}{\gamma^{1-\frac{\mu}2}\delta^{\frac{\mu}2}}\,\overline{\mathcal{F}}(f)},&
   \|(-\op{\Delta}_x)^{\frac12}\mathbb{J}[f]\|_{H^{m}(\mathbb{R}^d)}&\lesssim\sqrt{\frac{\nu^{-\theta}}{\beta^{1-\frac{\mu}4}\delta^{\frac{\mu}4}}\,\overline{\mathcal{D}}(f)}.\label{gradJfHm-barFD}
  \end{align}

  We are ready to plug these inequalities into the expressions of~$N_{m,s}(\cdot,\cdot)$ appearing in~\Cref{prop-estimates-RX}. For the first type of estimate, when we want to control with~$\sqrt{\,\overline{\mathcal{F}}(g)}$ and~$\sqrt{\,\overline{\mathcal{F}}(h)}$, we use the fact that~$\|\mathbb{J}[h]\|_{\dot{H}^{s}(\mathbb{R}^d)}\leqslant\|h\|_{\dot{H}^{s,0}}=\sqrt{\,\mathcal{F}_{0,s}(h)}$ and~$\|\mathbb{J}[h]\|_{{H}^{m}(\mathbb{R}^d)}\leqslant\|h\|_{{H}^{m,0}}$. For instance we get, thanks to~\eqref{fHm-barF}:
  \begin{equation*}
    %\label{Nmgh-F}
    N_{m,s}(g,h)\leqslant\sqrt{\nu^{-\theta}}\Big(\sqrt{\,\overline{\mathcal{F}}(g)\mathcal{F}_{0,s}(h)}+\sqrt{\,\mathcal{F}_{0,s}(g)\overline{\mathcal{F}}(h)}\Big).
    \end{equation*}
Therefore we get, using~\eqref{eq-RDeltas} and the fact that~$\|\op{A}f\|_{\dot{H}^{s,0}}=\sqrt{\mathcal{D}_{0,s}(f)}$,
\[\overline{\mathcal{R}}_{0,s}(f,g,h)=R_{(-\op{\Delta}_x)^s}(f,g,h)\lesssim\sqrt{\nu^{-\theta}\,\mathcal{D}_{0,s}(f)}\Big(\sqrt{\,\overline{\mathcal{F}}(g)\mathcal{F}_{0,s}(h)}+\sqrt{\,\mathcal{F}_{0,s}(g)\overline{\mathcal{F}}(h)}\Big),\]
    which gives for~$s=0$ and~$s=m-\theta$ that
    \begin{equation}\label{estimate-R0-F-1-2}
      \overline{\mathcal{R}}_{0,0}(f,g,h)\lesssim\sqrt{\,\overline{\mathcal{D}}(f)}\sqrt{\nu^{-\theta}\,\overline{\mathcal{F}}(g)\overline{\mathcal{F}}(h)},\quad  \overline{\mathcal{R}}_{0,m-\theta}(f,g,h)\lesssim\sqrt{\,\overline{\mathcal{D}}(f)}\sqrt{\nu^{-\theta}\,\overline{\mathcal{F}}(g)\overline{\mathcal{F}}(h)}.
    \end{equation}
    With the same procedure, we get, thanks to~\eqref{fHm-barF}--\eqref{gradJfHm-barFD} and~\eqref{eq-A-A2-s}--\eqref{eq-gradJf-s}:
\begin{align*}
  % \label{NmAgh-F}
  N_{m,s}(\op{A}g,h)&\leqslant\sqrt{\frac{2\nu^{-\theta}}{\alpha}}\Big(\sqrt{\,\overline{\mathcal{F}}(g)\mathcal{F}_{0,s}(h)}+\sqrt{\,\mathcal{F}_{1,s}(g)\overline{\mathcal{F}}(h)}\Big),\\
  % \label{NmSgh-F}
  N_{m,s}(\op{S}g,h)&\leqslant\sqrt{\frac{2\nu^{-\theta}}{\gamma}}\Big(\sqrt{\,\overline{\mathcal{F}}(g)\mathcal{F}_{0,s}(h)}+\sqrt{\,\mathcal{F}_{1,s}(g)\overline{\mathcal{F}}(h)}\Big), \\
  % \label{NmTgh-F}
  N_{m,s}(\op{T}g,h)&\leqslant\sqrt{\frac{\nu^{-\theta}}{\delta}}\Big(\sqrt{\,\overline{\mathcal{F}}(g)\mathcal{F}_{0,s}(h)}+\sqrt{\,\mathcal{F}_{1,s}(g)\overline{\mathcal{F}}(h)}\Big),\\
  % \label{NmgDh-F}
  N_{m,s}(g,(-\op{\Delta}_x)^{\frac12}h)&\lesssim\sqrt{\frac{\nu^{-\theta}}{\gamma^{1-\frac{\mu}2}\delta^{\frac{\mu}2}}}\Big(\sqrt{\,\overline{\mathcal{F}}(g)\mathcal{F}_{1,s}(h)}+\sqrt{\,\mathcal{F}_{0,s}(g)\overline{\mathcal{F}}(h)}\Big).
\end{align*} 

Therefore we use~\eqref{eq-RA2Deltas}--\eqref{eq-RT2Deltas}, controlling the first term of the product at the right-hand side of the inequalities using the corresponding estimate in~\eqref{eq-A-A2-s}--\eqref{eq-gradJf-s} involving dissipation terms~$\mathcal{D}_{0,s}$ or~$\mathcal{D}_{1,s}$ and we get, thanks to~\Cref{assumption-coeffs},
 \begin{gather*}
R_{-\op{A}^2(-\op{\Delta}_x)^s}(f,g,h)\lesssim\frac{\sqrt{\nu^{-\theta}}}{\alpha}\sqrt{\,\mathcal{D}_{1,s}(f)}\Big(\sqrt{\,\overline{\mathcal{F}}(g)\mathcal{F}_{0,s}(h)}+\sqrt{\,\mathcal{F}_{1,s}(g)\overline{\mathcal{F}}(h)}\Big),\displaybreak[0]\\
  \begin{split} R_{-\op{(AS+SA)}(-\op{\Delta}_x)^s}(f,g,h)&\lesssim\Big(\sqrt{\frac{2}{\gamma}}+\sqrt{\frac{2}{(d-1)\sqrt{\beta\delta}}}\Big)\sqrt{\,\mathcal{D}_{1,s}(f)} \sqrt{\frac{\nu^{-\theta}}{\alpha}}\Big(\sqrt{\,\overline{\mathcal{F}}(g)\mathcal{F}_{0,s}(h)}+\sqrt{\,\mathcal{F}_{1,s}(g)\overline{\mathcal{F}}(h)}\Big)\\
&\lesssim\frac{\sqrt{\nu^{-\theta}}}{\beta}\sqrt{\,\mathcal{D}_{1,s}(f)}\Big(\sqrt{\,\overline{\mathcal{F}}(g)\mathcal{F}_{0,s}(h)}+\sqrt{\,\mathcal{F}_{1,s}(g)\overline{\mathcal{F}}(h)}\Big),
\end{split} \displaybreak[0] \\
\begin{split} R_{-\op{S}^2(-\op{\Delta}_x)^s}(f,g,h)&\lesssim\Big(\sqrt{\frac{2}{\gamma}}+\sqrt{\frac{2}{(d-1)\sqrt{\beta\delta}}}\Big)\sqrt{\,\mathcal{D}_{1,s}(f)}\bigg[ \sqrt{\frac{2\nu^{-\theta}}{\gamma}}\Big(\sqrt{\,\overline{\mathcal{F}}(g)\mathcal{F}_{0,s}(h)}+\sqrt{\,\mathcal{F}_{1,s}(g)\overline{\mathcal{F}}(h)}\Big)\\
   &\hspace{6cm}+\sqrt{\frac{\nu^{-\theta}}{\gamma^{1-\frac{\mu}2}\delta^{\frac{\mu}2}}}\Big(\sqrt{\,\overline{\mathcal{F}}(g)\mathcal{F}_{1,s}(h)}+\sqrt{\,\mathcal{F}_{0,s}(g)\overline{\mathcal{F}}(h)}\Big)\bigg],\\
 &\lesssim\frac{\sqrt{\nu^{-\theta}}}{\gamma^{1-\frac{\mu}4}\delta^{\frac{\mu}4}}\sqrt{\,\mathcal{D}_{1,s}(f)}\Big(\sqrt{\,\overline{\mathcal{F}}(g)\big(\mathcal{F}_{0,s}(h)+\mathcal{F}_{1,s}(h)\big)}+\sqrt{\big(\mathcal{F}_{1,s}(g)+\mathcal{F}_{0,s}(g)\big)\overline{\mathcal{F}}(h)}\Big),
\end{split}\displaybreak[0]\\
\begin{split} R_{-\op{T}^2(-\op{\Delta}_x)^s}(f,g,h)&\lesssim\Big(\sqrt{\frac{1}{\delta}}+\sqrt{\frac{1}{\beta}}\Big)\sqrt{\,\mathcal{D}_{1,s}(f)}\bigg[ \sqrt{\frac{\nu^{-\theta}}{\delta}}\Big(\sqrt{\,\overline{\mathcal{F}}(g)\mathcal{F}_{0,s}(h)}+\sqrt{\,\mathcal{F}_{1,s}(g)\overline{\mathcal{F}}(h)}\Big)\\
   &\hspace{5cm}+\sqrt{\frac{\nu^{-\theta}}{\gamma^{1-\frac{\mu}2}\delta^{\frac{\mu}2}}}\Big(\sqrt{\,\overline{\mathcal{F}}(g)\mathcal{F}_{1,s}(h)}+\sqrt{\,\mathcal{F}_{0,s}(g)\overline{\mathcal{F}}(h)}\Big)\bigg],\\
 &\lesssim\frac{\sqrt{\nu^{-\theta}}}{\delta}\sqrt{\,\mathcal{D}_{1,s}(f)}\Big(\sqrt{\,\overline{\mathcal{F}}(g)\big(\mathcal{F}_{0,s}(h)+\mathcal{F}_{1,s}(h)\big)}+\sqrt{\big(\mathcal{F}_{1,s}(g)+\mathcal{F}_{0,s}(g)\big)\overline{\mathcal{F}}(h)}\Big).
\end{split}
\end{gather*}
These four estimates give, since~$\frac{\gamma}{\delta}\geqslant1$, that
\begin{align*}
  \mathcal{R}_{1,s}(f,g,h)&=R_{\big(-\alpha\,\op{A}^2-\beta\,(\op{AS+SA})-\gamma\,\op{S}^2-\delta\,\op{T}^2\big)(-\op{\Delta}_x)^s}(f,g,h)\\
  &\lesssim\sqrt{\nu^{-\theta}}\,\Big(\frac{\gamma}{\delta}\Big)^{\frac{\mu}4}\,\sqrt{\,\mathcal{D}_{1,s}(f)}\Big(\sqrt{\,\overline{\mathcal{F}}(g)\big(\mathcal{F}_{0,s}(h)+\mathcal{F}_{1,s}(h)\big)}+\sqrt{\big(\mathcal{F}_{1,s}(g)+\mathcal{F}_{0,s}(g)\big)\overline{\mathcal{F}}(h)}\Big).
\end{align*}
Summing this last estimate for~$s=0$,~$s=m-\theta$, and~$s=m-\theta+1$ (multiplied in this case by~$\nu$, and using~\eqref{eq-nu-sp1} to transform the occurences of~$\mathcal{F}_{0,m-\theta+1}$ and~$\mathcal{D}_{0,m-\theta+1}$), and adding~\eqref{estimate-R0-F-1-2}, we obtain~\eqref{estimate-bar-R-D-F2} according to~\Cref{def-F-D-B-R-combined}.

To produce the second estimate~\eqref{estimate-bar-R-D-F2}, we need to control terms of the form~$N_{m,s}(\cdot,\cdot)$ thanks to two types of inequalities. When $\op{A}g$,~$\op{S}g$, $\op{T}g$ or~$(-\op{\Delta}_x)^{\frac12}h$ is one of the argument, we may always chose an inequality involving~$\overline{\mathcal{D}}$ among~\eqref{AfHm-barFD}--\eqref{gradJfHm-barFD} when it concerns a~$H^m$ norm, or involving~$\mathcal{D}_{0,s}$ or~$\mathcal{D}_{1,s}$ among~\eqref{eq-A-A2-s}--\eqref{eq-gradJf-s} when it concerns a~$\dot{H}^s$ norm. Using the same strategy as before, we obtain
\begin{align*}
  % \label{NmAgh-D}
  N_{m,s}(\op{A}g,h)&\leqslant\sqrt{\nu^{-\theta}}\Big(\sqrt{\,\overline{\mathcal{D}}(g)\mathcal{F}_{0,s}(h)}+\sqrt{\,\mathcal{D}_{0,s}(g)\overline{\mathcal{F}}(h)}\Big),\\
  % \label{NmSgh-D}
  N_{m,s}(\op{S}g,h)&\leqslant\sqrt{\frac{\nu^{-\theta}}{\beta}}\Big(\sqrt{\,\overline{\mathcal{D}}(g)\mathcal{F}_{0,s}(h)}+\sqrt{\,\mathcal{D}_{1,s}(g)\overline{\mathcal{F}}(h)}\Big), \\
  % \label{NmTgh-D}
  N_{m,s}(\op{T}g,h)&\leqslant\sqrt{\frac{2\nu^{-\theta}}{(d-1)\sqrt{\beta\delta}}}\Big(\sqrt{\,\overline{\mathcal{D}}(g)\mathcal{F}_{0,s}(h)}+\sqrt{\,\mathcal{D}_{1,s}(g)\overline{\mathcal{F}}(h)}\Big),\\
  % \label{NmgDh-D}
  N_{m,s}(g,(-\op{\Delta}_x)^{\frac12}h)&\lesssim\sqrt{\frac{\nu^{-\theta}}{\beta^{1-\frac{\mu}4}\delta^{\frac{\mu}4}}}\Big(\sqrt{\,\overline{\mathcal{F}}(g)\mathcal{D}_{1,s}(h)}+\sqrt{\,\mathcal{F}_{0,s}(g)\overline{\mathcal{D}}(h)}\Big).
\end{align*} 
However, the last quantity needed to control,~$N_{m,s}(g,h)$ cannot be directly obtained exactly the same way, since we do not have an estimate with~$D_{0,s}$ corresponding to~\eqref{fHm-barF} with~$\overline{\mathcal{D}}$ instead. But we use the fact that~$\|\mathbb{J}[h]\|_{\dot{H}^{s}(\mathbb{R}^d)}\leqslant\frac{1}{\sqrt{d(d-1)}}\|\op{A}h\|_{\dot{H}^{s}(\mathbb{R}^d)}$, thanks to~\Cref{lemmeAU}--\ref{U2}, applied to~$(-\op{\Delta}_x)^{\frac{s}2}h$. Therefore we use~\eqref{eq-A-A2-s} and~\eqref{AfHm-barFD} to obtain
\begin{equation}
\label{Nmgh-D}N_{m,s}(g,h)\leqslant\sqrt{\frac{2\nu^{-\theta}}{d(d-1)\alpha}}\Big(\sqrt{\,\overline{\mathcal{F}}(g)\mathcal{D}_{0,s}(h)}+\sqrt{\,\mathcal{F}_{0,s}(g)\overline{\mathcal{D}}(h)}\Big).
\end{equation}
We can now procede as before, using~\eqref{eq-RDeltas}--\eqref{eq-RT2Deltas}, controlling the first term of the product at the right-hand side of the inequalities using the corresponding estimate in~\eqref{eq-A-A2-s}--\eqref{eq-gradJf-s} involving dissipation terms~$\mathcal{D}_{0,s}$ or~$\mathcal{D}_{1,s}$ and we get,
\begin{align*}
  R_{(-\op{\Delta}_x)^s}(f,g,h) &\lesssim\sqrt{\frac{\nu^{-\theta}}{\alpha}}\sqrt{\,\mathcal{D}_{0,s}(f)}\Big(\sqrt{\,\overline{\mathcal{F}}(g)\mathcal{D}_{0,s}(h)}+\sqrt{\,\mathcal{F}_{0,s}(g)\overline{\mathcal{D}}(h)}\Big)\displaybreak[0]\\
  R_{-\op{A}^2(-\op{\Delta}_x)^s}(f,g,h)&\lesssim\frac{\sqrt{\nu^{-\theta}}}{\alpha}\sqrt{\,\mathcal{D}_{1,s}(f)}\Big(\sqrt{\,\overline{\mathcal{F}}(g)\mathcal{D}_{0,s}(h)}+\sqrt{\,\mathcal{F}_{0,s}(g)\overline{\mathcal{D}}(h)}\\
  &\hspace{4cm}+\sqrt{\alpha\,\overline{\mathcal{D}}(g)\mathcal{F}_{0,s}(h)}+\sqrt{\alpha\,\mathcal{D}_{0,s}(g)\overline{\mathcal{F}}(h)}\Big),\displaybreak[0]\\
R_{-\op{(AS+SA)}(-\op{\Delta}_x)^s}(f,g,h)&\lesssim\frac{\sqrt{\nu^{-\theta}}}{\beta}\sqrt{\,\mathcal{D}_{1,s}(f)}\Big(\sqrt{\,\overline{\mathcal{F}}(g)\mathcal{D}_{0,s}(h)}+\sqrt{\,\mathcal{F}_{0,s}(g)\overline{\mathcal{D}}(h)}\\
  &\hspace{4cm}+\sqrt{\alpha\,\overline{\mathcal{D}}(g)\mathcal{F}_{0,s}(h)}+\sqrt{\alpha\,\mathcal{D}_{0,s}(g)\overline{\mathcal{F}}(h)}\Big),\displaybreak[0]\\
R_{-\op{S}^2(-\op{\Delta}_x)^s}(f,g,h)&\lesssim\sqrt{\frac{\nu^{-\theta}}{\gamma\,\beta^{1-\frac{\mu}4}\delta^{\frac{\mu}4}}}\sqrt{\,\mathcal{D}_{1,s}(f)}\Big(\sqrt{\,\overline{\mathcal{D}}(g)\mathcal{F}_{0,s}(h)}+\sqrt{\,\mathcal{D}_{1,s}(g)\overline{\mathcal{F}}(h)}\\
   &\hspace{5cm}+\sqrt{\,\overline{\mathcal{F}}(g)\mathcal{D}_{1,s}(h)}+\sqrt{\,\mathcal{F}_{0,s}(g)\overline{\mathcal{D}}(h)}\Big),\displaybreak[0]\\
R_{-\op{T}^2(-\op{\Delta}_x)^s}(f,g,h)&\lesssim\sqrt{\frac{\nu^{-\theta}}{\delta\sqrt{\beta\delta}}}\sqrt{\,\mathcal{D}_{1,s}(f)}\Big(\sqrt{\,\overline{\mathcal{D}}(g)\mathcal{F}_{0,s}(h)}+\sqrt{\,\mathcal{D}_{1,s}(g)\overline{\mathcal{F}}(h)}\\
   &\hspace{4.2cm}+\sqrt{\,\overline{\mathcal{F}}(g)\mathcal{D}_{1,s}(h)}+\sqrt{\,\mathcal{F}_{0,s}(g)\overline{\mathcal{D}}(h)}\Big).
\end{align*}
The first estimate provides
\begin{equation}\label{estimate-R0-D-1-2}
      \overline{\mathcal{R}}_{0,0}(f,g,h)\lesssim\sqrt{\frac{\nu^{-\theta}}{\alpha}}\sqrt{\,\overline{\mathcal{D}}(f)\overline{\mathcal{F}}(g)\overline{\mathcal{D}}(h)},\quad  \overline{\mathcal{R}}_{0,m-\theta}(f,g,h)\lesssim\sqrt{\frac{\nu^{-\theta}}{\alpha}}\sqrt{\,\overline{\mathcal{D}}(f)\overline{\mathcal{F}}(g)\overline{\mathcal{D}}(h)},
    \end{equation}
and the four next estimates allow to write, since $\gamma=\gamma^{1-\frac{\mu}2}\gamma^{\frac{\mu}2}\leqslant\beta^{1-\frac{\mu}2}\sqrt{\beta\delta}^{\frac{\mu}2}=\beta^{1-\frac{\mu}4}\delta^{\frac{\mu}4}$, and~$\delta\leqslant\beta$ by~\Cref{assumption-coeffs}:
\begin{align*}
  \mathcal{R}_{1,s}(f,g,h)&=R_{\big(-\alpha\,\op{A}^2-\beta\,(\op{AS+SA})-\gamma\,\op{S}^2-\delta\,\op{T}^2\big)(-\op{\Delta}_x)^s}(f,g,h)\\
&\lesssim\sqrt{\nu^{-\theta}}\,\sqrt{\,\mathcal{D}_{1,s}(f)}\Big(\sqrt{\,\overline{\mathcal{F}}(g)\mathcal{D}_{0,s}(h)}+\sqrt{\,\overline{\mathcal{F}}(g)\mathcal{D}_{1,s}(h)\big)}+\sqrt{\,\mathcal{F}_{0,s}(g)\overline{\mathcal{D}}(h)}\\
  &\hspace{4cm}+\sqrt{\,\mathcal{D}_{0,s}(g)\overline{\mathcal{F}}(h)}+\sqrt{\,\mathcal{D}_{1,s}(g)\overline{\mathcal{F}}(h)}+\sqrt{\,\overline{\mathcal{D}}(g)\mathcal{F}_{0,s}(h)}\Big).
\end{align*}
As before, we sum this last estimate for~$s=0$,~$s=m-\theta$, and~$s=m-\theta+1$ (multiplied in this case by~$\nu$, and using~\eqref{eq-nu-sp1} to transform the occurrences of~$\mathcal{F}_{0,m-\theta+1}$ and~$\mathcal{D}_{0,m-\theta+1}$), and add~\eqref{estimate-R0-D-1-2}, using~$\alpha\leqslant1$ to get~\eqref{estimate-bar-R-D2-F} according to~\Cref{def-F-D-B-R-combined}.
\end{proof} 
\subsection{Well-posedness through a fixed point method : proof of~\texorpdfstring{\Cref{thm-well-posedness-Hs0}}{Theorem 1.2}.}
\label{subsec-wellposed-Hs0}
We proceed as in~\Cref{sec-cauchy-l-infini} for our fixed point iterative mapping. Let~$h$ be a smooth function, and let~$f$ be the solution of the linear equation~\eqref{eq-FP-Lin} with~$\mathcal{J}=\mathbb{J}[h]$ and initial condition~$f^0$. This equation can be rewritten under the condensed form
\begin{equation*}
  %\label{eq-VFP-fix-point-condensed}
  \partial_tf+\op{T}f+\frac1d\,\op{A}(f\op{U}h)=\op{A}^2f,
\end{equation*}
and we denote~$\Psi_{f^0}(h)$ such a solution $f$. Our goal is to obtain a fixed point for~$\Psi_{f^0}$, in a space corresponding to our nonlinear functional. Let us take another function~$\widetilde{h}$, and let~$\widetilde{f}=\Psi_{\widetilde{f}^0}(\widetilde{h})$. We denote~$g=f-\widetilde{f}$ and~$u=h-\widetilde{h}$, and we have
\begin{equation*}
  %\label{eq-VFP-diff-condensed}
  \partial_tg+\op{T}g+\frac1d\,\op{A}(g\op{U}h+\widetilde{f}\op{U}u)=\op{A}^2g,
\end{equation*}

Therefore, similarly to~\eqref{ddtQXnonlin}, and using \Cref{def-RX}, we obtain:
\begin{equation*}
  \frac{\d}{\d t}Q_{\op{X}}(g)=Q_{\Phi(\op{X})}(g)+\frac2d\big(R_{\op{X}}(g,g,h)+R_{\op{X}}(g,\widetilde{f},u)\big),
\end{equation*}
and thus, thanks to~\Cref{def-Fs-Ds-Bs-Rs,def-F-D-B-R-combined},
  \begin{equation}
    \label{ddtFf-combined-linear-h}
    \frac{\d}{\d t}\overline{\mathcal{F}}(g)=-2\overline{\mathcal{D}}(g)+\overline{\mathcal{B}}(g)+\partial_\tau\overline{\mathcal{F}}(g)+\frac2d\big(\overline{\mathcal{R}}(g,g,h)+\overline{\mathcal{R}}(g,\widetilde{f},u)\big),
  \end{equation}
 still remembering that when~$g$ depends on time~$t$ and the context is clear, we may still write~$\overline{\mathcal{F}}(g)$ for~$\overline{\mathcal{F}}(t,g(t,\cdot))$ (that is to say with~$\tau=t$), and similarly for the other functionals~$\overline{\mathcal{D}}$,~$\overline{\mathcal{B}}$ and~$\overline{\mathcal{R}}$.

 We now proceed as in~\Cref{subsec-short-time}, by setting~$(\alpha(\tau),\beta(\tau),\gamma(\tau),\delta(\tau),\nu(\tau))=(\alpha_0\tau,\beta_0\tau^2,\gamma_0\tau^3,\delta_0\tau^4,\nu_0\tau^4)$.
  Thanks to~\Cref{prop-bar-B-dtauF}, we can ensure for~$\tau\leqslant1$ that~$\overline{\mathcal{B}}(g)$ and~$\partial_\tau\overline{\mathcal{F}}(g)$ are controlled by~$\varepsilon\overline{\mathcal{D}}(g)$, with~$\varepsilon$ arbitrary small, as soon as we have the same conditions on the coefficients as in~\eqref{conditions-coeffs0}:
  \[ \alpha_0\ll\sqrt{\beta_0},\quad \beta_0\ll1,\quad \gamma_0\ll\sqrt{\beta_0\delta_0}, \quad \delta_0\ll\beta_0. 
  \]
  This can be satisfied by chosing for instance~$(\beta_0,\gamma_0,\delta_0)=(\alpha_0^{\frac32},4\alpha_0^2,4\alpha_0^2)$ with~$\alpha_0$ sufficiently small. Notice that this choice of coefficients also implies that~\Cref{assumption-coeffs} is satisfied, as soon as~$\tau\leqslant1$, if we take~$\nu_0\leqslant\min(1,\frac{\gamma_0}2,\delta_0)=2\alpha_0^2$. We therefore fix such coefficients in order to have, for all~$\tau\leqslant1$ and all smooth function~$g$:
  \begin{equation}\label{estimate-bar-B-dtauF-1-D}
    \overline{\mathcal{B}}(g)+\partial_\tau\overline{\mathcal{F}}(g)\leqslant\overline{\mathcal{D}}(g).
  \end{equation}
  Notice that this choice of coefficients only depends on the dimension~$d$, and could be expressed explicitly by tracking all the numeric constants appearing in~\Cref{prop-bar-B-dtauF}.

  We now use the assumption~$s>\frac{d}2-\frac14$ (for~$d\geqslant3$) or $s>\frac78$ (for~$d=2$). When~$d\geqslant4$, we set~$\mu=0$, when~$d=2$, we set~$\mu=1$, and when~$d=3$ we pick~$\mu\in(0,1]$ sufficiently small, so that in all case we get~$\frac14-\frac{\mu}8>\frac{d}2-s$. We then take $\theta\in(\frac{d}2-s,\frac14-\frac{\mu}8)$ and set~$m=s+\theta$, so we get~$m>\frac{d}2$ and~$4\theta+\frac{\mu}2<1$.
  Using~\Cref{prop-estimates-barR} and Young’s inequality, we have therefore for all~$\tau\leqslant1$:
  \begin{align*} \frac{2}d\big(\overline{\mathcal{R}}(g,g,h)+\overline{\mathcal{R}}(g,\widetilde{f},u)\big)&\lesssim\frac2d\,\frac{\sqrt{\nu_0^{-\theta}}}{\tau^{2\theta+\frac{\mu}4}}\,\Big(\frac{\gamma_0}{\delta_0}\Big)^{\frac{\mu}4}\sqrt{\,\overline{\mathcal{D}}(g)}\Big(\sqrt{\overline{\mathcal{F}}(g)\overline{\mathcal{F}}(h)}+\sqrt{\overline{\mathcal{F}}(\widetilde{f})\overline{\mathcal{F}}(u)}\Big),\\
    &\leqslant\overline{\mathcal{D}}(g)+\frac{C}{\tau^{4\theta+\frac{\mu}2}}\big(\overline{\mathcal{F}}(g)\overline{\mathcal{F}}(h)+\overline{\mathcal{F}}(\widetilde{f})\overline{\mathcal{F}}(u)\big),
  \end{align*}
  where the constant~$C$ only depends on $d$,~$m$, $\theta$ and~$\mu$ (so only on~$d$ and~$s$).

  Let us now write back the dependence on~$\tau$ explicitly, With~\eqref{ddtFf-combined-linear-h}, we therefore obtain that for all~$t\leqslant1$, now writing back explicitly the dependence on~$t$:

  \begin{equation}\label{dt-barF-ghfu}
    \frac{\d}{\d t}\overline{\mathcal{F}}(t,g(t,\cdot))\leqslant\frac{C}{t^{4\theta+\frac{\mu}2}}\big(\overline{\mathcal{F}}(t,g(t,\cdot))\overline{\mathcal{F}}(t,h(t,\cdot))+\overline{\mathcal{F}}(t,\widetilde{f}(t,\cdot))\overline{\mathcal{F}}(t,u(t,\cdot))\big).
  \end{equation}

  We are now ready to define the Banach space in which we will look for our fixed point. We define~$\mathcal{X}_{T}$ as the space of functions~$h$ such that~$\overline{\mathcal{F}}(t,h(t,\cdot))$ is uniformly bounded on~$[0,T]$, endowed with the norm~$\|h\|_{\mathcal{X}_T}^2=\sup_{t\in[0,T]}e^{-Lt}\overline{\mathcal{F}}(t,h(t,\cdot))$, for a constant~$L>0$ that will be chosen later on. Notice that~$\mathcal{X}_{T}\subset L^\infty([0,T],H^{s,0}(\mathbb{R}^d\times\S))$ thanks to~\Cref{def-F-D-B-R-combined} as we have~$s=m-\theta$.
 
  We set~$\kappa=1-4\theta-\frac{\mu}2>0$. We fix~$M>0$, let~$T_M=\big(\frac{\kappa}{CM}\big)^{\frac1{\kappa}}$ and fix~$T<T_M$ (with~$T\leqslant1$). Finally the closed set~$\Omega_M$ is defined by
  \[\Omega_M=\Big\{h\in\mathcal{X}_{T}, \quad\forall t\in[0,T], \qquad \overline{\mathcal{F}}(t,h(t,\cdot))\leqslant\frac{M}{1-(\frac{t}{T_M})^{\kappa}}\Big\},\]

Let us show that~$\Psi_{f^0}$ is a contraction on~$\Omega_M$ when we suppose that~$\|f^0\|_{H^{s,0}}^2=\overline{\mathcal{F}}(0,f^0)\leqslant M$. 
  By density, it is sufficient to prove the stability and contracting properties on smooth compactly supported functions of~$\Omega_M$.
  If~$h\in\Omega_M$, we have from~\eqref{dt-barF-ghfu} and the definition of~$\kappa$ and~$T_M$:
  \begin{equation*}%\label{dt-barF-gfu-reduced}
    \frac{\d}{\d t}\overline{\mathcal{F}}(t,g(t,\cdot))\leqslant\frac{\frac{\kappa}{T_M}(\frac{t}{T_M})^{\kappa-1}}{1-(\frac{t}{T_M})^{\kappa}}\overline{\mathcal{F}}(t,g(t,\cdot))+\tfrac{\kappa}{MT_M}(\tfrac{t}{T_M})^{\kappa-1}\overline{\mathcal{F}}(t,\widetilde{f}(t,\cdot))\overline{\mathcal{F}}(t,u(t,\cdot)),
  \end{equation*}
  and solving this differential inequality we obtain
  \begin{equation}\label{barF-solved} \overline{\mathcal{F}}(t,g(t,\cdot))\leqslant\frac{1}{1-(\frac{t}{T_M})^{\kappa}}\bigg(\overline{\mathcal{F}}(0,g^0)+\int_0^t\tfrac{\kappa}{MT_M}(\tfrac{s}{T_M})^{\kappa-1}\Big(1-(\tfrac{s}{T_M})^{\kappa}\Big)\overline{\mathcal{F}}(s,\widetilde{f}(s,\cdot))\overline{\mathcal{F}}(s,u(s,\cdot))\d s\bigg).
  \end{equation}

  For the case~$\widetilde{f}=0$, we get~$g=f$ and so we have
  \begin{equation}\label{eq-bound-F-local-time}
    \overline{\mathcal{F}}(t,f(t,\cdot))\leqslant\frac{\overline{\mathcal{F}}(0,f^0)}{1-(\frac{t}{T_M})^{\kappa}},
  \end{equation}
  thus proving~$f=\Psi_{f^0}(h)\in\Omega_M$, which gives the stability. If now~$\widetilde{h}\in\Omega_M$, we get~$\widetilde{f}=\Psi_{f^0}(\widetilde{h})\in\Omega_M$, and the estimate~\eqref{barF-solved} becomes
  \begin{align*}
    \overline{\mathcal{F}}(t,g(t,\cdot))&\leqslant\frac{1}{1-(\frac{t}{T_M})^{\kappa}}\bigg(\overline{\mathcal{F}}(0,g^0)+\int_0^t\tfrac{\kappa}{T_M}(\tfrac{s}{T_M})^{\kappa-1}\overline{\mathcal{F}}(s,u(s,\cdot))\d s\bigg)\\
                                    &\leqslant\frac{\overline{\mathcal{F}}(0,g^0)}{1-(\frac{t}{T_M})^{\kappa}}+\frac{\tfrac{\kappa}{T_M}(\tfrac{T}{T_M})^{\kappa-1}}{1-(\frac{T}{T_M})^{\kappa}}\int_0^te^{Ls}\|u\|_{\mathcal{X}_T}^2\d s,
  \end{align*}
  which gives
  \begin{align*} \|\Psi_{f^0}(h)-\Psi_{\widetilde{f}^0}(\widetilde{h})\|_{\mathcal{X}_T}^2=\|g\|_{\mathcal{X}_T}^2&\leqslant\frac{\overline{\mathcal{F}}(0,g^0)}{1-(\frac{T}{T_M})^{\kappa}}+\frac{\tfrac{\kappa}{T_M}(\tfrac{T}{T_M})^{\kappa-1}}{L\Big(1-(\frac{T}{T_M})^{\kappa}\Big)}\|u\|_{\mathcal{X}_T}^2\\
    &\leqslant\frac{\|f^0-\widetilde{f}^0\|_{H^{s,0}}^2}{1-(\frac{T}{T_M})^{\kappa}}+\frac{\tfrac{\kappa}{T_M}(\tfrac{T}{T_M})^{\kappa-1}}{L\Big(1-(\frac{T}{T_M})^{\kappa}\Big)}\|h-\widetilde{h}\|_{\mathcal{X}_T}^2,
  \end{align*}
  thus providing, for~$L$ sufficiently large, both the contraction property (when~$\widetilde{f}^0=f^0$) and the continuity in~$\mathcal{X}_T$ of the unique fixed point with respect to the initial condition. This ends the proof of~\Cref{thm-well-posedness-Hs0}, since the estimates in time come from the bound~\eqref{eq-bound-F-local-time} (remember~$\overline{\mathcal{F}}(0,f^0)=\|f\|_{H^{m-\theta,0}}^2$ and~$s=m-\theta$) and the estimates~\eqref{eq-A-A2-s}--\eqref{eq-T-TA-s}.

\section{Nonlinear stability of isotropic equilibrium up to the phase transition threshold}\label{hypocoercivity-nonlinear}

\subsection{Evolution of the perturbation and refined estimates.} We fix~$0\leqslant\rho<d$ and consider the solution of~\eqref{eq-VFP} with initial condition~$f^0=\rho+g^0$, where~$g^0\in H^{s,0}$. The evolution can be written in condensed form:
\begin{equation}
  \label{eq-VFPg-condensed}
  \partial_tg+\op{T}g+\frac1d\,\op{A}(\rho\op{U}g+g\op{U}g)=\op{A}^2g.
\end{equation}
As was done in \Cref{subsec-wellposed-Hs0}, similarly to~\eqref{ddtQXnonlin}, and using \Cref{def-RX}, we obtain:
\begin{equation}\label{eq-ddtQXg}
  \frac{\d}{\d t}Q_{\op{X}}(g)=Q_{\Phi(\op{X})}(g)+\frac2d\big(\rho\,R_{\op{X}}(g,1,g)+R_{\op{X}}(g,g,g)\big),
\end{equation}
and therefore we obtain, thanks to~\Cref{def-Fs-Ds-Bs-Rs,def-F-D-B-R-combined},
  \begin{equation}
    \label{ddtFg-nonlinear}
    \frac{\d}{\d t}\overline{\mathcal{F}}(g)=-2\overline{\mathcal{D}}(g)+\overline{\mathcal{B}}(g)+\partial_\tau\overline{\mathcal{F}}(g)+\frac2d\big(\rho\,\overline{\mathcal{R}}(g,1,g)+\overline{\mathcal{R}}(g,g,g)\big).
  \end{equation}
  We therefore have new quadratic terms~$\overline{\mathcal{R}}(g,1,g)$ that have to be estimated differently compared to~\Cref{prop-estimates-barR}, since~$\overline{\mathcal{F}}(1)$ is not finite, but simplifications will come from simpler integration by parts and the nice properties of the operator~$\op{U}$ given in~\Cref{lemmeAU}. We start by giving the counterpart of~\Cref{prop-estimates-RX}, which now do not require the use of~$H^m$ norms.
  \begin{lemma}\label{lemma-RXg1g}
Let $s\geqslant0$ (not necessarily integers). For any smooth and compactly supported function~$g$, we have the following estimates:
  \begin{align}
    R_{(-\op{\Delta}_x)^s}(g,1,g) &\leqslant \|\op Ag\|_{\dot{H}^{s,0}}^2, \label{eq-Rg1g-Deltas}\\
    R_{-\op A^2(-\op{\Delta}_x)^s}(g,1,g) &\leqslant(d-1)\|\op Ag\|_{\dot{H}^{s,0}}^2, \label{eq-Rg1g-A2Deltas}\\
    R_{-(\op{AS+SA})(-\op{\Delta}_x)^s}(g,1,g) &\leqslant(d-1)\big(\|\op Sg\|_{\dot{H}^{s,0}}+\|\nabla_x\op{\Pi}_1g\|_{\dot{H}^{s,0}}\big)\|\op Ag\|_{\dot{H}^{s,0}},\label{eq-Rg1g-ASpSADeltas}\\
    R_{-\op{S}^2(-\op{\Delta}_x)^s}(g,1,g) &\leqslant (d-1) \|\op Sg\|_{\dot{H}^{s,0}}\|\nabla_x\op{\Pi}_1g\|_{\dot{H}^{s,0}},\label{eq-Rg1g-S2Deltas} \\
    R_{-\op{T}^2(-\op{\Delta}_x)^s}(g,1,g) &\leqslant (d-1) \|\op Tg\|_{\dot{H}^{s,0}}\|\nabla_x\op{\Pi}_1g\|_{\dot{H}^{s,0}}.\label{eq-Rg1g-T2Deltas}.
  \end{align}  \end{lemma}
\begin{proof}[\textbf{Proof of~\Cref{lemma-RXg1g}}] Since~$(-\op{\Delta}_x)^{\frac{s}2}$ commutes with~$\op{A}$,~$\op{T}$,~$\op{S}$ and~$\op{U}$, it is sufficient to prove the estimates for~$s=0$ and apply them to~$(-\op{\Delta}_x)^{\frac{s}2}g$ to get the general case.

  We have~$R_{\op{Id}}(g,1,g)=\langle\op{A}g,\op{U}g\rangle_{L^2}=\|\op{U}g\|^2_2\leqslant\|\op{A}g\|^2_2$ thanks to~\Cref{lemmeAU}--\ref{U2}, and this gives~\eqref{eq-Rg1g-Deltas}. Then similarly~$R_{-\op{A^2}}(g,1,g)=\langle-\op{A}^3g,\op{U}g\rangle_{L^2}=(d-1)\|\op{U}g\|^2_2\leqslant(d-1)\|\op{A}g\|^2_2$ thanks to~\Cref{lemmeAU}--\ref{AU2}, giving~\eqref{eq-Rg1g-A2Deltas}.
We then have
\begin{align*}
  R_{-(\op{AS+SA})}(g,1,g)&=-\langle\op{S}g,\op{A}^2\op{U}g\rangle_{L^2}-\langle\op{A}g,\op{S}\op{AU}g\rangle_{L^2},\\
  R_{-\op{S}^2}(g,1,g)&=-\langle\op{A}g,\op{S}\op{AU}g\rangle_{L^2} ,\\
  R_{-\op{T}^2}(g,1,g)&= -\langle\op{A}g,\op{T}\op{AU}g\rangle_{L^2}.
\end{align*}
Let us recall that~$\op{\Pi}_1g=d\,v\cdot\mathbb{J}[g]$, thus thanks to~\Cref{lemmeAU}--\ref{U2}, we have~$\op{AU}g=-(d-1)\op{\Pi}_1g$. Using the fact that~$\op{\Delta}_x=\op{T}^2+\op{S}^2$ to bound both~$\|\op{S}\op{AU}g\|_2^2$ and~$\|\op{S}\op{AU}g\|_2^2$ by~$(d-1)^2\|\nabla_x\op{\Pi}_1g\|^2$, we use Cauchy--Schwarz inequality and obtain~\eqref{eq-Rg1g-ASpSADeltas}--\eqref{eq-Rg1g-T2Deltas}.
\end{proof}

We are now ready, thanks to~\Cref{lemma-Sg2}, to give conditions on the coefficients in order to ensure that~$\overline{\mathcal{R}}(g,1,g)$ is sufficiently small.
\begin{proposition}\label{prop-control-Rg1g}
  We let~$\mu=1$ for~$d=2$,~$\mu\in(0,1]$ for~$d=3$, or~$\mu=0$ for~$d=4$.

For all~$s\geqslant0$, we have~$\mathcal{R}_{0,s}(g,1,g)\leqslant\mathcal{D}_{0,s}(g,1,g)$ and $\mathcal{R}_{1,s}(g,1,g)\leqslant\varepsilon(\mathcal{D}_{0,s}(g,1,g)+\mathcal{D}_{1,s}(g,1,g))$ for $\varepsilon$ as small as we want, as soon as we have~\Cref{assumption-coeffs} and the following conditions on the coefficients:
  \begin{equation}\label{conditions-coeffs-Rg1g}
    \alpha\ll1,\quad \beta\ll\delta^{\frac{\mu}{4+\mu}},\quad \gamma\ll\beta^{1-\frac{\mu}8}\delta^{\frac{\mu}8},\quad \delta\ll\beta,
  \end{equation}
  where we use the symbol~$\ll$ to denote inequality up to a small positive constant~$C_\varepsilon$ that may depend only on~$\varepsilon$,~$d$ and~$\mu$. 
  Consequently, under these conditions, we have~$\overline{\mathcal{R}}(g,1,g)\leqslant(1+\varepsilon)\overline{\mathcal{D}}(g)$.
\end{proposition}
\begin{proof}[\textbf{Proof of~\Cref{prop-control-Rg1g}}]
Starting from~\Cref{def-Fs-Ds-Bs-Rs}, the conditions obtained here are simple consequences of~\Cref{lemma-RXg1g,lemma-inequalities-s}, using~\eqref{eq-def-Pi1}, and~\Cref{assumption-coeffs} (ensuring that~$\beta\leqslant\delta$).
\end{proof}

\subsection{Control of the solution in short time}
Notice that the choice~$(\alpha,\beta,\gamma,\delta)=(\alpha_0\tau,\beta_0\tau^2,\gamma_0\tau^3,\delta_0\tau^4)$, as was taken in~\Cref{subsec-wellposed-Hs0}, with~$\alpha_0$ sufficiently small and~$(\beta_0,\delta_0,\gamma_0)=(\alpha_0^{\frac32},4\alpha_0^2,4\alpha_0^2)$ allows to satisfy the conditions~\eqref{conditions-coeffs-Rg1g} for all~$\mu\in[0,1]$ and~$\tau\leqslant1$.

Therefore we may proceed similarly as in~\Cref{subsec-wellposed-Hs0}. If~$s>\frac{d}2-\frac14$ (or~$s>\frac78$ for~$d=2$) we pick~$m>\frac{d}2$,~$\mu$ and~$\theta$ such that~$s=m-\theta$ and~$\kappa=1-4\theta-\frac{\mu}2>0$. 
Using~\Cref{prop-control-Rg1g} and~\Cref{prop-bar-B-dtauF}, we proceed as in~\eqref{estimate-bar-B-dtauF-1-D} to find appropriate coefficients such that for all~$\tau\leqslant1$, we have
\[\overline{\mathcal{B}}(g)+\partial_\tau\overline{\mathcal{F}}(g)+\frac{2\rho}d\,\overline{\mathcal{R}}(g,1,g)\leqslant(1-\frac{\rho}d)\overline{\mathcal{D}}(g).\]
Using~\Cref{prop-estimates-barR}, with Young’s inequality, we find then an explicit constant~$C$ depending only on~$d$,~$\rho$ and~$s$ such that
\[\frac{2}{d}\,\overline{\mathcal{R}}(g,g,g)\leqslant(1-\frac{\rho}d)\overline{\mathcal{D}}(g)+\frac{C\overline{\mathcal{F}}(g)^2}{\tau^{1-\kappa}}.\]
Inserting this estimates into~\eqref{ddtFg-nonlinear}, we obtain for all~$t\leqslant1$:

\begin{equation*}%\label{dt-barFg-barFg2}
    \frac{\d}{\d t}\overline{\mathcal{F}}(t,g(t,\cdot))\leqslant\frac{C\overline{\mathcal{F}}(t,g(t,\cdot))^2}{t^{1-\kappa}}.
\end{equation*}
By letting~$T_{g^0}=\big(\frac{\kappa}{C\overline{\mathcal{F}}(0,g^0)}\big)^\kappa$ this inequality may be solved on~$[0,T_{g^0})$, giving
\[\overline{\mathcal{F}}(t,g(t,\cdot))\leqslant\frac{\overline{\mathcal{F}}(0,g^0)}{1-(\frac{t}{T_{g^0}})^\kappa}.\]
By taking~$\eta$ sufficiently small, we may for instance ensure~$T_{g^0}\geqslant2$ as soon as~$\overline{\mathcal{F}}(0,g^0)\leqslant\eta^2$, and therefore, thanks to the equivalence of~$\overline{\mathcal{F}}$ with Sobolev norms, we have proven the following statement.

\begin{proposition}\label{prop-short-time-nonlinear}
  Let~$s>\frac{d}2-\frac14$ (or~$s>\frac78$ when~$d=2$). There exist explicit positive constants~$\eta$ and~$C$, depending only on~$\rho$,~$d$ and~$s$, such that any solution~$g$ of~\eqref{eq-VFPg-condensed} with initial condition~$g^0$ such that~$\|g^0\|_{H^{s,0}}\leqslant\eta$ is well-defined up to time~$t=1$, and we have
  \[\forall t\in[0,1],\|g(t,\cdot)\|_{H^{s,0}}\leqslant C\|g^0\|_{H^{s,0}},\quad \|g(1,\cdot)\|_{H^{s+2,1}}\leqslant C\|g^0\|_{H^{s,0}}.\]
\end{proposition}

In view of the forthcoming study of the large time behaviour in the case of the whole space, we also need to be able to control~$\mathcal{M}(g)=\int_{\mathbb{R}^d\times\S}|x|^2g^2\,\dd x \dd v$ in this short time interval we can directly apply~\eqref{eq-ddtQXg}, seeing that~$\mathcal{M}(t)=Q_{|x|^2\op{Id}}$. We immediately compute~$\Phi(|x|^2\op{Id})=2|x|^2\op{A}^2+2v\cdot x\,\op{Id}$ thanks to its definition~\eqref{def-Phi}. We then have, thanks to~\Cref{lemmeAU}--\ref{U2}:
\[  R_{|x|^2\op{Id}}(g,1,g)=\langle|x|^2\op{A}g,\op{U}g\rangle\leqslant\int_{\mathbb{R}^d\times\S}|x|^2|\op{A}g|^2\,\dd x \dd v,\]
We therefore get, similarly as~\eqref{estimate-ddtM-linear}:
\begin{align}
  \frac{\d}{\d t}\mathcal{M}(g)&\leqslant-2\int_{\mathbb{R}^d\times\S}\Big((1-\tfrac{\rho}d)|x|^2|\op{A}g|^2-v\cdot x\,g^2\Big)\,\dd x \dd v + R_{|x|^2\op{Id}}(g,g,g)\nonumber\\
  &\leqslant-\big(1-\tfrac{\rho}d\big)\int_{\mathbb{R}^d\times\S}\Big(2|x|^2|\nabla_vg|^2-\tfrac{4d}{(d-1)(d-\rho)}g\,x\cdot\nabla_vg\Big)\,\dd x \dd v + R_{|x|^2\op{Id}}(g,g,g)\nonumber\\
  &\leqslant-\big(1-\tfrac{\rho}d\big)\int_{\mathbb{R}^d\times\S}|x|^2|\nabla_vg|^2\,\dd x \dd v +\frac{4\|g\|_2^2}{(d-1)^2(1-\frac{\rho}d)}+R_{|x|^2\op{Id}}(g,g,g).\label{estimate-ddtM-nonlinear}
\end{align}

We have, using~\Cref{def-A-U} and Sobolev embedding:
\begin{align*}
  R_{|x|^2\op{Id}}(g,g,g)=\langle|x|^2\op{A}g,g\op{U}g\rangle&\leqslant d\int_{\mathbb{R}^d\times\S}|x|^2|\nabla_vg||g||\mathbb{J}[g]|\,dv\,\d x\nonumber\\
  &\lesssim\sqrt{\mathcal{M}(g)}\sqrt{\int_{\mathbb{R}^d\times\S}|x|^2|\nabla_vg|^2\,\dd x \dd v}\,\|\mathbb{J}[g]\|_{H^{m}(\mathbb{R}^d)},
\end{align*}
that we can insert into~\eqref{estimate-ddtM-nonlinear}, using Young’s inequality to get
\begin{equation}\label{estimate-ddtM-F-MJ2}
    \frac{\d}{\d t}\mathcal{M}(g)\leqslant C\big(\|g\|_2^2+\mathcal{M}(g)\|\mathbb{J}[g]\|_{H^{m}(\mathbb{R}^d)}^2\big),
\end{equation}
where the constant~$C$ only depends on~$m$,~$d$ and~$\rho$. This estimation will also be used in the next subsection for the behaviour in large time.

In short times, we have the following result.
\begin{proposition}\label{prop-M-short-time}
  Under the assumptions of~\Cref{prop-short-time-nonlinear}, there exists a constant~$C$ depending only on~$\rho$,~$d$ and~$s$ such that if furthermore~$\mathcal{M}(g^0)$ is finite, then $\mathcal{M}(g)$ is bounded for all~$t\leqslant1$:
  \begin{equation*}
    \mathcal{M}(g(t,\cdot))\leqslant C\big(\mathcal{M}(g^0)+\|g^0\|_{H^{s,0}}\big).
  \end{equation*}
\end{proposition}

\begin{proof}[\textbf{Proof of~\Cref{prop-M-short-time}}]
  We use \eqref{estimate-ddtM-F-MJ2}, bounding~$\|\mathbb{J}[g]\|_{H^{m}(\mathbb{R}^d)}$ by~$\|f\|_{H^{m,0}}$ and using~\eqref{fHm-barF}, remembering that~$\nu(\tau)=\nu_0\tau^4$, to obtain for~$t\leqslant1$, thanks to the uniform estimates of~\Cref{prop-short-time-nonlinear}:
     \[\frac{\d}{\d t}\mathcal{M}(g)\leqslant C\Big(\|g^0\|_{H^{s,0}}^2+\mathcal{M}(g)\frac{\|g^0\|_{H^{s,0}}^2}{t^{4\theta}}\Big),\]
 and since~$\theta<\frac14$, $t\mapsto t^{-4\theta}$ is integrable in the neighborhood of~$0$ which allows to solve this differential inequality and end the proof.
\end{proof}
\subsection{Nonlinear stability and behaviour in large time: proof of \texorpdfstring{\Cref{thm-stabilite-nonlineaire-Hs1}}{Theorem 1.3}}

We now take~$m>\frac{d}2$ and~$\theta=0$, and consider constant coefficients~$\alpha$,~$\beta$,~$\gamma$, and~$\delta$.

The second estimate of~\Cref{prop-estimates-barR} then becomes
\[\overline{\mathcal{R}}(g,g,g)\leqslant C\sqrt{\frac{\overline{\mathcal{F}}(g)}{\alpha}}\overline{\mathcal{D}}(g),\]
where~$C$ only depends on~$m$ and~$\mu$.

Since we will not need here higher derivatives, and no estimate involving~$\nu$ will be needed, we may also take~$\nu=0$, simplifying \Cref{def-F-D-B-R-combined} in
  \begin{gather*}
    \overline{\mathcal{F}}(f)=\mathcal{F}_{0,0}(f)+\mathcal{F}_{0,m}(f)+\mathcal{F}_{1,0}(f)+\mathcal{F}_{1,m}(f),\quad 
    \overline{\mathcal{D}}(f)=\mathcal{D}_{0,0}(f)+\mathcal{D}_{0,m}(f)+\mathcal{D}_{1,0}(f)+\mathcal{D}_{1,m}(f),\\
    \overline{\mathcal{B}}(f)=\mathcal{B}_{0}(f)+\mathcal{B}_{m}(f),\quad
    \overline{\mathcal{R}}(f,g,h)=\mathcal{R}_{0,0}(f,g,h)+\mathcal{R}_{0,m}(f,g,h)+\mathcal{R}_{1,0}(f,g,h)+\mathcal{R}_{1,m}(f,g,h).
  \end{gather*}

We therefore now have
  \begin{equation*}
    %\label{ddtFg-nonlinear-notime}
    \frac{\d}{\d t}\overline{\mathcal{F}}(g)=-2\overline{\mathcal{D}}(g)+\overline{\mathcal{B}}(g)+\frac2d\big(\rho\,\overline{\mathcal{R}}(g,1,g)+\overline{\mathcal{R}}(g,g,g)\big),
  \end{equation*}
  and the estimates coming from~\Cref{prop-bar-B-dtauF} and~\Cref{prop-control-Rg1g} allow, for $\varepsilon$ as small as needed, to find coefficients only depending on~$m$,~$d$,~$\mu$ and~$\varepsilon$ (such as~$(\beta,\gamma,\delta)=(\alpha^{\frac32},4\alpha^2,4\alpha^2)$ with $\alpha$ sufficiently small) such that

  \begin{equation*}
    %\label{ddtFg-control-nonlinear-notime}
    \frac{\d}{\d t}\overline{\mathcal{F}}(g)=-2\left(1-\frac{\rho}d-\left(\frac12+\frac{2\rho}d\right)\varepsilon-C\rho\sqrt{\overline{\mathcal{F}}(g)}\right)\overline{\mathcal{D}}(g),
  \end{equation*}
  where~$C$ only depends on~$m$,~$d$,~$\mu$ and~$\varepsilon$. Fixing~$\mu$ arbitrarily in~$(0,1]$ when~$d=3$, and chosing~$\varepsilon$ sufficiently small, we therefore get the following proposition.
  \begin{proposition}\label{prop-decrease-F-nonlinear} Let~$m>\frac{d}2$. There exist positive coefficients~$\alpha$,~$\beta$,~$\gamma$,~$\delta$ satisfying~\Cref{assumption-coeffs}, and positive constants~$\eta$ and~$C$, only depending on~$m$,~$d$ and~$\rho$ such that any solution~$g$ of~\eqref{eq-VFPg-condensed} with initial condition~$g^0$ such that~$\|g^0\|_{H^{m+1,1}}\leqslant\eta$ is well-defined globally in time, and satisfying the following energy~--~energy-dissipation estimate
    \[\frac{\d}{\d t}\overline{\mathcal{F}}(g)\leqslant-C\,\overline{\mathcal{D}}(g).\]
  \end{proposition}

  We are now ready to finish the proof of~\Cref{thm-stabilite-nonlineaire-Hs1}. Combining~\Cref{prop-decrease-F-nonlinear} with~\Cref{prop-short-time-nonlinear} and using the equivalence between norms, it is sufficient to prove the decay estimates with an initial condition sufficiently small in~$H^{m+1,1}$ with~$m>\frac{d}2$, and apply them for~$t\geqslant1$, since we have~$s+1>\frac{d}2$. Therefore we start from~\Cref{prop-decrease-F-nonlinear}. In the case of a flat torus of dimension~$d$, as explained in the beginning of the proof of~\Cref{prop-decay-linear-H1}, the Poincaré inequality allows to obtain the existence of a constant~$\lambda>0$ such that for any initial condition~$g^0$ with zero average (a property that is conserved for all time), and sufficiently small in the sense of~\Cref{prop-decrease-F-nonlinear}, we have
  \begin{equation*}%\label{eq-exp-decay-nonlinear-Fg}
    \overline{\mathcal{F}}(g(t,\cdot))\leqslant e^{-\lambda t}\overline{\mathcal{F}}(g^0).
  \end{equation*}

  In the case of the whole space~$\mathbb{R}^d$, we proceed similarly as in~\Cref{prop-decay-linear-H1-M}, by assuming that~$\mathcal{M}(g^0)$ is finite. Indeed, by \Cref{prop-M-short-time}, such a property is conserved up to time~$t=1$. We use~\eqref{estimate-ddtM-F-MJ2}, and~\Cref{lemmeAU}--\ref{U2} together with~\eqref{AfHm-barFD} to bound~$\|\mathbb{J}[g]\|^2_{H^m(\mathbb{R}^d)}$ by~$\overline{\mathcal{D}}(g)$, and therefore we obtain

  \begin{equation}\label{dtMg-F-DM}
    \frac{\d}{\d t}\mathcal{M}(g)\leqslant C\big(\overline{\mathcal{F}}(g)+\overline{\mathcal{D}}(g)\mathcal{M}(g)\big),
  \end{equation}
  where the constant~$C$ only depends on~$m$,~$d$ and~$\rho$.
  Compared to the case of the kinetic Brownian motion in \Cref{subsec-long-time-KBM}, this differential inequality for~$\mathcal{M}$ seems to give a little less control than the corresponding inequality given in~\eqref{dtM-F}, but we can actually take advantage of the fact that~$\overline{\mathcal{D}}(g)$ is integrable in time, and the procedure is similar.

  We still use Heisenberg’s uncertainty principle to obtain~$\mathcal{F}_{0,0}(g)\leqslant C\sqrt{\mathcal{M}(g)\overline{\mathcal{D}}(g)}$, as in the proof of~\Cref{prop-decay-linear-H1-M}, and therefore, as in~\eqref{dtF-F2}, to obtain a constant~$c$ depending only on~$m$,~$d$,~$\rho$ such that under the assumptions of~\eqref{prop-decrease-F-nonlinear}, we have
    
    \begin{equation*}
      \dt \overline{\mathcal{F}}(g) \leq - c \, \frac{\overline{\mathcal{F}}(g)^2}{\overline{\mathcal{F}}(g)+\mathcal{M}(g)}. %\label{dtbarF-F2}
    \end{equation*}
    
    The differential inequality given by~\eqref{dtMg-F-DM} can be solved explicitly, we obtain
    \begin{align*}
     \mathcal{M}(g)&\leqslant\mathcal{M}(g^0)\exp\Big(C\int_0^t\overline{\mathcal{D}}(g(s,\cdot))\d s\Big)+C\int_0^t\overline{\mathcal{F}}(g(s,\cdot))\exp\Big(C\int_s^t\overline{\mathcal{D}}(g(u,\cdot))\d u\Big)\d s\\
                    &\leqslant\mathcal{M}(g^0)e^{\overline{\mathcal{F}}(g^0)-\overline{\mathcal{F}}(g(t,\cdot))}+C\int_0^t\overline{\mathcal{F}}(g(s,\cdot))e^{\overline{\mathcal{F}}(g^0)-\overline{\mathcal{F}}(g(s,\cdot))}\d s\\
      &\leqslant e^{\eta}\Big(\mathcal{M}(g^0)+C\int_0^t\overline{\mathcal{F}}(g(s,\cdot))\d s\Big).                    
    \end{align*}
    Therefore we proceed as in~\Cref{subsec-long-time-KBM}, by letting~$\mathcal{U}(t)=\mathcal{M}(g^0)+\overline{\mathcal{F}}(g^0)+C\int_0^t\overline{\mathcal{F}}(g(s,\cdot))$. Since~$\overline{\mathcal{F}}(g)$ is nonincreasing, we get
    \[\mathcal{M}(g)+\overline{\mathcal{F}}(g)\leqslant\overline{\mathcal{F}}(g^0)+e^{\eta}\Big(\mathcal{M}(g^0)+C\int_0^t\overline{\mathcal{F}}(g(s,\cdot))\d s\Big)\leqslant e^{\eta}\,\mathcal{U}(t),\]
    and therefore we obtain
    \[\mathcal{U}''(t)=C\overline{\mathcal{F}}(g)\leqslant-\frac{c}{Ce^\eta}\frac{\mathcal{U}'(t)^2}{\mathcal{U}}.\]
    We solve this differential inequality exactly as in the proof of~\Cref{prop-decay-linear-H1-M}, with~$p=\frac{c}{Ce^\eta}$, and we obtain
    \[\overline{\mathcal{F}}(g(t,\cdot))\leqslant\widetilde{C}\overline{\mathcal{F}}(g^0)\Big(1+\frac{\overline{\mathcal{F}}(g^0)}{\mathcal{U}(0)}t\Big)^{-\frac{p}{p+1}}.\]
Together with the estimates for the short time behaviour given in~\Cref{prop-short-time-nonlinear,prop-M-short-time}, and the equivalence of~$\overline{\mathcal{F}}$ with the~$H^{m+1,1}$ squared norm (thus controlling the $H^{s,0}$ squared norm), this ends the proof of~\Cref{thm-stabilite-nonlineaire-Hs1}.

\appendix
\section{Controlling \texorpdfstring{$H^1$}{H¹} seminorms of moments by \texorpdfstring{$\|\op{S}g\|^2$}{‖Sg‖²}.}
\label{section-appendix}
For a given~$g\in L^2(\mathbb{R}^d\times\S)$, we denote~$\op{\Pi}_\ell g$ the projection of~$g$ on the space of spherical harmonics of degree~$\ell$ (in its velocity variable~$v$). For instance we have~$\op{\Pi}_0g=\int_{\S}g(\cdot,v)\d v$ and~$\op{\Pi}_1g=d \, v\cdot\mathbb{J}[g]$.
\begin{lemma}\label{lemma-Sg2}
  Let~$\ell\geqslant0$ and~$g\in H^{1,0}(\mathbb{R}^d\times\S)$ (the result is also true for a flat torus of dimension~$d$ instead of~$\mathbb{R}^d$). Recall that~$\|\op{S}g\|_2^2=\|P_{v^\perp}\nabla_xg\|_2^2$ and~$\|\op{S}g\|_2^2=\|P_{v^\perp}\nabla_xg\|_2^2$. We have 
  \[\|\nabla_x\op{\Pi}_\ell g\|_2^2\lesssim\begin{cases}
      \|\op{S}g\|_2^2+\|\op{S}g\|_2\|\op{T}g\|_2 & \text{if }d=2,\\
      \|\op{S}g\|_2^2+\|\op{S}g\|_2^{2-\mu}\|\op{T}g\|_2^{\mu} & \text{if }d=3 \quad (\text{for all }\mu\in(0,1]),\\
      \|\op{S}g\|_2^2 & \text{if }d\geqslant4,\\
      \end{cases}
    \]
    where the symbol~$\lesssim$ denotes an inequality up to a multiplicative constant only depending on~$d$,~$\mu$ and~$\ell$.
  \end{lemma}

  \begin{proof}[\textbf{Proof of~\Cref{lemma-Sg2}.}]
      We work in Fourier variable in space, writing $g(x,v)=\int_{\mathbb{R}^d} \hat{g}(\xi,v)e^{ix\cdot\xi}\d \xi$, and proceed mode by mode. Therefore we have $\int_{\mathbb{R}^d\times\S}|P_{v^\perp}\nabla_xg|^2=\int_{\mathbb{R}^d}|\xi|^2\int_{\S}(1-(e\cdot v)^2)|\hat{g}(\xi,v)|^2\d v \d \xi$, where we have denoted~$e=\frac{\xi}{|\xi|}$, for~$\xi\neq0$.
  We therefore fix~$e\in\S$ and want to study integrals of the form
  \[\mathcal{I}_{e}(h)=\int_{\S}(1-(e\cdot v)^2)h(v)^2\d v.\]
  We therefore have
  \[\|\op{S}g\|_2^2=\int_{\mathbb{R}^d}|\xi|^2\mathcal{I}_{e}(\hat{g}(\xi,\cdot))\d \xi,\  \|\op{S}g\|_2^2+\|\op{T}g\|_2^2=\int_{\mathbb{R}^d}|\xi|^2\|\hat{g}(\xi,\cdot)\|^2\d \xi, \  \|\nabla_x\op{\Pi}_\ell g\|_2^2=\int_{\mathbb{R}^d}|\xi|^2\|\op{\Pi}_\ell\hat{g}(\xi,\cdot)\|^2\d \xi.\]
  In virtue of Hölder inequality, \Cref{lemma-Sg2} will be proven if we manage to prove that for any~$h\in L^2(\S)$, 
  \begin{equation}
    \label{goal-on-I}
    \|\op{\Pi}_\ell h\|_2^2\lesssim\begin{cases}
      \sqrt{\mathcal{I}_{e}(h)}\|h\|_2 & \text{if }d=2,\\
      \big(\mathcal{I}_{e}(h)\big)^{1-\frac{\mu}2}\|h\|_2^{\mu} & \text{if }d=3 \quad (\text{for all }\mu\in(0,1]),\\
      \mathcal{I}_{e}(h) & \text{if }d\geqslant4.
      \end{cases}
    \end{equation}
    
We decompose~$h$ in spherical harmonics (see for instance~\cite[Appendix A.1]{frouvelle2012dynamics}):
\[h(v)=\sum_{k,\ell,m}c_{\ell,m,k}(h)Y_{\ell,m,k}(v),\]
where~$(Y_{\ell,m,k})_{0\leqslant m\leqslant\ell,1\leqslant k\leqslant n_m=\binom{d+m-2}{d-2}-\binom{d+m-4}{d-2}}$ is an orthonormal basis of spherical harmonics
(of degree~$\ell$, implying that~$\op{\Delta}_vY_{\ell,m,k}(v)=-\ell(\ell+d-2)Y_{\ell,m,k}(v)$) for which we precise next the construction. Our goal is then to show that for any~$\ell,m,k$, we have (with~$\mu=1$ for~$d=2$ and~$\mu=0$ for~$d=4$):
\begin{equation}\label{goal-clmk}
  c_{\ell,m,k}(h)^2\lesssim\big(\mathcal{I}_e(h)\big)^{1-\frac{\mu}2}\|h\|_2^{\mu}.
\end{equation}

We write~$v=\cos \theta\,e + \sin \theta\, w$ with~$\theta\in[0,\pi]$ and~$w$ seen as a vector of~$\S^{d-2}$ (the unit vectors orthogonal to~$e$). We suppose we are given~$(Z_{m,k})_{1\leqslant k\leqslant n_m}$ an orthogonal basis of the spherical harmonics of degree~$m$ on the sphere $\S^{d-2}$. We then have~$Y_{\ell,m,k}(v)=Q_{\ell,m}(\cos \theta) \sin^{m}\theta Z_{m,k}(w)$ where~$Q_{\ell,m}$ is a polynomial on degree~$\ell-m$, obtained from the so called Gegenbauer polynomials, but we will not need their specific expression here. Indeed, we write~$h_{m,k}(\theta)=\sum_{\ell\geqslant m}c_{\ell,m,k}Q_{\ell,m}(\cos \theta)$ and therefore we have
\[h(v)=\sum_{m,k}h_{m,k}(\theta)\sin^{m}\theta Z_{m,k}(w).\]
To simplify notations, we suppose that the normalization of the basis~$(Z_{m,k})$ is done in such a way that
\[\int_{\S}\sin^{2m}\theta Z^2_{m,k}(w)\d v=\int_{0}^\pi\sin^{2m+d-2}\theta\, \d \theta,\] and therefore we have
\begin{equation}
  \label{eq-h2-Ieh}\|h\|_2^2=\sum_{m,k}\int_{0}^\pi h_{m,k}(\theta)^2\sin^{2m+d-2}\theta\, \d \theta,\quad \mathcal{I}_e(h)=\sum_{m,k}\int_{0}^\pi h_{m,k}(\theta)^2\sin^{2m+d}\theta\, \d \theta.
\end{equation}
We now have, using Cauchy--Schwarz inequality,
\[ \begin{split}
    c_{\ell,m,k}(h)^2=\Big(\int_{\S}h(v)Y_{\ell,m,k}(v)\d v\Big)&=\Big(\int_0^\pi h_{m,k}(\theta)Q_{\ell,m}(\cos \theta)\sin^{2m+d-2}\theta\, \d \theta\Big)\\
    &\lesssim\int_0^\pi h_{m,k}(\theta)^2\sin^{4m+2d-4}\theta\, \d \theta,
  \end{split}
\]
and therefore when~$4m+2d-4\geqslant2m+d$ (that is to say~$2m+d\geqslant4$), we obtain~$c_{\ell,m,k}(h)^2\lesssim\mathcal{I}_e(h)$. So the only remaining cases to treat are~$m=0$ for~$d\in\{2,3\}$ (for which the only possibility is~$k=1$).

For the case~$d=3$ and~$m=0$, we have, since~$\sin^{-1+\mu}{\theta}$ is integrable on~$[0,\pi]$, 
\[ \begin{split}
    c_{\ell,0,1}(h)^2&=\Big(\int_0^\pi h_{0,1}(\theta)Q_{\ell,0}(\cos \theta)\sin^{\frac{3-\mu}2} \theta\sin^{\frac{-1+\mu}2} \theta\, \d \theta\Big)\\
    &\lesssim\int_0^\pi h_{0,1}(\theta)^2\sin^{3-\mu}\theta\, \d \theta\leqslant\Big(\int_0^\pi h_{0,1}(\theta)^2\sin^3\theta\, \d \theta\Big)^{1-\frac{\mu}2}\Big(\int_0^\pi h_{0,1}(\theta)^2\sin \theta\, \d \theta\Big)^{\frac{\mu}2},
  \end{split}
\]
thanks to Holder’s inequality, which gives~\eqref{goal-clmk} for~$m=0$ and~$d=3$.

For the case~$d=2$ (and~$m=0$), the same strategy would lead to~\eqref{goal-clmk} with any~$\mu>1$, so to refine it a little bit and get the case~$\mu=1$, we explicitly write the decomposition. Indeed we have~$Q_{0,0}=\frac{1}{\sqrt{\pi}}$ and~$Q_{\ell,0}(\cos \theta)=\sqrt{\frac{2}{\pi}}\cos \ell\theta$ (the Chebyshev polynomials), and therefore we obtain~$h_{0,1}$ under the form of its Fourier series, writing~$a_0=\frac{2}{\sqrt{\pi}} c_{0,0,1}$ and~$a_\ell=\sqrt{\frac{2}{\pi}}c_{\ell,0,1}$ for~$\ell\geqslant1$:
\[h_{0,1}(\theta)=\frac{a_0}2+\sum_{\ell\geqslant1}a_\ell\cos(\ell\theta)\]
giving for instance~$\int_0^\pi h_{0,1}(\theta)^2\sin^2\theta\,\d \theta=\frac{\pi}2\sum_{\ell\geqslant0}a_\ell^2$. To compute~$\int_0^\pi h_{0,1}(\theta)^2\sin^2\theta\,\d \theta$, we know the Fourier decomposition of~$\theta\mapsto\sin \theta h(\theta)$:
\[
  \sin(\theta)h(\theta)=\frac{a_0}2\sin\theta+\frac12\sum_{\ell\geqslant1}a_\ell(\sin((\ell+1)\theta)-\sin(\ell-1)\theta)=\frac12\sum_{\ell\geqslant1}(a_{\ell-1}-a_{\ell+1})\sin(\ell\theta)
\]
We therefore obtain
\begin{equation}
  \label{I-Fourier-2d} \int_0^\pi h_{0,1}(\theta)^2\sin^2\theta\,\d \theta=\frac{\pi}8\sum_{\ell\geqslant0}(a_\ell-a_{\ell+2})^2.
\end{equation}
Now we have, for any~$\ell_0\geqslant0$,
\begin{equation*}%\label{trick-al}
  \frac{\pi}8a_{\ell_0}^2+\frac{\pi}8a_{\ell_0+1}^2=\frac{\pi}8\sum_{\ell\geqslant\ell_0}a_{\ell}^2-a_{\ell+2}^2=\frac{\pi}8\sum_{\ell\geqslant\ell_0}(a_{\ell}-a_{\ell+2})^2-2a_{\ell+2}(a_{\ell}-a_{\ell+2}),
\end{equation*}
and thanks to Cauchy--Schwarz inequality and~\eqref{I-Fourier-2d}, we get
\[\begin{split}\frac{\pi}8a_{\ell_0}^2+\frac{\pi}8a_{\ell_0+1}^2&\leqslant\int_0^\pi h_{0,1}(\theta)^2\sin^2\theta\,\d \theta+\sqrt{\int_0^\pi h_{0,1}(\theta)^2\sin^2\theta\,\d \theta}\sqrt{\frac{\pi}2\sum_{\ell\geqslant\ell_0+2}a_\ell^2}\\
  &\leqslant2\sqrt{\int_0^\pi h_{0,1}(\theta)^2\sin^2\theta\,\d \theta}\sqrt{\int_0^\pi h_{0,1}(\theta)^2\,\d \theta},
\end{split}
\]
which finally proves the case~$d=2$ of~\eqref{goal-on-I}, thanks to~\eqref{eq-h2-Ieh}.
\end{proof}

\bibliographystyle{plain}
\bibliography{biblio-clean.bib}

\end{document}